\documentclass[10pt]{amsart}
\usepackage[UKenglish]{isodate}
\usepackage{amsmath}
\usepackage{amsthm}
\usepackage{amsbsy}
\usepackage{amssymb}
\usepackage{amsfonts}
\usepackage[dvips]{lscape}
\usepackage{xcolor}
\usepackage{amscd}
\usepackage[all,cmtip]{xy}
\usepackage{euscript}
\usepackage{parskip}
\usepackage{enumerate}
\usepackage{fullpage}
\usepackage{yfonts}
\usepackage{xcolor}
\usepackage{graphicx}
\usepackage{fancyhdr}
\usepackage{tikz-cd}

\usepackage[margin=1in]{geometry}
\usepackage[expansion=false]{microtype}
\usepackage[pdfusetitle,unicode,hidelinks]{hyperref}

\theoremstyle{plain}

\theoremstyle{plain}
\newtheorem{theorem}{Theorem}[section]

\newtheorem{proposition}[theorem]{Proposition}
\newtheorem{lemma}[theorem]{Lemma}
\newtheorem{corollary}[theorem]{Corollary}

\newtheorem{conjecture}[theorem]{Conjecture}

	\renewcommand{\vec}[1]{\mathbf{#1}}
\theoremstyle{remark}
\newtheorem{remark}[equation]{Remark}
\newtheorem{example}[theorem]{Example}

\theoremstyle{definition}
\newtheorem{definition}[theorem]{Definition}

\let\lim=\relax

\DeclareMathOperator*{\lim}{lim}
\newcommand{\ord}{\textup{ord}}

\newcommand{\diag}{\textup{diag}}

\newcommand{\tr}{\textup{tr}}

\renewcommand{\hat}{\widehat}
\newcommand{\Ki}{K_\infty}
\newcommand{\ki}{K_\infty}
\newcommand{\hQ}{\widehat{Q}}

\newcommand{\hR}{\widehat{R}}

\newcommand{\hl}{\widehat l}
\newcommand{\Kl}{\textup{Kl}}
\newcommand{\Sa}{\textup{Sa}}

\newcommand{\GL}{\textup{GL}}

\newcommand{\FF}{\mathbb{F}}
\newcommand{\ZZ}{\mathbb{Z}}

\newcommand{\NNz}{\mathbb{N}_{\geq 0}}

\newcommand{\CC}{\mathbb{C}}

\newcommand{\TT}{\mathbb{T}}
\newcommand{\PGL}{\textup{PGL}}
\newcommand{\PSL}{\textup{PSL}}

\newcommand{\sumstar}{\sideset{}{^*}\sum}

\newcommand{\cA}{\mathcal{A}}

\newcommand{\cO}{\mathcal{O}}

\newcommand{\cG}{\mathcal{G}}
\newcommand{\vol}{\textup{vol}}
\newcommand{\exc}{\textup{exc}}

\renewcommand{\contentsname}{}
\begin{document}
\title{Optimal Strong Approximation for quadrics over $\mathbb{F}_q[t]$}
\author{Naser Talebizadeh  Sardari and Masoud Zargar}

\address{Penn State department of Mathematics, McAllister Building, Pollock Rd, State College, PA 16802 USA}
\email{nzt5208@psu.edu}

\address{Department of Mathematics, University of Southern California, 3620 S. Vermont
Ave., Los Angeles, CA 90089-2532, U.S.A.}
\email{mzargar1225@gmail.com}

\renewcommand{\contentsname}{}

\begin{abstract}
Suppose $q$ is a fixed odd prime power, $F(\vec{x})$ is a non-degenerate quadratic form over $\mathbb{F}_q[t]$ of discriminant $\Delta$ in $d\geq 5$ variables $\vec{x}$, and $f,g\in\mathbb{F}_q[t]$, $\boldsymbol{\lambda}\in\mathbb{F}_q[t]^d$. We show that whenever $\deg f\geq (4+\varepsilon)\deg g+O_{\varepsilon,F}(1)$, $\gcd(\Delta^{\infty},fg)=O(1)$, and the necessary local conditions are satisfied, we have a solution $\vec{x}\in\mathbb{F}_q[t]^d$ to $F(\vec{x})=f$ such that $\vec{x}\equiv\boldsymbol{\lambda}\bmod g$. For $d=4$, we show that the same conclusion holds if we instead have $\deg f\geq (6+\varepsilon)\deg g+O_{\varepsilon,F}(1)$. This gives us a new proof (independent of the Ramanujan conjecture over function fields proved by Drinfeld) that the diameter of any $k$-regular Morgenstern Ramanujan graphs $G$ is at most $(2+\varepsilon)\log_{k-1}|G|+O_{\varepsilon}(1)$.  In contrast to the $d=4$ case, our result is optimal for $d\geq 5$. Our main new contributions are a stationary phase theorem over function fields for bounding oscillatory integrals, and a notion of anisotropic cones to circumvent isotropic phenomena in the function field setting.\footnotetext{\date\today}
\end{abstract}
\maketitle
\setcounter{tocdepth}{1}
\tableofcontents
\section{Introduction}\label{intro}
Let $\mathbb{F}_q[t]$ be the polynomial ring over the finite field $\mathbb{F}_q$ with $q$ elements, where $q$ is a fixed odd prime power. Suppose $F(\vec{x})$ is a non-degenerate quadratic form over $\mathbb{F}_q[t]$ in $d\geq 4$ variables $\vec{x}$, and $f\in\mathbb{F}_q[t]$. In this paper, we study the optimal strong approximation problem for the quadric $X_f$ given by the equation $F(\vec{x})=f$.
Precisely, given $g\in\mathbb{F}_q[t]$ and polynomials $\lambda_1,\hdots,\lambda_d\in\mathbb{F}_q[t]$, we study integral solutions $\vec{x}:=(x_1,\dots,x_d)\in \mathbb{F}_q[t]^d $ to the system
\begin{equation}\label{maineq}\begin{cases}F(\vec{x})=f,\\ \vec{x}\equiv\boldsymbol{\lambda}\bmod g,\end{cases}\end{equation}
where $\boldsymbol{\lambda}=(\lambda_1,\dots,\lambda_d)$ and $\vec{x}\equiv\boldsymbol{\lambda}\bmod g$ means $x_i\equiv \lambda_i\bmod g$ for every $1\leq i\leq d$.\\
\\
Throughout this paper, we let $\mathcal{O}:=\mathbb{F}_q[t]$, and $K:=\mathbb{F}_q(t)$. For $g\in\mathcal{O}$ and an irreducible $\varpi \in \mathcal{O}$, we write  $\ord_{\varpi}(g)$ for the highest power of $\varpi$ dividing $g$. Let $\mathcal{O}_{\varpi}$ be the completion of $\mathcal{O}$ with respect to the valuation $\ord_{\varpi}$. Furthermore,
\[\Ki=\mathbb{F}_q(\!(1/t)\!):=\left\{\sum_{i\leq N}a_it^i: \mbox{for $a_i\in \FF_q$ and some $N\in\ZZ$} \right\}
\] 
is the completion of $K$ (at $\infty$) with respect to the norm
\[|a/b|:=q^{\deg a-\deg b},\]
where for any $x\in\Ki$, we write $\deg x$ for the highest power of $t$ appearing in  the series expansion of $x$. In particular, $\deg(0)=-\infty.$ We equip $K_{\infty}^d$ with the norm $|\vec{x}|:=\max_i|x_i|$ and write $\deg\vec{x} :=\max_i \deg x_i$ for any $\vec{x}=(x_1,\hdots,x_d).$\\
\\
Given an $\mathcal{O}$-algebra $R$, we let $X_f(R):=\{\vec{x}\in R^d: F(\vec{x})=f\}$. We say all local conditions for the system~\eqref{maineq} are satisfied if $X_f(\Ki)\neq\emptyset,$ and for all $\varpi$, $F(\vec{x})=f$ has a
solution $\vec{x}_{\varpi} \in  \mathcal{O}_{\varpi}^d$ such that $\vec{x}_{\varpi}\equiv\boldsymbol{\lambda}\bmod \varpi^{\ord_{\varpi}(g)}.$ Note that this is a necessary condition for the existence of integral solutions to the system~\eqref{maineq}. We prove the following strong approximation result, a consequence of the more general Theorem~\ref{mainthm} discussed later.
\begin{theorem}[Strong approximation]\label{strong}
Suppose $q$ is a fixed odd prime power, $\varepsilon>0$, and  $F(\vec{x})$ is a fixed non-degenerate quadratic form over $\mathcal{O}$ in $d\geq 4$ variables with discriminant $\Delta$. Let $f, g\in\mathcal{O},$ where $\gcd(\Delta^{\infty},fg)=O(1),$ and  $\boldsymbol{\lambda}\in\mathcal{O}^d.$  Suppose that  all local conditions to the system~\eqref{maineq} are satisfied.  Then there is a constant $C_{\varepsilon,F}$ independent of $f$,  $g$ and $\boldsymbol{\lambda}$ such that the following hold:
\\
\begin{enumerate}[(i)]
\item if $d\geq 5$  and $\deg f \geq (4+\varepsilon)\deg g +C_{\varepsilon,F}$, then  there is a solution $\vec{x} \in \mathcal{O}^d $ to \eqref{maineq};
\\
\item if $d=4$ and $\deg f \geq (6+\varepsilon)\deg g +C_{\varepsilon,F}$, then  there is a solution $\vec{x} \in \mathcal{O}^4 $ to \eqref{maineq}.
\end{enumerate}
\end{theorem}
Throughout this paper, the notation $\gcd(P^{\infty},Q)=O(1)$ means that the irreducible divisors of $P$ appear with bounded multiplcity in $Q$, independently of $Q$.
\begin{remark}
 As we will see in Lemma~\ref{cone}, the local condition at $\infty$, that is $X_f(\Ki)\neq\emptyset$, is always satisfied for non-degenerate quadratic forms over $\Ki$ in $d\geq 4$ variables.
\end{remark}

\begin{conjecture}\label{conjectured4} 
Suppose all the conditions of Theorem~\ref{strong} are satisfied and $d=4$. If $\deg f \geq (4+\varepsilon)\deg g +O_{\varepsilon,F}(1)$, then  there is a solution $\vec{x} \in \mathcal{O}^4 $ to \eqref{maineq}. 
\end{conjecture}
 Another motivation for Theorem~\ref{strong} is related to bounding the diameters of Morgenstern Ramanujan graphs. We begin by defining Ramanujan graphs. Fix an integer $k\geq 3$, and let $G$ be a $k$-regular connected graph. Let $V_G$ be the set of vertices of $G$, and define the adjacency matrix of $G$  by
  \[
 A_G:=[a_{i,j}]_{ i,j\in V_G},
  \] where $a_{i,j}$ is the number of edges between $i$ and $j.$ Since $G$ is $k$-regular and connected, 
  $k$ is a simple eigenvalue of $A_G,$ and $-k$ is also a simple eigenvalue if $G$ is bipartite. All the the other eigenvalues are inside the open interval $(-k,k).$   Let $\lambda_G<k$ be the maximum of the absolute value of eigenvalues inside $(-k,k)$. By the Alon--Boppana Theorem~\cite{Lubotzky1988}, $\lambda_G\geq 2\sqrt{k-1}-o(1)$, where $o(1)$ goes to zero as $|G|\to \infty$.  We say that $G$ is a Ramanujan graph if  $\lambda_G \leq 2\sqrt{k-1}.$\\
\\
The first explicit construction of Ramanujan graphs is due to Lubotzky\textendash Phillips\textendash Sarnak~\cite{Lubotzky1988}, and independently by Margulis \cite{Margulis}. It is a Cayley graph of $\PGL_2(\mathbb{Z}/q\mathbb{Z})$ or $\PSL_2(\mathbb{Z}/q\mathbb{Z})$ with $p+1$ explicit generators for every prime $p$ and integer $q$.   The optimal spectral gap on the LPS construction is a consequence of the Ramanujan bound on the Fourier coefficients of the weight 2 holomorphic modular forms, which justifies their naming. 
We refer the reader to \cite[Chapter 3]{Peter}, where a complete history of the construction of Ramanujan graphs and other extremal properties of them are recorded. In particular, Lubotzky\textendash Phillips\textendash Sarnak proved that the diameter of every $k$-regular Ramanujan graph $G$ is bounded by $2\log_{k-1}|G|+O(1).$ This is still the best known upper bound on the diameter of a Ramanujan graph. It was conjectured that the diameter is bounded by $(1+\varepsilon)\log_{k-1}|G|$ as $|G|\to \infty;$ see \cite[Chapter 3]{Peter}. However, the first author proved that for some infinite families of LPS Ramanujan graphs the diameter is bigger than $4/3 \log_{k-1}|G|+O(1);$ see \cite{Sardari2018}. The first author has conjectured that the diameter of the LPS Ramanujan graphs is asymptotically $4/3\log_{k-1}|G|+o(\log_{k-1}|G|)$; the upper bound follows from an optimal strong approximation conjecture for integral quadratic forms in 4 variables; see \cite[Conjecture 1.3]{Naser2}. 
\\

Morgenstern generalized the LPS construction to prime power degrees~\cite{Morgenstern}. 
Theorem~\ref{strong} above can be used to give a new proof, independent of the Ramanujan conjecture over function fields~\cite{Drinfeld}, that the diameter of $k$-regular Morgenstern Ramanujan graphs $G$ are bounded above by $(2+\varepsilon)\log_{k-1}|G|+O_{\varepsilon}(1)$. We recall the construction of Ramanujan graphs due to Morgenstern when $q$ is odd. The quaternion algebra  for even $q$ can be found in Section 5 of Morgenstern's paper \cite{Morgenstern}. Consider the quaternion algebra
$$\mathcal{A}:=K\mathbf{1}+K\mathbf{i}+K\mathbf{j}+K\mathbf{ij},\ \mathbf{i}^2=\nu,\ \mathbf{j}^2=t-1,\ \mathbf{ij}=-\mathbf{ji},$$
where $\nu\in \mathbb{F}_q^*$ is not a square in $\mathbb{F}_q^*.$ Let
$$\mathcal{S}:=\mathcal{O}\mathbf{1}+\mathcal{O}\mathbf{i}+\mathcal{O}\mathbf{j}+\mathcal{O}\mathbf{ij}$$
be the integral part of $\mathcal{A}$. Given $\xi=a+b\mathbf{i}+c\mathbf{j}+d\mathbf{ij}$ in $\mathcal{A}$, its conjugate is defined as $\overline{\xi}:=a-b\mathbf{i}-c\mathbf{j}-d\mathbf{ij}$. Furthermore, we have the norm
$$N(\xi):=\xi\overline{\xi}=a^2-b^2\nu+(d^2\nu-c^2)(t-1).$$ 
Morgenstern's quadratic form over $\mathcal{O}$ is
\begin{equation}\label{Morgquad}
F_M(a,b,c,d):=a^2-b^2\nu+(d^2\nu-c^2)(t-1).
\end{equation}
Note that the quadratic equation $x^2-\nu y^2=1$ has exactly $q+1$ solution over $\mathbb{F}_q$~\cite[Lemma 4.2]{Morgenstern}.
Following \cite[Definition 4.3]{Morgenstern}, for every solution  $c_i^2-\nu d_i^2=1,$ where $1\leq i\leq q+1$, define  \textit{basic norm} $t$ element
\[
\xi_i:=1+c_i\mathbf{j}+d_i\mathbf{ij}.
\]
We may suppose that $(c_i,d_i)=-(c_{i+\frac{q+1}{2}}, d_{i+\frac{q+1}{2}})$ for $1\leq i\leq \frac{q+1}{2},$ which implies $\xi_i=\bar{\xi}_{i+\frac{q+1}{2}}.$ Let \(B:=\{\xi_i:  1\leq i\leq q+1  \}   \). From Lemmas 4.2 and 4.4 of Morgenstern's \cite{Morgenstern}, every  $x\in \mathcal{S}$ with $N(x)=t^n$ has a unique factorization
\begin{equation}\label{xfactor}x=ut^r\prod_{i=1}^m\theta_i ,\end{equation}
where $2r+m=n$, $N(u)=1$, $\theta_i\in B$, and $t$ does not divide $\prod_{i=1}^m\theta_i $. It follows that any $x=a+b\mathbf{i}+c\mathbf{j}+d\mathbf{ij} \in \mathcal{S}$ with $N(x)=t^n$  has $u=1$ if and only if $a-1,b\equiv 0\bmod{t-1}$.  Define
\[\Lambda(t-1):=\left\{x=a+b\mathbf{i}+c\mathbf{j}+d\mathbf{ij}\in\mathcal{S}:\begin{array}{c}a-1,b\equiv 0\bmod{t-1},\\ N(x)\ \textup{is a power of }t\end{array}\right\}/\sim,\]
where $\sim$ means we identify $x$ with $t^m x$ for every positive integer $m.$
From the above discussion, it follows that $\Lambda(t-1)$ is a free group generated by $\xi_1,\hdots,\xi_{\frac{q+1}{2}}$.  Let $g\in\mathbb{F}_q[t]$ be an irreducible polynomial prime to $t(t-1)$, and define 
\[\Lambda(g):=\left\{x=a+b\mathbf{i}+c\mathbf{j}+d\mathbf{ij}\in\Lambda(t-1):\begin{array}{c}b,c,d\equiv 0\bmod{g}\end{array}\right\}.\]
$\Lambda(g)$ is a normal subgroup of $\Lambda(t-1);$  define the quotient group $\Gamma_g:=\Lambda(t-1)/\Lambda(g).$ 
\\


The \textit{Cayley graph}  $C(\Gamma_g,B)$ of $\Gamma_g$ with respect to $B$  is the graph with the vertex set $\Gamma_g$ and with the edge set $\{\{v_1,v_2\} : v_1,v_2\in \Gamma_g, \text{ and } v_1=bv_2 \text{ for some } b\in B \}.$ Morgenstern proved that  $C(\Gamma_g,B)$ is  a $q+1$-regular Ramanujan graph~\cite[Theorem 4.11]{Morgenstern}.  See Morgenstern's paper \cite{Morgenstern} for a detailed discussion.\\

\begin{corollary}\label{graphoptimal} Suppose that $q$ is a fixed odd prime power, $g\in\mathcal{O}$ is an irreducible polynomial prime to $t(t-1)$, and the $q+1$-regular $G:=C(\Gamma_g,B)$ is as above. For any $\varepsilon >0,$
the diameter of $G$ is at most
\[(2+\varepsilon)\log_{q}|G|+O_{\varepsilon}(1),\]
where  $O_{\varepsilon}(1)$ is a constant independent of $g.$
\end{corollary}
\begin{proof}Since $G$ is a Cayley graph, it suffices to bound the distance of $1\in G$ to any other vertex $v\in G.$  Suppose that $v$ is represented by 
$a+b\mathbf{i}+c\mathbf{j}+d\mathbf{ij}\in\Lambda(t-1).$ This implies that 
\[F_M(a,b,c,d)=t^{\alpha}\]
for some $\alpha\geq 0,$ where $F_M$ is defined in~\eqref{Morgquad}. By~\eqref{xfactor}, this representative gives a path from $1$ to $v$ with length equal to the number of basic norm $1$ elements in the unique factorization. This is at most $\alpha$. 
\\

Suppose  that $\alpha\geq (6+\varepsilon)\deg ((t-1)g)+C_{\varepsilon,F_M}+2,$ where $C_{\varepsilon,F_M}$ is the constant appearing in part (ii) of Theorem~\ref{strong} for $F_M$. We  use Theorem~\ref{strong} to find a new representative of $v$ of  norm $t^{\alpha-2}$. Let  
\[
\boldsymbol{\lambda}\equiv t^{-1}(a,b,c,d) \bmod (t-1)g.
\]
We claim that for every prime ideal  $\varpi$ , $F_M(\vec{x})=t^{\alpha-2}$ has a
solution $\vec{x}_{\varpi} \in  \mathcal{O}_{\varpi}^4$ such that $\vec{x}_{\varpi}\equiv\boldsymbol{\lambda}\bmod \varpi^{\ord_{\varpi}((t-1)g)}.$ If $\varpi\neq (t)$, then $ \vec{x}:=t^{-1}(a,b,c,d)\in  \mathcal{O}_{\varpi}^4$ is a solution. It remains to check the local condition for $\varpi=(t).$ If $\alpha$ is even then $\vec{x}=(t^{\frac{\alpha-2}{2}},0,0,0)$ is a local solution. For odd $\alpha$, $\vec{x}=\xi_1t^{\frac{\alpha-3}{2}}$ is a local solution. Since all local conditions are satisfied, an application of Theorem~\ref{strong} gives that there is an integral solution $\vec{x}\in\mathcal{O}^4$ to $F_M(\vec{x})=t^{\alpha-2}$ such that $\vec{x}\equiv \boldsymbol{\lambda}\bmod (t-1)g$. $\vec{x}$ gives another representative of $v\in G$. By continuing this process, we may reduce $\alpha$ until we obtain that $v$ is of distance at most $(6+\varepsilon)\deg ((t-1)g)+C_{\varepsilon,F_M}$ from $1.$ Since $v$ was arbitrary, this concludes the proof if  $|G|\gg q^{3\deg g}$. We show this in what follows. Define the subgroup
\[\Lambda^+(g):=\left\{x=a+b\mathbf{i}+c\mathbf{j}+d\mathbf{ij}\in\mathcal{S}:\begin{array}{c}b,c,d\equiv 0\bmod{(t-1)g},\\ a\equiv t^k\bmod (t-1)g,\textup{ for some }k\geq 0\\ N(x)\ \textup{is a power of }t\end{array}\right\}/\sim\]
of $\Lambda(g)$ which has index $[\Lambda(g):\Lambda^+(g)]=O(q^2)$. Also define the finite group
\[\overline{\Lambda(t-1)}:=\left\{\overline{x}=\overline{a}+\overline{b}\mathbf{i}+\overline{c}\mathbf{j}+\overline{d}\mathbf{ij}\in\mathcal{S}/((t-1)g\mathcal{S}):\begin{array}{c}\overline{a}-1,\overline{b}\equiv 0\bmod{t-1},\\ N(\overline{x})\ \textup{is a power of }t\end{array}\right\}/\sim.\]
The natural group homomorphism
\[\mu:\Lambda(t-1)\rightarrow\overline{\Lambda(t-1)}\]
given by taking coordinates modulo $(t-1)g$ has $\Lambda^+(g)$ in its kernel. A similar application of Theorem~\ref{strong} implies that $\mu$ is surjective. It is easy to see that the size of the image of $\mu$ satisfies  $|\overline{\Lambda(t-1)}|\gg q^{3\deg g}.$ Hence,
\[
|G|=|\Lambda(t-1)/\Lambda(g)|=\frac{|\Lambda(t-1)/\Lambda^+(g)|}{[\Lambda(g):\Lambda^+(g)]}= \frac{|\overline{\Lambda(t-1)}|}{[\Lambda(g):\Lambda^+(g)]}\gg q^{3\deg g}.
\]

%
\end{proof}
Note that our proof is independent of the Ramanujan conjecture over function fields. Using the Ramanujan conjecture, one obtains the stronger statement that the diameter of $G$ is at most $2\log_q|G|+2$~\cite{Morgenstern}. Conjecture~\ref{conjectured4} implies that $(\frac{4}{3}+\varepsilon)\log_q|G|+O_{\varepsilon}(1)$ is an upper bound on the diameters of Morgenstern's Ramanujan graphs, following the proof of Corollary~\ref{graphoptimal}. In~\cite{SZ2019}, we showed that $\frac{4}{3}\log_q|G|-O(1)$ is a lower bound on the diameters of an infinite family of Morgenstern Ramanujan graphs.
\subsection{Method of proof}
Our method is based on a version of the circle method that is developed in the work of Heath-Brown over the integers~\cite{Brown}, and that was further developed in the first author's paper \cite{Naser2} to prove an optimal strong approximation result for quadratic forms over the integers. Browning and Vishe constructed a version of this circle method for function fields~\cite{Browning}. In this paper, we extend the circle method over function fields by proving a stationary phase theorem that is also of independent interest. The stationary phase theorem allows us to bound certain oscillatory integrals that appear in the circle method. 
\\

 A notion that is very important in this paper is that of an anisotropic cone. 
  Anisotropic cones are defined as follows. We use the notation $\hat{R}:=q^R$ for any given real number $R$.
\begin{definition}[Anisotropic cone]\label{def:aniso}
$\Omega\subset K_{\infty}^d$ is an \textit{anisotropic cone} with respect to the quadratic form $F(\vec{x})$ if there exist fixed positive integers $\omega$  and  $\omega^{\prime}$ such that: 
\begin{enumerate}
\item if $\vec{x}\in \Omega$ then $f\vec{x}\in \Omega$ for every $f\in K_\infty;$
\item \label{prop2} if $\vec{x}\in\Omega$ and $\vec{y} \in  K_\infty^d $ with $|\vec{y}|\leq |\vec{x}|/ \hat{\omega},$ then $\vec{x}+\vec{y} \in \Omega;$ and 
\item \label{prop3}  if $\vec{x}\in\Omega,$  $ \hat{\omega^{\prime}} |F(\vec{x})| \geq |\vec{x}|^2.$
\end{enumerate}
\end{definition}
\begin{example}\label{exampleaniso}
For the quadratic form $F(x_1,\hdots,x_d):=x_1^2+\dots+x_d^2$, $\Omega:=\{\vec{x}\in\Ki^d|\forall i<d,\ \deg x_i<\deg x_d\}$ is an anisotropic cone. We can take $\omega=1$ and $\omega'=0$.
\end{example}
It follows that there are only finitely many points of $X_f(\mathcal{O})$ inside an anisotropic cone. In Lemma~\ref{cone}, we show that given a quadratic form $F$ in $d\geq 4$ variables, there is an anisotropic cone $\Omega$ such that for any $f\in\mathcal{O}$, $\Omega\cap X_f(\Ki)\neq\emptyset$. This implies that the following is a generalization of Theorem~\ref{strong}.
\begin{theorem}
\label{mainthm}Suppose $q$ is a fixed odd prime power, $\varepsilon>0$, and  $F(\vec{x})$ is a fixed non-degenerate quadratic form over $\mathcal{O}$ in $d\geq 4$ variables with discriminant $\Delta$.  Let $f, g\in\mathcal{O},$ where $\gcd(\Delta^{\infty},fg)=O(1),$ and $\boldsymbol{\lambda}\in\mathcal{O}^d.$  Suppose that  all local conditions to the system~\eqref{maineq} are satisfied. Furthermore, let  $\Omega$ be an anisotropic cone such that $\Omega\cap X_f(\Ki)\neq\emptyset$.  Then there is a constant $C_{\varepsilon,F,\Omega}$ independent of $f,g,\boldsymbol{\lambda}$ such that the following hold:
\begin{enumerate}[(i)]
\item \label{mainthmpart1}if $d\geq 5$ and $\deg f \geq (4+\varepsilon)\deg g +C_{\varepsilon,F,\Omega}$, then  there is a solution $\vec{x} \in \Omega\cap\mathcal{O}^d $ to \eqref{maineq};
\\
\item if $d=4$ and $\deg f \geq (6+\varepsilon)\deg g +C_{\varepsilon,F,\Omega}$, then there is a solution $\vec{x} \in \Omega\cap\mathcal{O}^4 $ to \eqref{maineq}.
\end{enumerate}
\end{theorem}

In the following proposition, we prove that the condition $\deg f\geq 4\deg g-O_{F,\Omega}(1)$ in Theorem~\ref{mainthm} is necessary for every quadratic form in $d\geq 4$ variables.
\begin{proposition}\label{optimality} Given $F(\vec{x})$ in $d\geq 4$ variables, $\Omega,$ $g$ satisfying all the assumptions in Theorem~\ref{mainthm}, there exist $f$ and $\boldsymbol{\lambda}$ satisfying all the assumptions in Theorem~\ref{mainthm}, and 
\[
\deg f \geq 4\deg g - O_{F,\Omega}(1).
\]
such that~\eqref{maineq} does not have any solution $\vec{x}\in  \Omega\cap\mathcal{O}^d $. 
\end{proposition}
\begin{proof}
Suppose that $f\equiv F(\boldsymbol{\lambda})+gt^{\deg(g)-1} \bmod g^2.$  Suppose $\vec{x}\in  \Omega\cap\mathcal{O}^d $ is a solution to \eqref{maineq}. Then  $F(\vec{x})=f,$ and   $\vec{x}=g\vec{t}+ \boldsymbol{\lambda}$ for some $\vec{t}\in \mathcal{O}^d$. Write $F(\vec{x})= \vec{x}^{\intercal} A \vec{x}$ for some symmetric matrix $A$ with $\mathcal{O}$ coefficients. We have 
\[
(g\vec{t}+ \boldsymbol{\lambda})^{\intercal}A(g\vec{t}+ \boldsymbol{\lambda}) \equiv F(\boldsymbol{\lambda})+gt^{\deg(g)-1} \bmod g^2
\]
This implies 
\[
2\boldsymbol{\lambda}^{\intercal}A\vec{t}\equiv t^{\deg{g}-1} \bmod g.
\]
It follows that 
\[
\deg{\vec{t}} \geq \deg{g}-1-\deg{\boldsymbol{\lambda}^{\intercal}A}.
\]
Assume that $\boldsymbol{\lambda}$ has constant polynomial coordinates. Then,
\[
\deg\vec{x}\geq 2 \deg{g}- O_{F}(1).
\]
By (3) of definition~\ref{def:aniso}, we have 
\[
\deg{f}=\deg{F(\vec{x})} \geq 2\deg\vec{x}-O_{\Omega}(1)\geq 4 \deg{g}- O_{F,\Omega}(1).
\]
The conclusion follows.
\end{proof}
\begin{remark} 
For $d\geq 5$, part~\eqref{mainthmpart1} of Theorem~\ref{mainthm} combined with Proposition~\ref{optimality} imply that the constant $4$ in the condition $\deg f\geq (4+\varepsilon)\deg g+O_{\varepsilon,F,\Omega}(1)$ is optimal. For $d=4$, Conjecture~\ref{conjectured4} combined with Proposition~\ref{optimality} would imply that $\deg f\geq (4+\varepsilon)\deg g+O_{\varepsilon,F,\Omega}(1)$ is not only necessary, but also sufficient.
\end{remark}

In the following theorem, we state a version of our optimal strong approximation for the place at $\infty$. 
\begin{theorem}\label{infmainthm}
Suppose $q$ is a fixed odd prime power, $\varepsilon>0$, and  $F(\vec{x})=\vec{x}^{\intercal}A\vec{x}$ is a fixed non-degenerate quadratic form over $\mathcal{O}$ in $d\geq 4$ variables. Suppose $A=[a_{ij}(t)]_{ij}$. Let $f\in\mathcal{O}$ be such that $\gcd(f,t\Delta)=1$, and suppose the quadratic residue of $f$ is that of $a_{k\ell}$, an entry of $A$ with $\gcd(a_{k\ell},t)=1$ and of maximal degree among the entries. Suppose $\Omega$ is an anisotropic cone. If $\boldsymbol{\lambda}\in X_f(\Ki)\cap\Omega$ and $N\geq 0$ is an integer, then there is a constant $C_{\varepsilon,F,\Omega}$ independent of $f,A, \boldsymbol{\lambda}$ such that the following hold:
\begin{enumerate}[(i)]
\item \label{mainthmpart1}if $d\geq 5$ and $\deg f \geq (4+\varepsilon)N +C_{\varepsilon,F,\Omega}$, then  there is a solution $\vec{x} \in \mathcal{O}^d $ to $F(\vec{x})=f$ such that $\frac{|\vec{x}-\boldsymbol{\lambda}|}{|f|^{1/2}}< q^{-N}$;
\\
\item if $d=4$ and $\deg f \geq (6+\varepsilon)N+C_{\varepsilon,F,\Omega}$, then  there is a solution $\vec{x} \in \mathcal{O}^4 $ to $F(\vec{x})=f$ such that $\frac{|\vec{x}-\boldsymbol{\lambda}|}{|f|^{1/2}}< q^{-N}$.
\end{enumerate}
\end{theorem}

The lack of optimality for $d=4$ in our method, in contrast to when $d\geq 5$, appears in Proposition~\ref{prop:small}, where the triangle inequality along with the Weil bound are used. When $d=4$, the triangle inequality leads to some loss. We proved in~\cite{SZ2019} that the optimality for Morgenstern's quadratic form $F_M$ of~\eqref{Morgquad} reduces to a \textit{twisted} version of the Linnik\textendash Selberg conjecture over function fields (Conjecture 1.4 of \textit{loc.cit}) which we suspect is true.  Over function fields, the classical Linnik\textendash Selberg conjecture is true and is equivalent to the Ramanujan conjecture proved by Drinfeld~\cite{Drinfeld}. See the work of Cogdell and Piatetski-Shapiro \cite{CPS} for a proof of this. We expect that a generalization of Conjecture 1.4 of~\cite{SZ2019} holds for arbitrary quadratic forms in $d=4$ variables, leading to Conjecture~\ref{conjectured4}. Since the Ramanujan conjecture over $\mathbb{F}_q(t)$ is proved, in contrast to that over $\mathbb{Q}$, there is greater hope of proving such a result over function fields.
\subsection{Comparison with other  results}
Sarnak studied the distribution of integral points on the sphere $S^3$. Indeed, given $R>0$ such that $R^2\in \mathbb{Z}$, we let $C(R)$ denote the maximum volume  of any cap on  the $(d-1)$-dimensional sphere $S^{d-1}(R)$ of radius $R$ which contains no integral points. Sarnak defined  \cite{Sarnak2}
 the covering exponent of integral points on the sphere by:
 \begin{equation*}
 \begin{split}
 K_d&:=\limsup_{R \to \infty}\frac{\log \left(\# S^{d-1}(R)\cap \mathbb{Z}^d \right) }{\log\left( \text{vol } S^{d-1}(R)/C(R)\right)}.
 \end{split}
 \end{equation*}
 
In his letter \cite{Sarnak2} to  Aaronson and  Pollington, Sarnak showed that $4/3\leq K_4 \leq 2$. To show that $K_4\leq 2$, he appealed to the Ramanujan bound on the Fourier coefficients of weight $k$ modular forms, while the lower bound $4/3\leq K_4$ is a consequence of an elementary number theory argument. Furthermore, Sarnak states some open problems \cite[Page $24$]{Sarnak2}.  The first one is to show that  $K_4<2$ or  even that $K_4=4/3$.\\
\\
It follows from Theorem 1.8 and Corollary 1.9 of \cite{Naser2} that $K_d=2-\frac{2}{d-1}$ for $d\geq 5$ and $4/3\leq K_4 \leq 2$; see also \cite{Siegel} for bounds on the average covering exponent.  Browning--Kumaraswamy--Steiner~\cite{BKS} showed that $K_4 = 4/3$, subject to the validity of a twisted version of a conjecture of Linnik about cancellation in sums of Kloosterman sums; see also Remark 6.8 of \cite{Naser2}. 
\\

More generally, Ghosh, Gorodnick and Nevo studied the covering exponent of the orbits of a lattice subgroup $\Gamma$ in a connected Lie group $G$, acting on a suitable homogeneous spaces $G/H$; see \cite{GN,GGN, GGN1,GGN2}. They linked  the covering exponent $K_{\Gamma}$ of $\Gamma$ to the spectrum of $H$ in the automorphic representation on $L^2(\Gamma\backslash G).$ In particular, they showed that $K_{\Gamma} \leq 2$ if  the restriction of the unitary representation on $L^2(\Gamma\backslash G)$ to $H$ has tempered spherical spectrum as a representation of $H$; see~\cite[Theorem 3.5]{GGN2} and \cite[Theorem 3.3]{GGN1}.  This recovers  the above result of Sarnak for $d=4.$  For $d\geq 5$, by using the best bound on  the generalized Ramanujan conjecture for $SO_{d},$ they showed that    $1 \leq K_d\leq 4-4/(d-1) $ for odd $d$ and $1 \leq K_d\leq  4-16/(d+2)$ for even $d$~\cite[Page 12]{GGN}. They raised the question of improving  these bounds in \cite[Page 11]{GGN}. As pointed out above, the first author gave a definite answer to this question and showed that $K_d=2-\frac{2}{d-1}$ for $d\geq 5$ in~\cite{Naser2}. In this paper, we find the optimal covering exponent for the quadratic forms in $d\geq 5$ variables in the function field setting. 

\subsection{Outline of the paper}In Section~\ref{deltamethod}, we set up the circle method suitable for our strong approximation problem. In the circle method, there are two type of quantities that must be bounded: exponential sums and oscillatory integrals. In Section~\ref{boundsexponential}, we prove necessary bounds on exponential sums. In Section~\ref{analyticsection}, we develop and prove a stationary phase theorem over function fields. The stationary phase theorem is then used in Section~\ref{osil} to bound oscillatory integrals by relating them to Kloosterman and Sali\'e sums. In Section~\ref{mainboundsection}, Hensel's Lemma is used to estimate the main contribution in the circle method. Finally, in Section~\ref{finalproof}, the previous estimates are used to prove the main results.

\subsection{Summary of notations}\label{Summaryofnot}
We summarize here our notations and explain them further when they appear. 
\begin{itemize}
\item $\mathbb{F}_q$ is the finite field with $q$ elements, $q$ a fixed odd prime power;
\item $e_q(a)=\exp(2\pi i \tr(a)/p)$, where $\tr: \FF_q\rightarrow \FF_p$ denotes the trace map;
\item $\mathcal{O}:=\mathbb{F}_q[t];$ whenever we consider elements of $\mathcal{O}$ in our formulae, we implicitly assume that they are monic;
\item $K:=\mathbb{F}_q(t);$
\item for prime ideal $\varpi\subset\mathcal{O}$ and $g\in\mathcal{O}$, $\ord_{\varpi}(g)$ is the higher power of $\varpi$ dividing $g$;
\item $\mathcal{O}_{\varpi}$ is the completion of $\mathcal{O}$ with respect to the valuation $\ord_{\varpi}$;
\item $\Ki=\mathbb{F}_q(\!(1/t)\!):=\left\{\sum_{i\leq N}a_it^i: \mbox{$a_i\in \FF_q$ and $N\in\ZZ$} \right\};$
\item for $x\in \Ki$, $\deg x$ is the highest power of $t$ appearing in  the series expansion of $x$, with $\deg(0)=-\infty;$ 
\item for $\vec{x}=(x_1,\hdots,x_d)\in \Ki^d,$ $\deg\vec{x} :=\max_i \deg x_i;$ 
\item for $x\in \Ki$, $|x|:=q^{\deg x};$
\item $K_{\infty}^d$ is equipped  with the norm $|\vec{x}|:=\max_i|x_i|$;
\item $\mathcal{O}_{\infty}:=\{x\in\Ki:|x|\leq 1\}=\left\{\sum_{i\leq 0}a_it^i: \mbox{for $a_i\in \FF_q$} \right\};$
\item $\mathbb{T}:=\{x\in\Ki:|x|<1\}=\left\{\sum_{i\leq -1}a_it^i: \mbox{$a_i\in \FF_q$} \right\};$
\item for $Q\in\mathbb{Z}$, $\hat{Q}=q^Q;$
\item $\textup{M}_d(R)$ is the ring of square matrices of size $d$ with coefficients in $R$;
\item $\GL_d(R)$ is the ring of invertible square matrices of size $d$ with coefficients in $R$;
\item for $A=[a_{ij}]\in\textup{M}_d(\Ki)$, its norm is $|A|:=\max|a_{ij}|$;
\item for $A=[a_{ij}]\in\textup{M}_d(\Ki)$, its degree is $\deg A:=\max \deg a_{ij}$;
\item for $F$ a quadratic form in $d$ variables, $F(\vec{x})=\vec{x}^{\intercal}A\vec{x}$ for some symmetric matrix $A$ of size $d$; 
\item $\Delta=\det(A).$
\item For $f,g\in\mathcal{O}$, $\gcd(f^{\infty},g)=O(1)$ means that the irreducible divisors of $f$ appear with bounded multiplicity in $g$, independently of $g$.
\end{itemize}

\section{The circle method for small target}\label{deltamethod}
In this section, we define a weighted sum $N(w, \boldsymbol{\lambda})$  counting the number of integral solutions to \eqref{maineq}. We then use the circle  method to give an expression for $N(w, \boldsymbol{\lambda})$  in terms of exponential sums and oscillatory integrals. This is done by giving an expansion of the delta function using the decomposition of $\TT$ (that we shall define below) found in the paper \cite{Browning} of Browning and Vishe. 
\\

Consider  
$$
\TT:=\{\alpha\in K_\infty: |\alpha|<1\}=\left\{\sum_{i\leq -1}a_it^i: \mbox{$a_i\in \FF_q$}
\right\},
$$
and let
$$\mathcal{O}_{\infty}:=\{x\in\Ki:|x|\leq 1\}=\left\{\sum_{i\leq 0}a_it^i: \mbox{for $a_i\in \FF_q$} \right\}.$$
$\TT$  is the maximal ideal of the local ring $\mathcal{O}_{\infty},$ and  is a compact subset of $K_\infty.$ $\Ki$ is a locally compact abelian group, and so we equip it with the Haar measure $d\alpha$ normalized so that $\int_{\TT}d\alpha=1$ (as in Kubota~\cite[p.9]{kubota}). The space $S(\Ki^d)$ is the space of Schwarz--Bruhat functions on $\Ki^d$, that is, locally constant functions $f:\Ki^d\rightarrow\mathbb{C}$ of compact support.

\subsection{Characters}\label{s:add-characters}
Recall the notations from Subsection~\ref{Summaryofnot}. Let $p$ be the characteristic of $\mathbb{F}_q$. For $N\in\mathbb{Z}$, we write $\hat{N}:=q^N$. There is a non-trivial additive character $e_q:\FF_q\rightarrow \CC^*$ defined
for each $a\in \FF_q$ by taking 
$e_q(a)=\exp(2\pi i \tr(a)/p)$, where $\tr: \FF_q\rightarrow \FF_p$ denotes the
trace map.
This character induces a non-trivial (unitary) additive character $\psi:
K_\infty\rightarrow \CC^*$ by defining $\psi(\alpha)=e_q(a_{-1})$ for any 
$\alpha=\sum_{i\leq N}a_i t^i$ in $\Ki$. In particular it is clear that
$\psi|_\cO$ is trivial. 
More generally, given any $\gamma \in \Ki$, the map $\alpha\mapsto \psi(\alpha\gamma)$ is an additive
character on $\Ki$. We then have the following orthogonality property.
\begin{lemma}[Kubota, Lemma 7 of
\cite{kubota}]\label{lem:orthogsum}
$$
\sum_{\substack{b\in \cO\\ |b|<\hat N}}\psi(\gamma b)=\begin{cases}
\hat N, & \mbox{if $|(\!(\gamma)\!)|<\hat N^{-1}$,}\\
0, & \mbox{otherwise},
\end{cases}
$$
for any $\gamma \in \ki$ and any integer $N\geq 0$, where $(\!(\gamma)\!)$ is the part of $\gamma$ with all degrees negative.
\end{lemma}
We also have the following
\begin{lemma}[Kubota, Lemma 1(f) of \cite{kubota}]\label{lem:orthog}
Let $Y\in \ZZ$ and $\gamma\in  \ki$. Then 
$$
\int_{|\alpha|<\hat{Y}} \psi(\alpha \gamma) d \alpha=\begin{cases}
\hat Y, &\mbox{if $|\gamma|<\hat Y^{-1}$},\\
0, &\mbox{otherwise.}
\end{cases}
$$
\end{lemma}
In particular, if we set $Y=0$, then we obtain the following expression for the delta function on $\mathcal{O}$:
\[\delta(x)=\int_{\TT}\psi(\alpha x)d\alpha,\]
where 
\[
\delta(x)=\begin{cases} 1 &\text{ if } x=0,\\ 0 &\text{ otherwise.}  \end{cases}
\]
\subsection{The delta function}The idea now is to decompose $\TT$ into a disjoint union of balls (with no minor arcs) which is the analogue of Kloosterman's version of the circle method in this function field setting. This is done via the following lemma of Browning and Vishe \cite[Lemma~4.2]{Browning}.
\begin{lemma}\label{lem:dissection}
For any $Q>1$ we have a disjoint union 
\[
\TT=\bigsqcup_{\substack{r\in \cO\\ |r|\leq \hat Q\\
\text{$r$ monic}}}
\bigsqcup_{\substack{a\in \cO\\ |a|<|r|\\ (a,r)=1} }
\left\{\alpha\in \TT: |r\alpha-a|<\hat Q^{-1}\right\}.\]
\end{lemma}
The following follows from Lemma \ref{lem:dissection}.
\begin{lemma}\label{delta} Let $Q\geq  1$ and $n\in \cO.$ 
We have 
\begin{equation}\label{deltaeq}
\delta(n)= \frac{1}{\hQ^{2}}
\sum_{\substack{
r\in \cO\\
|r|\leq \hat Q\\
\text{$r$ monic}
} }
\sumstar_{\substack{
|a|<|r|} }
\psi\left(\frac{an}{r}\right) h\left(\frac{r}{t^Q},\frac{n}{t^{2Q}}\right)
\end{equation}
where we henceforth put
\[
\sumstar_{\substack{
|a|<|r| }}
:=\sum_{\substack{
a\in \cO\\ 
|a|<|r|\\ (a,r)=1} }.
\]
and $h$ is only  defined for $x\neq 0$ as:
\[h(x,y)=\begin{cases}|x|^{-1} & \mbox{if $|y|<|x|$}\\ 0 & \mbox{otherwise.}\end{cases}\]
%
\end{lemma}
\begin{proof}
Indeed, using Lemma~\ref{lem:dissection}, we may rewrite the integral expression of the delta function as
\begin{eqnarray*}
\delta(x)&=&\int_{\TT}\psi(\alpha x)d\alpha\\
&=& \sum_{\substack{r\in \cO\\
|r|\leq \hat Q\\
\text{$r$ monic}
} }
\sumstar_{\substack{
|a|<|r|} }\int_{|r\alpha-a|<\hat Q^{-1}}\psi(\alpha x)d\alpha\\ &=& \sum_{\substack{r\in \cO\\
|r|\leq \hat Q\\
\text{$r$ monic}
} }
\sumstar_{\substack{
|a|<|r|} }\psi\left(\frac{ax}{r}\right)\int_{|\alpha|<|r|^{-1}\hat Q^{-1}}\psi(\alpha x)d\alpha,
\end{eqnarray*}
where the last equality follows from a linear change of variables. Note that if we define
\[
h(x,y):=|x|^{-1}\int_{\TT}\psi(y x^{-1} u) du,
\]
then
\begin{eqnarray*}
h\left(\frac{r}{t^Q},\frac{x}{t^{2Q}}\right)&=& \hat{Q}|r|^{-1}\int_{\TT}\psi\left(\frac{xu}{rt^Q}\right)du\\
&=&\hat{Q}^2\int_{|\alpha|<|r|^{-1}\hat{Q}^{-1}}\psi(\alpha x)d\alpha.
\end{eqnarray*}
The last statement follows from Lemma~\ref{lem:orthog}.
\end{proof}
\subsection{Smooth sum $N(w,\boldsymbol{\lambda})$}
As previously stated, we want to take a weight function $w\in S(K_{\infty}^d)$ and use it to define a weighted sum over all the solutions whose existence we want to show. We will denote such a sum by $N(w,\boldsymbol{\lambda})$, and then we will use the circle method to give a lower bound for this quantity. A positive lower bound would prove existence of the desired solutions.\\
\\

Let $w$ be a Schwarz-Bruhat weight function defined on $\Ki^d$. Assume that $\vec{x}\in \mathcal{O}^d$ satisfies the conditions $F(\vec{x})=f$ and $\vec{x}\equiv \boldsymbol{\lambda}\bmod g$. We uniquely write $\vec{x}=g\vec{t}+ \boldsymbol{\lambda},$ where $\vec{t}\in  \mathcal{O}^d$ and $\boldsymbol{\lambda}=(\lambda_1,\dots,\lambda_d)$ for $\lambda_i$ of degree strictly less than that of $g$. Define 
\begin{equation}\label{defk}
k:= \frac{f-F(\boldsymbol{\lambda})}{g}.
\end{equation}
Write $F(\vec{x})=\vec{x}^{\intercal}A\vec{x}$ for some symmetric $d\times d$ matrix $A$ with $\mathcal{O}$-coefficients. This is possible because $q$ is odd. If
$F(\vec{x}) =f$, then $g^2F(\vec{t})+2g\boldsymbol{\lambda}^{\intercal} A\vec{t} =f-F(\boldsymbol{\lambda})$ which implies that $g|2\boldsymbol{\lambda}^{\intercal}A\vec{t}-k.$ 
 Then, 
$F(\vec{t})+\frac{1}{g}(2\boldsymbol{\lambda}^{\intercal} A\vec{t}-k )=0.$
We also define 
\begin{equation}\label{Gteq}G(\vec{t}):=\frac{F(g\vec{t}+\boldsymbol{\lambda})-f}{g^2}=F(\vec{t})+\frac{1}{g}(2\boldsymbol{\lambda}^{\intercal} A\vec{t}-k).\end{equation}
Finally, we  define 
\begin{equation}\label{Nsum}N(w,\boldsymbol{\lambda}):=\sum_{\vec{t}\in\mathcal{O}^d} w(g\vec{t}+\boldsymbol{\lambda}) \delta{(G(\vec{t}))}.
\end{equation}  Note that $N(w,\boldsymbol{\lambda})$ is the weighted number of $\boldsymbol{x}\in\mathcal{O}^d$ satisfying the conditions  $F(\vec{x})=f$ and $\vec{x}\equiv \boldsymbol{\lambda}\bmod g$. Furthermore, $\delta(G(\vec{t}))\neq 0$ if and only if $G(\vec{t})=0$, in which case~\eqref{Gteq} implies that $g|2\boldsymbol{\lambda}^{\intercal} A\vec{t}-k$. In what follows, we write this latter condition in terms of character sums. Using Lemma~\ref{lem:orthogsum}, we have for $\gamma\in\Ki$
\[
\frac{1}{|g|}\sum_{\substack{\ell\in \cO\\ |\ell|<|g|}}\psi(\gamma \ell)=\begin{cases}1 & \mbox{if $|(\!(\gamma)\!)|<|g|^{-1}$}\\ 0 & \mbox{otherwise.}\end{cases}
\]
In particular,
\begin{equation}\label{inparticular}
\frac{1}{|g|}\sum_{\substack{\ell\in \cO\\ |\ell|<|g|}}\psi\left(\frac{(2\boldsymbol{\lambda}^{\intercal} A\vec{t}-k)\ell}{g}\right)=\begin{cases}1 & \mbox{if $|(\!(\frac{2\boldsymbol{\lambda}^{\intercal} A\vec{t}-k}{g})\!)|<|g|^{-1}$}\\ 0 & \mbox{otherwise.}\end{cases}
\end{equation}
The condition 
\[\left|(\!(\frac{2\boldsymbol{\lambda}^{\intercal} A\vec{t}-k}{g})\!)\right|<|g|^{-1}\]
is satisfied precisely when 
\[(\!(\frac{2\boldsymbol{\lambda}^{\intercal} A\vec{t}-k}{g})\!)=0,\]
that is, when $g|2\boldsymbol{\lambda}^{\intercal} A\vec{t}-k$. Combining~\eqref{Nsum} with~\eqref{inparticular}, we may rewrite
\begin{equation*}
N(w,\boldsymbol{\lambda})=\frac{1}{|g|}\sum_{\substack{\ell\in \cO\\ |\ell|<|g|}}\sum_{\vec{t}\in\mathcal{O}^d} \psi\left(\frac{(2\boldsymbol{\lambda}^{\intercal} A\vec{t}-k)\ell}{g}\right)w(g\vec{t}+\boldsymbol{\lambda}) \delta{(G(\vec{t}))}.
\end{equation*}
Note that in~\eqref{inparticular}, the character sum is nonzero if and only if $g|2\boldsymbol{\lambda}^{\intercal} A\vec{t}-k$, in which case $G(\vec{t})\in\mathcal{O}$. Therefore, the conditions of Lemma~\ref{delta} are satisfied when the character sum is nonzero. Applying \eqref{deltaeq} to $\delta(G(\vec{t}))$, we obtain

\begin{eqnarray*}
&&N(w,\boldsymbol{\lambda})\\&=&
\frac{1}{|g|\hQ^2}\sum_{\substack{\ell\in \cO\\ |\ell|<|g|}}\sum_{\vec{t}\in\mathcal{O}^d}
\sum_{\substack{
r\in \cO\\
|r|\leq \hat Q\\
\text{$r$ monic}
} }
\sumstar_{\substack{
|a|<|r|} }\psi\left(\frac{(2\boldsymbol{\lambda}^{\intercal} A\vec{t}-k)\ell}{g}+\frac{aG(\vec{t})}{r}\right)w(g\vec{t}+\boldsymbol{\lambda})h\left(\frac{r}{t^Q},\frac{G(\vec{t})}{t^{2Q}}\right)\\ &=& \frac{1}{|g|\hQ^2}\sum_{\substack{\ell\in \cO\\ |\ell|<|g|}}\sum_{\vec{t}\in\mathcal{O}^d}
\sum_{\substack{
r\in \cO\\
|r|\leq \hat Q\\
\text{$r$ monic}
} }
\sumstar_{\substack{
|a|<|r|} }\psi\left(\frac{(a+r\ell)(2\boldsymbol{\lambda}^{\intercal} A\vec{t}-k)+agF(\vec{t})}{gr}\right)w(g\vec{t}+\boldsymbol{\lambda})h\left(\frac{r}{t^Q},\frac{G(\vec{t})}{t^{2Q}}\right)\\&=&\frac{1}{|g|\hQ^2}\sum_{\substack{\ell\in \cO\\ |\ell|<|g|}}
\sum_{\substack{
r\in \cO\\
|r|\leq \hat Q\\
\text{$r$ monic}
} }
\sumstar_{\substack{
|a|<|r|} }\sum_{\vec{b}\in\mathcal{O}^d/(gr)}\sum_{\vec{s}\in\mathcal{O}^d}\psi\left(\frac{(a+r\ell)(2\boldsymbol{\lambda}^{\intercal} A\vec{b}-k)+agF(\vec{b})}{gr}\right)w(g(\vec{b}+gr\vec{s})+\boldsymbol{\lambda})\\&&\cdot h\left(\frac{r}{t^Q},\frac{G(\vec{b}+gr\vec{s})}{t^{2Q}}\right).
\end{eqnarray*}
From the Poisson summation formula, one deduces (see Lemma 2.1 of \cite{Browning}, for example) that for $v\in S(\Ki^d)$,
\[\sum_{\vec{t}\in\mathcal{O}^d}v(\vec{t})=\sum_{\vec{c}\in\mathcal{O}^d}\int_{\Ki^d}\psi(\left<\vec{c},\vec{t}\right>)v(\vec{t})d\vec{t}.\]
Applying this to the $\vec{s}$ variable in the above expression of $N(w,\boldsymbol{\lambda})$, we obtain the expression
\begin{eqnarray*}\label{form13}
&&N(w,\boldsymbol{\lambda})\\ &=&\frac{1}{|g|\hQ^2}\sum_{\substack{\ell\in \cO\\ |\ell|<|g|}}
\sum_{\substack{
r\in \cO\\
|r|\leq \hat Q\\
\text{$r$ monic}
} }
\sumstar_{\substack{
|a|<|r|} }\sum_{\vec{b}\in\mathcal{O}^d/(gr)}\sum_{\vec{c}\in\mathcal{O}^d}\psi\left(\frac{(a+r\ell)(2\boldsymbol{\lambda}^{\intercal} A\vec{b}-k)+agF(\vec{b})}{gr}\right)\\&&\cdot\int_{\Ki^d}\psi\left(\left<\vec{c},\vec{t}\right>\right)w(g(\vec{b}+gr\vec{t})+\boldsymbol{\lambda})h\left(\frac{r}{t^Q},\frac{G(\vec{b}+gr\vec{t})}{t^{2Q}}\right)d\vec{t}\\ &=& \frac{1}{|g|\hQ^2}\sum_{\substack{\ell\in \cO\\ |\ell|<|g|}}
\sum_{\substack{
r\in \cO\\
|r|\leq \hat Q\\
\text{$r$ monic}
} }
\sumstar_{\substack{
|a|<|r|} }\sum_{\vec{c}\in\mathcal{O}^d}\sum_{\vec{b}\in\mathcal{O}^d/(gr)}|gr|^{-d}\psi\left(\frac{(a+r\ell)(2\boldsymbol{\lambda}^{\intercal} A\vec{b}-k)+agF(\vec{b})-\left<\vec{c},\vec{b}\right>}{gr}\right)\\&&\cdot\int_{\Ki^d}\psi\left(\frac{\left<\vec{c},\vec{t}\right>}{gr}\right)w(g\vec{t}+\boldsymbol{\lambda})h\left(\frac{r}{t^Q},\frac{G(\vec{t})}{t^{2Q}}\right)d\vec{t}
\end{eqnarray*}
We express this in the condensed form
\begin{equation}\label{newequu}N(w,\boldsymbol{\lambda})=\frac{1}{|g|\hQ^2}\sum_{\substack{
r\in \cO\\
|r|\leq \hat Q\\
\text{$r$ monic}
} }\sum_{\vec{c}\in\mathcal{O}^d}|gr|^{-d}S_{g,r}(\vec{c})I_{g,r}(\vec{c}),\end{equation}
where $I_{g,r}(\vec{c})$ and $S_{g,r}(\vec{c})$ are defined by 
\begin{equation}\label{intt}
I_{g,r}(\vec{c}):= \int_{\Ki^d} h\left(\frac{r}{t^Q},\frac{G(\vec{t})}{t^{2Q}}\right) w(g\vec{t}+\boldsymbol{\lambda}) \psi\left(\frac{\left< \vec{c},\vec{t} \right>}{gr}    \right) d\vec{t},
\end{equation}
and
\begin{equation}\label{sala}S_{g,r}(\vec{c}):=\sum_{\substack{\ell\in \cO\\ |\ell|<|g|}}{\sum_{|a|<|r|}}^* S_{g,r}(a,\ell,\vec{c})\end{equation}
with
\begin{equation}\label{salam}
S_{g,r}(a,\ell,\vec{c}):=\sum_{\vec{b}\in \mathcal{O}^d/(gr)} \psi\left(\frac{(a+r\ell)(2\boldsymbol{\lambda}^{\intercal} A\vec{b}-k) + ag F(\vec{b})-\left< \vec{c},\vec{b} \right>} {gr}\right).
\end{equation}

In the next two sections, we bound from above $S_{g,r}$ and $I_{g,r}$.

\section{Bounds on the exponential sums $S_{g,r}(\vec{c})$}\label{boundsexponential}
Recall that $F(\vec{x})=\vec{x}^{\intercal}A\vec{x}$, and $\Delta=\det(A)$ is the discriminant. 
As in the statement of Theorem~\ref{strong}, we assume that $\gcd(f\Delta,g)=1$, and give an upper bound on an averaged sum of the $S_{g,r}(\vec{c})$. 
\begin{proposition}\label{prop:Supperbound}
We have the following upper bound
\[\sum_{\substack{r\in\mathcal{O}\\|r|<\hat{X}}}|g|^{-d}|r|^{-\frac{d+1}{2}}|S_{g,r}(\vec{c})|\ll_{F,\varepsilon}|g|^{\varepsilon}\hat{X}^{1+\varepsilon},\]
where $\hat{X}=O(|f|^C)$ for some fixed $C$.
\end{proposition}
Initially, a version of this result was proved by Heath-Brown (Lemma 28 of \cite{Brown}). This is a function field analogue of proposition 4.1 of the first author in \cite{Naser2}. We first prove a lemma indicating that most $S_{g,r}(a,\ell,\vec{c})$ vanish.
\begin{lemma}\label{lem:cform}Unless $\vec{c}\equiv 2(ar+\ell)A\boldsymbol{\lambda}\bmod g$, we have $S_{g,r}(a,\ell,\vec{c})=0$. Consequently, $S_{g,r}(\vec{c})=0$ unless $\vec{c}\equiv\alpha A\boldsymbol{\lambda}\bmod g$ for some $\alpha\in\mathcal{O}$.
\end{lemma}
\begin{proof} Recall summation~\eqref{salam} over $\vec{b}\in \mathcal{O}^d/(gr).$
Write $\vec{b}=r\vec{b}_1+\vec{b}_2$, where $\vec{b}_1$ is a vector modulo $g$ and $\vec{b}_2$ is a vector modulo $r$. We may then rewrite
\[S_{g,r}(a,\ell,\vec{c})=\sum_{\vec{b}_2} \psi\left(\frac{(a+r\ell)(2\boldsymbol{\lambda}^{\intercal} A\vec{b}_2-k) + ag F(\vec{b}_2)-\left< \vec{c},\vec{b}_2 \right>} {gr}\right)\sum_{\vec{b}_1}\psi\left(\frac{2(a+r\ell)\boldsymbol{\lambda}^{\intercal} A\vec{b}_1-\left< \vec{c},\vec{b}_1 \right>} {g}\right).\]
From Lemma~\ref{lem:orthogsum}, the second sum vanishes unless $\vec{c}\equiv 2(a+r\ell)A\lambda\bmod g$, which gives the first statement in the lemma. Since $S_{g,r}(\vec{c})$ is a sum of the $S_{g,r}(a,\ell,\vec{c})$, we obtain that it is zero unless possibly $\vec{c}\equiv\alpha A\boldsymbol{\lambda}\bmod g$ for some $\alpha\in\mathcal{O}$.
\end{proof}
By definition,
\[S_{g,r}(\vec{c})=\sum_{\substack{\ell\in \cO\\ |\ell|<|g|}}{\sum_{|a|<|r|}}^*\sum_{\vec{b}\in \mathcal{O}^d/(gr)} \psi\left(\frac{(a+r\ell)(2\boldsymbol{\lambda}^{\intercal} A\vec{b}-k) + ag F(\vec{b})-\left< \vec{c},\vec{b} \right>} {gr}\right).\]
Since the sum over $\ell$ is zero unless $g|2\boldsymbol{\lambda}^{\intercal}A\vec{b}-k$, in which case it contributes a factor of $|g|$, we have
\[S_{g,r}(\vec{c})=|g|{\sum_{|a|<|r|}}^*\sum_{\substack{\vec{b}\in \mathcal{O}^d/(gr)\\ g|2\boldsymbol{\lambda}^{\intercal}A\vec{b}-k}} \psi\left(\frac{a(2\boldsymbol{\lambda}^{\intercal} A\vec{b}-k) + ag F(\vec{b})-\left< \vec{c},\vec{b} \right>} {gr}\right).\]
We will give a bound on each of the $S_{g,r}(\vec{c})$. We do so by first decomposing $S_{g,r}(\vec{c})$ into the product of two sums and then bounding each of the two sums separately.\\
\\
Write $r=r_1r_2$, where $r_i\in\mathcal{O}$ and $\gcd(r_1,\Delta g)=1$ and such that the prime divisors of $r_2$ are among the prime divisors of $\Delta g$. In particular, $\gcd(r_1,gr_2)=1$, and so we may write
\[k=gr_2k_1+r_1k_2\]
and
\[a=r_2a_1+r_1a_2\]
for some $k_1,k_2\in\mathcal{O}$ and unique $a_1\in\mathcal{O}/(r_1)$, $a_2\in\mathcal{O}/(r_2)$. Similarly, we may find vectors $\vec{b}_1\in\mathcal{O}^d/(r_1)$ and $\vec{b}_2\in\mathcal{O}^d/(gr_2)$ such that
\[\vec{b}=gr_2\vec{b}_1+r_1\vec{b}_2.\]
If we set
\begin{equation}\label{S1}
S_1:=\sumstar_{|a_1|<|r_1|}\sum_{\vec{b}_1\in \mathcal{O}^d/(r_1)} \psi\left(\frac{2r_2a_1\boldsymbol{\lambda}^{\intercal}A\vec{b}_1+a_1(gr_2)^2F(\vec{b}_1)-\left<\vec{c},\vec{b}_1\right>-r_2a_1k_1    }{r_1}\right),
\end{equation}
and 
\begin{equation}\label{S2}
S_2:=|g|\sumstar_{|a_2|<|r_2|}\sum_{\substack{\vec{b}_2\in \mathcal{O}^d/(gr_2)\\ g|2\boldsymbol{\lambda}^{\intercal}A\vec{b}_2-k\overline{r_1}}}  \psi\left(\frac{2r_1a_2\boldsymbol{\lambda}^{\intercal}A\vec{b}_2+a_2gr_1^2F(\vec{b}_2)-\left<\vec{c},\vec{b}_2\right>-r_1a_2k_2}{gr_2}\right),
\end{equation}
where $\overline{r_1}$ is the inverse of $r_1$ mod $g.$ Then we see from a simple substitution of the above that
\[S_{g,r}(\vec{c})=S_1S_2.\]
We now proceed to bound $S_1$ and $S_2$. The bound on $S_1$ uses the Weil bound on exponential sums, while the bound on $S_2$ follows from the Cauchy--Schwarz inequality.\\
\\
In order to bound $S_1$ from above, consider the following situation. Let $G(\vec{x}):=\vec{x}^{\intercal}B\vec{x}$, where $B$ is a symmetric matrix $B\in M_d(\mathcal{O})$ with $D:=\det(B)\neq 0$. Furthermore, let $r\in\mathcal{O}$ be such that $\gcd(r,D)=1$, and for each $e\in\mathcal{O}/(r)$, $\vec{c},\vec{c}'\in\mathcal{O}^d/(r)$, define
\begin{equation}\label{SrG}S_r(G,\vec{c},\vec{c}',e):=\sumstar_{|a|<|r|}\sum_{\vec{b}\in\mathcal{O}^d/(r)} \psi\left(\frac{a(G(\vec{b})+\left<\vec{c}^{\prime},\vec{b}\right>+e)-\left<\vec{c},\vec{b}\right>}{r}\right).\end{equation}

\begin{proposition}\label{lem:S1bound}With the notation as above,
\[|S_1|\leq |r_{1}|^{\frac{d+1}{2}}\tau(r_1)|\gcd(r_1,f)|^{1/2},\]
and 
\[|S_2|\ll_{\Delta}|g|^d|r_2|^{\frac{d}{2}+1},\]
where $\tau(.)$ is the divisor function. In particular,
\[
|S_{g,r}(\vec{c})|\ll_{\Delta}|g|^d|r|^{\frac{d+1}{2}}|r_2|^{\frac{1}{2}}|\tau(r_1)|\gcd(r_1,f)|^{1/2}
\]
\end{proposition}
Prior to proving Proposition~\ref{lem:S1bound}, we prove a number of lemmas. The following two results pertain to diagonalizing our quadratic forms. We will also use them in the following sections.
\begin{lemma}\label{diag}Suppose $R$ is a discrete valuation ring with maximal ideal $\mathfrak{m}=(\pi)$, residue field $k$ of odd characteristic, and fraction field $\textup{Frac}(R)$. If $A$ is a symmetric matrix of size $d$ with coefficients in $\textup{Frac}(R)$, then there is a matrix $g\in\GL_d(R)$ such that
\[g^{\intercal}Ag\]
is a diagonal matrix. 
\end{lemma}
\begin{proof}We proceed by induction on $d$. The case $d=1$ is trivial. By multiplying by a suitable power of $\pi$, we may assume without loss of generality that $A$ is a matrix over $R$, and $\bar{A}:=A\bmod\mathfrak{m}$ is nonzero as a matrix over $R/\mathfrak{m}=k$. Since $k$ is a field of odd characteristic, there is a matrix $\bar{h}\in\GL_d(k)$ such that $\bar{h}^{\intercal}\overline{A}\bar{h}=\text{diag}(\bar{\eta}_1,\hdots,\bar{\eta}_d)$. Choose a lift $h$ of $\bar{h}$ with coefficients in $R$. Since $\bar{A}\neq 0$, we may assume without loss of generality that $\bar{\eta}_1\neq 0.$ Let $A_1:=h^{\intercal}Ah=\begin{bmatrix}\vec{a}_1,\dots,\vec{a}_d  \end{bmatrix}=\begin{bmatrix}a_{i,j}  \end{bmatrix},$ where $\vec{a}_i$ is the $i$th column vector of $A_1,$ and $a_{i,j}$ is the $i$th and $j$th coordinate of $A_1.$ Since $\bar{\eta}_1\neq 0,$  $a_{1,1} \in R^*$.
Let 
\[
H:=\begin{bmatrix}1 &-\frac{a_{1,2}  }{a_{1,1}} & \dots & -\frac{a_{1,d}}{a_{1,1}}
\\
0 &
\\
\vdots &  & I_{{d-1}\times{d-1}}
\\
0
   \end{bmatrix}\in \GL_d(R).
\]
It is easy to check that 
\[
H^{\intercal}A_1H=\begin{bmatrix}a_{11} &0 & \dots & 0
\\
0 &
\\
\vdots &  & A_2
\\
0
   \end{bmatrix},
\]
where $A_2^{\intercal}=A_2\in M_{{(d-1)}\times {(d-1)}}(R).$ The lemma follows from  the induction hypothesis on $A_2.$
\end{proof}
\begin{corollary}\label{cor:diag}Suppose $R$ is a $k$-algebra that is a principal ideal domain, $k$ a field of odd characteristic. Furthermore, suppose $I=\mathfrak{m}_1^{k_1}\dots\mathfrak{m}_s^{k_s}$ is a product of maximal ideals of $R$, $\mathfrak{m}_i$ distinct. If $A$ is a symmetric matrix over $R/I$ of size $d$, then there is a matrix $g\in\GL_d(R/I)$ such that
\[g^{\intercal}Ag\]
is a diagonal matrix.
\end{corollary}
\begin{proof}
By the Chinese remainder theorem, $R/I\simeq R/\mathfrak{m}_1^{k_1}\times\dots\times R/\mathfrak{m}_s^{k_s}$. Furthermore, this is a product of discrete valuation rings with residue fields of odd characteristic. The conclusion follows by applying Lemma~\ref{diag} to each local ring $R/\mathfrak{m}_i^{k_i}$.
\end{proof}
In bounding $S_1$, the following lemma will allow us to reduce to the case where $r_1=\varpi^k$, $\varpi$ an irreducible polynomial of $\mathcal{O}$.
\begin{lemma}[Multiplicativity of $S_r(G,\vec{c},\vec{c}',e)$]\label{lem:multiplicative}Suppose $r=uv$ for coprime $u,v\in\mathcal{O}$. Then
\[S_r(G,\vec{c},\vec{c}',e)=S_u(G,\bar{v}\vec{c},\vec{c}',e)S_v(G,\bar{u}\vec{c},\vec{c}',e),\]
where $\bar{v}$ is the inverse of $v$ mod $u$ and $\bar{u}$ is the inverse of $u$ mod $v.$
\end{lemma}
\begin{proof}Since $u$ and $v$ are coprime, as $\vec{b}_1$ ranges over $\mathcal{O}^d/(u)$ and $\vec{b}_2$ ranges over $\mathcal{O}^d/(v)$, the vector
\[\vec{b}=v\vec{b}_1+u\vec{b}_2\]
ranges over a complete set of vectors modulo $uv=r$. Similarly, as $a_1$ ranges over $\mathcal{O}/(u)$ and $a_2$ ranges over $\mathcal{O}/(v)$,
\[a=va_1+ua_2\]
ranges over a complete set of polynomials modulo $uv=r$. Making these substitutions, the summands in $S_r(G,\vec{c},\vec{c}',e)$
become
\begin{eqnarray*}
&&\psi\left(\frac{a(G(\vec{b})+\left<\vec{c}^{\prime},\vec{b}\right>+e)-\left<\vec{c},\vec{b}\right>}{r}\right)\\&=&
\psi\left(\frac{(va_1+ua_2)(G(v\vec{b}_1+u\vec{b}_2)+\left<\vec{c}^{\prime},v\vec{b}_1+u\vec{b}_2\right>+e)-\left<\vec{c},v\vec{b}_1+u\vec{b}_2\right>}{uv}\right)\\&=& 
\psi\left(\frac{(va_1+ua_2)(v^2G(\vec{b}_1)+u^2G(\vec{b}_2)+v\left<\vec{c}',\vec{b}_1\right>+u\left<\vec{c}',\vec{b}_2\right>+e)-v\left<\vec{c},\vec{b}_1\right>-u\left<\vec{c},\vec{b}_2\right>}{uv}\right)\\&=& \psi\left(\frac{a_1(v^2G(\vec{b}_1)+\left<v\vec{c}',\vec{b}_1\right>+e)-\left<\vec{c},\vec{b}_1\right>}{u}\right)\psi\left(\frac{a_2(u^2G(\vec{b}_2)+\left<u\vec{c}',\vec{b}_2\right>+e)-\left<\vec{c},\vec{b}_2\right>}{v}\right)\\&=& \psi\left(\frac{a_1(G(v\vec{b}_1)+\left<\vec{c}',v\vec{b}_1\right>+e)-\left<\bar{v}\vec{c},v\vec{b}_1\right>}{u}\right)\psi\left(\frac{a_2(G(u\vec{b}_2)+\left<\vec{c}',u\vec{b}_2\right>+e)-\left<\bar{u}\vec{c},u\vec{b}_2\right>}{v}\right).
\end{eqnarray*}
Since $u$ and $v$ are coprime, $u\vec{b}_2$ and $v\vec{b}_1$ range over a complete set of residues modulo $v$ and $u$, respectively. As a result,
\[S_r(G,\vec{c},\vec{c}',e)=S_u(G,\bar{v}\vec{c},\vec{c}',e)S_v(G,\bar{u}\vec{c},\vec{c}',e),\]
as required.
\end{proof}
We will obtain Proposition~\ref{lem:S1bound} using the function-field analogue of the Weil bound on Kloosterman and Sali\'e sums, whose proof we sketch in the following.
\begin{lemma}\label{Weilbound}Suppose $m,n,c\in\mathbb{F}_q[t]$, $c\neq 0$, and $\theta\in\{0,1\}$. Then
\[\left|\sumstar_{|x|<|c|}\left(\frac{x}{c}\right)^{\theta}\psi\left(\frac{mx+n\overline{x}}{c}\right)\right|\leq \tau(c)|c|^{1/2}|\gcd(m,n,c)|^{1/2}.\]
\end{lemma}
\begin{proof}
By a standard computation as in Lemma~\ref{lem:multiplicative}, we may reduce to when $c$ is a prime power $\varpi^{k}$. Furthermore, we may assume that $\varpi\nmid mn$; otherwise we have Ramanujan sums which may be explicitly computed as in the case of integers and shown to satisfy the above bound (see equations (3.1)-(3.3) of~\cite{MR2061214} for usual Ramanujan sums).\\
\\
First, note that when $k=1$, then this is the Weil bound on Kloosterman sums over the finite field $\mathcal{O}/(\varpi)$. This is a consequence of Theorem 10 of~\cite{Kowalski:expsums}. For Sali\'e sums, it is a consequence of Theorem 2.19 of \textit{loc.cit}. We may therefore assume that $k\geq 2$.\\

Note that by factoring out a factor of $|\gcd(m,n,c)|$ and summing modulo $c/\gcd(m,n,c)$, we may assume without loss of generality that $\gcd(m,n,c)=1$. Let us assume furthermore that we have a Kloosterman sum, that is, $\theta=0$.\\
\\
Write $x=a_1+a_2\varpi^{\lfloor k/2\rfloor}$, where $a_1$ is chosen modulo $\varpi^{\lfloor k/2\rfloor}$ and is relatively prime to $\varpi$, and $a_2$ is chosen modulo $\varpi^{\lceil k/2\rceil}$. Furthermore, note that
\[\overline{a_1+a_2\varpi^{\lfloor k/2\rfloor}}\equiv \overline{a_1}-\overline{a_1}^2a_2\varpi^{\lfloor k/2\rfloor}+\overline{a_1}^3a_2^2\varpi^{2\lfloor k/2\rfloor}\bmod\varpi^{k},\]
where the inverses are computed modulo $\varpi^k$. Making these substitutions, we obtain 
\begin{eqnarray*}
&&\psi\left(\frac{mx+n\overline{x}}{\varpi^k}\right)\\&=&\psi\left(\frac{m(a_1+a_2\varpi^{\lfloor k/2\rfloor})+n(\overline{a_1+a_2\varpi^{\lfloor k/2\rfloor}})}{\varpi^k}\right)\\ &=& \psi\left(\frac{(a_1+a_2\varpi^{\lfloor k/2\rfloor})m+n(\overline{a_1}-\overline{a_1}^2a_2\varpi^{\lfloor k/2\rfloor}+\overline{a_1}^3a_2^2\varpi^{2\lfloor k/2\rfloor})}{\varpi^k}\right)\\ &=& \psi\left(\frac{ma_1+n\overline{a_1}+\varpi^{\lfloor k/2\rfloor}\left(a_2\left(m-\overline{a_1}^2n\right)+\overline{a_1}^3a_2^2\varpi^{\lfloor k/2\rfloor}n\right)}{\varpi^k}\right)\\ &=& \psi\left(\frac{ma_1+n\overline{a_1}}{\varpi^k}\right)\psi\left(\frac{a_2\left(m-\overline{a_1}^2n\right)+\overline{a_1}^3a_2^2\varpi^{\lfloor k/2\rfloor}n}{\varpi^{\lceil k/2\rceil}}\right).
\end{eqnarray*}
Summation over $a_2\bmod\varpi^{\lceil k/2\rceil}$ gives us zero unless
\[m-n\overline{a_1}^2\equiv 0\bmod \varpi^{\lfloor k/2\rfloor}.\]
For such $a_1$, if $k$ is even, summing over $a_2$ contributes a factor of $|\varpi|^{k/2}$. If $k$ is odd, then for such $a_1$, summing over $a_2$ contributes a factor of
\[|\varpi|^{\lfloor k/2\rfloor}\sum_{y\bmod \varpi}\psi\left(\frac{-\overline{a_1}y^2n}{\varpi}\right).\]
This sum is a Gauss sum, and is of norm $|\varpi|^{1/2}$ unless $\varpi|n$, which we assumed not to be the case at the beginning of this proof. Therefore, when $k$ is odd, summing over $a_2$ contributes a factor of $|\varpi|^{k/2}$ as well. Since $\varpi\nmid mn$, the congruence above has at most $2$ solutions $a_1$ modulo $\varpi^{\lfloor k/2\rfloor}$. Putting these together, the conclusion follows.\\
\\
For the case of Sali\'e sums, that is $\theta=1$, the proof is similar. See Lemmas 12.2 and 12.3 of~\cite{MR2061214}.
\end{proof}
\begin{proof}[Proof of Proposition~\ref{lem:S1bound}]
Recall~\eqref{SrG} and its notation. We make some general reductions of~\eqref{SrG} and then specialize to~\eqref{S1} to obtain the desired bound on $S_1$. In calculating $S_r(G,\vec{c},\vec{c}',e)$, we may assume without loss of generality that $G$ is a diagonal matrix modulo $r$, that is,
\[G(\vec{x})=\sum_{i=1}^d\alpha_ix_i^2\]
for some $\alpha_i\in\mathcal{O}$. This is possible by Corollary~\ref{cor:diag} with $R=\mathcal{O}$ and $I=(r)$. Consequently,
\[S_r(G,\vec{c},\vec{c}',e)=\sumstar_{|a|<|r|} \psi\left(\frac{ae}{r}\right)  \prod_{j=1}^{d} \sum_{b\in\mathcal{O}/(r)} \psi\left(\frac{a(\alpha_jb^2+c_j^{\prime}b)-c_jb}{r} \right).\]
We complete the square to obtain
\begin{eqnarray*}S_r(G,\vec{c},\vec{c}',e)&=&\sumstar_{|a|<|r|} \psi\left(\frac{ae}{r}\right)\prod_{j=1}^{d} \sum_{b\in\mathcal{O}/(r)} \psi\left(\frac{a\alpha_j\left(b+\overline{2a\alpha_j}(ac_j'-c_j)\right)^2-\overline{4a\alpha_j}(ac_j^{\prime}-c_j)^2)}{r}\right)
\\
 &=&\sumstar_{|a|<|r|} \psi\left(\frac{ae}{r}\right)\prod_{j=1}^{d} \psi\left(\frac{-\overline{4a\alpha_j}(ac_j^{\prime}-c_j)^2)}{r}\right)\sum_{b\in\mathcal{O}/(r)} \psi\left(\frac{a\alpha_j\left(b+\overline{2a\alpha_j}(ac_j'-c_j)\right)^2}{r}\right)  \end{eqnarray*}
The internal sum is equal to $\left(\frac{a\alpha_j}{r}\right)\tau_r$, where $\left(\frac{a\alpha_j}{r}\right)$ is the Jacobi symbol and $\tau_r:=\sum_{x\in\mathcal{O}/(r)}\psi\left(\frac{x^2}{r}\right)$ is a Gauss sum. Therefore,
\[S_r(G,\vec{c},\vec{c}',e)=\tau_r^d\left(\frac{D}{r}\right)\psi\left(\frac{\sum_{j}\overline{2\alpha_j}c_jc_j'}{r}\right){\sum_{|a|<|r|}}^* \left(\frac{a}{r} \right)^d \psi\left(\frac{a(e-\sum_{j}\overline{4\alpha_j}c_j^{\prime 2})-\bar{a}\sum_{j}\overline{4\alpha_j}c_j^2}{r}  \right).\]
Note that $|\tau_r^d|=|r|^{d/2}$. In this formula, $\left(\frac{D}{r}\right)$ and $\left(\frac{a}{r} \right)$ are Jacobi symbols. In light of Lemma~\ref{lem:multiplicative}, we proceed to bound $S_{\varpi^k}(G,\vec{c},\vec{c}',e)$ for $k\geq 1$ and $\varpi\in\mathcal{O}$ irreducible. It suffices to bound the sums
\[\sumstar_{|a|<|\varpi^k|} \left(\frac{a}{\varpi^k} \right)^d \psi\left(\frac{a(e-\sum_{j}\overline{4\alpha_j}c_j^{\prime 2})-\bar{a}\sum_{j}\overline{4\alpha_j}c_j^2}{\varpi^k}  \right).\]
Specializing now to $S_1$ as in~\eqref{S1}, we take $\varpi^k\|r_1$, $G=(gr_2)^2F$, $\vec{c}'=2r_2A\boldsymbol{\lambda}$, and $e=-r_2k_1$. In this case,
\[e-\sum_{j}\overline{4\alpha_j}c_j^{\prime 2}\equiv -r_2k_1-F(\boldsymbol{\lambda})\bar{g}^2\equiv (gr_1k_2-f)\bar{g}^2\equiv -f\bar{g}^2\bmod \varpi^k.\]
Making this substitution and changing $a$ to $ag^2$, we obtain
\[\sumstar_{|a|<|\varpi^k|} \left(\frac{a}{\varpi^k} \right)^d \psi\left(\frac{-af-\bar{a}\overline{g}^4\sum_{j}\overline{4\alpha_j}c_j^2}{\varpi^k}  \right).\]
Applying Lemma~\ref{Weilbound}, we obtain
\[|S_{\varpi^k}(G,\vec{c},\vec{c}',e)|=\left|\tau_{\varpi^k}^d\sumstar_{|a|<|\varpi^k|} \left(\frac{a}{\varpi^k} \right)^d \psi\left(\frac{-af-\bar{a}\overline{g}^4\sum_{j}\overline{4\alpha_j}c_j^2}{\varpi^k}  \right)\right|\leq \tau(\varpi^k)|\varpi^k|^{\frac{d+1}{2}}|\gcd(f,\varpi^k)|^{1/2}.\]
Applying Lemmas~\ref{lem:multiplicative}, we obtain for every $r_1$
\[|S_1|\leq \tau(r_1)|r_1|^{\frac{d+1}{2}}|\gcd(r_1,f)|^{1/2}.\]
This concludes the proof of the first part of Proposition~\ref{lem:S1bound}.
\\
\\
 We now bound $S_2$ from above. The proof uses the Cauchy--Schwarz inequality. Recall that
\[S_2:=|g|\sumstar_{|a_2|<|r_2|}\sum_{\substack{\vec{b}_2\in \mathcal{O}^d/(gr_2)\\ g|2\boldsymbol{\lambda}^{\intercal}A\vec{b}_2-k\overline{r_1}}}  \psi\left(\frac{2r_1a_2\boldsymbol{\lambda}^{\intercal}A\vec{b}_2+a_2gr_1^2F(\vec{b}_2)-\left<\vec{c},\vec{b}_2\right>-r_1a_2k_2}{gr_2}\right).\]
Applying the Cauchy--Schwarz inequality to the $a_2$ variable, we obtain
\begin{eqnarray*}
&&|S_2|^2\\&\leq& |g|^2\varphi(r_2)\sumstar_{|a_2|<|r_2|}\left|\sum_{\substack{\vec{b}_2\in \mathcal{O}^d/(gr_2)\\ g|2\boldsymbol{\lambda}^{\intercal}A\vec{b}-k\overline{r_1}}}  \psi\left(\frac{2r_1a_2\boldsymbol{\lambda}^{\intercal}A\vec{b}_2+a_2gr_1^2F(\vec{b}_2)-\left<\vec{c},\vec{b}_2\right>-r_1a_2k_2}{gr_2}\right)\right|^2\\ &=& |g|^2\varphi(r_2)\sumstar_{|a_2|<|r_2|}\sum_{\substack{\vec{b}_2,\vec{b}_2'\in \mathcal{O}^d/(gr_2)\\ g|2\boldsymbol{\lambda}^{\intercal}A\vec{b}_2-k\overline{r_1}, g|2\boldsymbol{\lambda}^{\intercal}A\vec{b}_2'-k\overline{r_1}}}  \psi\left(\frac{2r_1a_2\boldsymbol{\lambda}^{\intercal}A(\vec{b}_2-\vec{b}_2')+a_2gr_1^2(F(\vec{b}_2)-F(\vec{b}_2'))-\left<\vec{c},\vec{b}_2-\vec{b}_2'\right>}{gr_2}\right)
\end{eqnarray*}
Making the substitution $\vec{u}=\vec{b}_2-\vec{b}_2'$, we obtain
\[|S_2|^2\leq |g|^2\varphi(r_2)\sumstar_{|a_2|<|r_2|}\sum_{\substack{\vec{b}_2,\vec{u}\in \mathcal{O}^d/(gr_2)\\ g|2\boldsymbol{\lambda}^{\intercal}A\vec{b}_2-k\overline{r_1}, g|2\boldsymbol{\lambda}^{\intercal}A\vec{u}}}  \psi\left(\frac{2r_1a_2\boldsymbol{\lambda}^{\intercal}A\vec{u}+a_2gr_1^2(2\vec{b}_2^{\intercal}A\vec{u}+F(\vec{u}))-\left<\vec{c},\vec{u}\right>}{gr_2}\right).\]
The sum over $\vec{b}_2$ is zero unless $r_2|\Delta\gcd(\vec{u})$, where $\gcd(\vec{u}):=\gcd(u_1,\dots,u_d).$ This implies that the summation is non-zero only if
\[\vec{u}\in(r_2\mathcal{O}/(\gcd(\Delta,r_2)gr_2))^d\simeq (\mathcal{O}/(\gcd(\Delta,r_2)g))^d.\]
Hence,
\begin{eqnarray*}|S_2|^2&\leq& |g|^2\varphi(r_2)\sumstar_{|a_2|<|r_2|}\sum_{\substack{\vec{b}_2\in \mathcal{O}^d/(gr_2)\\g|2\boldsymbol{\lambda}^{\intercal}A\vec{b}_2-k\overline{r_1}}}\sum_{\substack{\vec{u}\in (\mathcal{O}/(\gcd(\Delta,r_2)g))^d\\ g|2\boldsymbol{\lambda}^{\intercal}A\vec{u}}}1\\&\ll_{\Delta}& |g|^{2d}\varphi(r_2)^2|r_2|^d\\&\ll_{\Delta}& |g|^{2d}|r_2|^{d+2}.\end{eqnarray*}
Taking square roots, we obtain
\[|S_2|\ll_{\Delta}|g|^{d}|r_2|^{\frac{d}{2}+1},\]
as required.
\end{proof}
We now put together the above results to prove Proposition~\ref{prop:Supperbound}.
\begin{proof}[Proof of Proposition~\ref{prop:Supperbound}] 
As before, write $r=r_1r_2$, where $\gcd(r_1,g\Delta)=1$ and the prime divisors of $r_2$ are among those of $g\Delta$. By construction, we know that $|S_{g,r}(\vec{c})|=|S_1||S_2|$. Therefore, from Proposition~\ref{lem:S1bound}, we have
\begin{eqnarray*}&&\sum_{\substack{r\in\mathcal{O}\\|r|<\hat{X}}}|g|^{-d}|r|^{-\frac{d+1}{2}}|S_{g,r}(\vec{c})|\\&\ll_{\Delta}& \sum_{\substack{r\in\mathcal{O}\\|r|<\hat{X}}}\tau(r_1)|r_2|^{1/2}|\gcd(r_1,f)|^{1/2}\\ &\leq& \hat{X}^{\varepsilon}\sum_{\substack{r\in\mathcal{O}\\|r|<\hat{X}}}|r_2|^{1/2}|\gcd(r_1,f)|^{1/2}\\&=&\hat{X}^{\varepsilon}\sum_{\substack{r_1\in\mathcal{O}\\|r_1|<\hat{X}}}|\gcd(r_1,f)|^{1/2}\sum_{\substack{r_2\in\mathcal{O}\\|r_2|<\hat{X}/|r_1|}}|r_2|^{1/2}.
\end{eqnarray*}
The second (internal) sum can be bounded using
\[\sum_{\substack{r_2\in\mathcal{O}\\|r_2|<\hat{X}/|r_1|}}|r_2|^{1/2}\leq\sum_{d|(g\Delta)^{\infty}, |d|<\hat{X}/|r_1|} |d|^{1/2} \frac{\hat{X}}{|r_1d|} \leq \hat{X}/|r_1|\sum_{d|(g\Delta)^{\infty}, |d|<\hat{X}/|r_1|} 1
 \ll \frac{\hat{X}}{|r_1|}  |g\Delta|^{\varepsilon}\hat{X}^{\varepsilon}.\]
Hence,
\[\sum_{\substack{r_1\in\mathcal{O}\\|r_1|<\hat{X}}}|\gcd(r_1,f)|^{1/2}\sum_{\substack{r_2\in\mathcal{O}\\|r_2|<\hat{X}/|r_1|}}|r_2|^{1/2}\ll \hat{X}|g\Delta|^{\varepsilon}\hat{X}^{\varepsilon}\sum_{\substack{r_1\in\mathcal{O}\\|r_1|<\hat{X}}}\frac{|\gcd(r_1,f)|^{1/2}}{|r_1|},\]
from which the conclusion follows since this latter sum is $\ll \hat{X}^{\varepsilon}$.
\end{proof}

\section{Analytic functions on $\TT^d$}\label{analyticsection}
In order to prove our main theorem, it turns out that we need to do analysis not just using polynomials over $\Ki$, but also using  convergent Taylor series. We begin by defining a space of analytic functions defined on $\TT^d$  that extends the space of polynomials. Let  $\cO_{\infty}:=\{x\in K_{\infty}: |\alpha|\leq 1  \}.$ Define
 \[C^{\omega}(\TT^d):=\left\{\sum_{(n_1,\dots,n_d)\in \mathbb{Z}_{\geq 0}^d}a_{(n_1,\dots,n_d)} x_1^{n_1}\dots x_d^{n_d}: a_{(n_1,\dots,n_d)} \in \cO_{\infty} \right\}. \]
It is easy to see that the above Taylor expansions are convergent for $(u_1,\dots,u_d)\in \TT^d.$ 
 When $d=1$, aside from polynomials in $\cO_{\infty}[x]$, examples of analytic functions on $\TT$ are
\[\frac{1}{1-x}:=\sum_{k=0}^{\infty}x^k,\]
and
\[(1+x)^{1/2}:=\sum_{k=0}^{\infty}\binom{1/2}{k}x^k.\]
This square root function is defined since the base characteristic is odd. We define the partial derivatives $\frac{\partial}{\partial x_i}$ for $1 \leq i\leq d$ on $C^{\omega}(\TT^d)$ to be the formal derivation  operator which acts on the monomials as  $\frac{\partial}{\partial x_i} x_1^{n_1}\dots x_d^{n_d}=n_i  x_1^{n_1}\dots x_i^{n_i-1}\dots x_d^{n_d}$ and extend them by linearity to power series. It is easy to check that it sends $C^{\omega}(\TT^d)$  to itself. A point $\vec{a}\in\TT^d$ is said to be a \textit{critical} point of $\phi\in C^{\omega}(\TT^d)$ if all partial derivatives at $\vec{a}$ are zero. The \textit{Hessian} of $\phi$ is given by the matrix
\[H_{\phi}:=\left[\frac{\partial^2\phi}{\partial x_i\partial x_j}\right]_{1\leq i,j\leq d}.\]
Define the space
\[
C^{\omega}(\TT^m,\TT^n):= \{\Phi=(\phi_1,\dots,\phi_n): \phi_j\in C^{\omega}(\TT^m) \text{ and }\phi_j(\vec{0})\in\TT  \}.
\] 
For $\Phi \in C^{\omega}(\TT^m,\TT^n)$ define the Jacobian matrix
\(
J\Phi:=\begin{bmatrix}   
\frac{\partial \phi_i}{\partial x_j} 
\end{bmatrix}_{i,j},
\)
 where $1\leq i\leq n$  and $1\leq j\leq m.$ For $m=n$ define the Jacobian determinant to be $\det (J\Phi).$ We also have the following change of variables formula, which readily follows from Igusa \cite[Proposition 7.4.1]{igusa}.

\begin{lemma}\label{lem:change}
Suppose $\Phi\in C^{\omega}(\TT^d,\TT^d)$, and suppose that $\Phi^{-1}$ exists and is also in $C^{\omega}(\TT^d,\TT^d)$. Then for every integrable $f:\TT^d\rightarrow\mathbb{C}$,
\[\int_{\TT^d}f(\vec{u})d\vec{u}=\int_{\TT^d} f(\Phi\vec{v}) |\det(J\Phi(\vec{v}))|d\vec{v}.\]
\end{lemma}

\subsection{The analytic automorphisms of $\TT^d$ }
In this section, we define the group of the analytic automorphisms of $\TT^d.$ We use this group in order to simplify and reduce the computations of our oscillatory integrals into Gaussian integrals. Recall that by  Schwarz's Lemma the analytic automorphisms  of the disk in the complex plane  which fixes the origin are just rotations. Unlike the disk in the complex plane the group of analytic automorphisms of the disk $\TT^d$ is enormous.  Define
\[
\cA_{\infty}(\TT^d):=\big\{\Phi\in C^{\omega}(\TT^d,\TT^d): |\det (J\Phi(\vec{0})) |=1,  \text{ and }  \Phi(\vec{0})=0   \big\}.
\]
\begin{proposition}\label{auto}
$\cA_{\infty}(\TT^d)$ is a group under the composition of functions and it preserves the Haar measure on $\TT^d$. 
\end{proposition}

\begin{proof}
By the product rule of the Jacobian it is easy to see that $\cA_{\infty}(\TT^d)$ is closed under the composition of functions.  The identity function is the identity element of $\cA_{\infty}(\TT^d).$ It is enough to construct the inverse of $\Phi \in \cA_{\infty}(\TT^d).$ We prove the existence of the inverse by solving a  recursive system of linear equations.
\\

 First,  we explain it when  $d=1.$ We have $\Phi=\sum_{i=1}^{\infty} a_ix^i$ with $|a_1|=|\det(J\Phi(\vec{0}))|=1$.  Let $\bar{\Psi}(x):=a_1^{-1}x.$    We note that 
$J(\bar{\Psi} \circ \Phi)(\vec{0})= 1.$ Without loss of the generality, we assume that $a_1= 1.$ We wish to find $\Psi=\sum_{i=1}^{\infty} b_ix^i \in \cA_{\infty}(\TT)$ such that $ \Psi \circ \Phi(x) =x. $ This implies that $b_1=1,$ and we obtain the following system of linear equations in $(b_n)_{n\geq 2}$ by equating the $x^n$ coefficients of $\Psi \circ \Phi(x) $ for $n\geq 2:$
\[
b_n  +\sum_{i=1}^{n-1} b_i (\text{some polynomial in $a_1,\dots ,a_{n-i+1}$})=0.
\]
The above system of linear equations has a unique solution  $(b_n)_{n\geq 2} \in \cO_{\infty},$ by recursively finding  $b_n.$  
\\

For general $d$, suppose that \( \Phi:=(\phi_1(x_1,\dots x_d),\dots,\phi_d(x_1,\dots,x_d)) \in \cA_{\infty}(\TT^d) \).  By the definition of $\cA_{\infty}(\TT^d),$ we have $ J\Phi(\vec{0}) \in \GL_d(\cO_{\infty}).$ Let $\bar{\Psi}:=\left(J\Phi(\vec{0})\right)^{-1}\in \GL_d(\cO_{\infty}). $  We note that 
$J(\bar{\Psi} \circ \Phi)(\vec{0})= I_{d\times d}.$ Without loss of the generality, we assume that $J\Phi(\vec{0})= I_{d\times d}.$
We wish to find  \( \Psi:=(\psi_1(x_1,\dots x_d),\dots,\psi_d(x_1,\dots,x_d))\in \cA_{\infty}(\TT^d)  \) such that 
\[\psi_i\big(\phi_1(x_1,\dots,x_d), \dots,\phi_d(x_1,\dots,x_d) \big)= x_i\]
for every $1\leq i\leq d.$ Suppose that 
\begin{equation*}
\begin{split}
 \phi_i:=\sum_{(n_1,\dots,n_d)\in \mathbb{Z}_{\geq 0}^d}a_{i,(n_1,\dots,n_d)} x_1^{n_1}\dots x_d^{n_d},
\\
 \psi_i:=\sum_{(n_1,\dots,n_d)\in\mathbb{Z}_{\geq 0}^d}b_{i,(n_1,\dots,n_d)} x_1^{n_1}\dots x_d^{n_d},
 \end{split}
 \end{equation*}
 where $1\leq i\leq d.$ Let $|(n_1,\dots,n_d)|:= \sum_{i=1}^dn_i.$ For  $(n_1,\dots,n_d)\in \NNz^d,$ with $|(n_1,\dots,n_d)| \geq 2,$ we have 
 \begin{equation}
0= b_{i,(n_1,\dots,n_d)}+ \sum_{\substack{{(m_1,\dots,m_d)< (n_1,\dots,n_d)}}}   b_{i,(m_1,\dots,m_d)} (\text{some polynomial in $a_{j,(k_1,\dots,k_d)}$}),
 \end{equation}
 where $|(k_1,\dots,k_d)|\leq |(n_1,\dots,n_d)|$, and $(m_1,\dots,m_d)< (n_1,\dots,n_d)$ means that $m_i\leq n_i$ for $1\leq i\leq d$ and $(m_1,\dots,m_d)\neq (n_1,\dots,n_d).$ Similarly, the above system of recursive linear equations has a unique solution where $b_{i,(n_1,\dots,n_d)} \in \cO_{\infty}.$\\
\\
Finally, we check that $\Phi \in \cA_{\infty}(\TT^d)$  preserves the Haar measure on $\TT^d.$ By the definition of $\cA_{\infty}(\TT^d),$ we have $ |\det (J\Phi(\vec{0}))|=1.$ This implies that $|\det (J\Phi(\vec{x}))|=1$ for every $\vec{x}\in \TT^d.$  By Lemma~\ref{lem:change}, $\Phi$ preserves the Haar measure. This completes the proof of our proposition. 
\end{proof}
Next, we prove a version of the Morse lemma for functions in $C^{\omega}(\TT^d).$
\begin{proposition}[Morse lemma over $\Ki$]\label{prop:morse} Assume that $\phi\in C^{\omega}(\TT^d)$ with a  critical point at $0$. Suppose furthermore that the Hessian $H_{\phi}$ satisfies $|\det(H_{\phi}(\vec{0}))|=1.$  Then there exists $\Psi\in \cA_{\infty}(\TT^d)$ with $J\Psi(\vec{0})=I_{d\times d}$ such that  
\[
\phi(\vec{x})=\phi(\vec{0})+\frac{1}{2}\Psi(\vec{x})^{\intercal}H_{\phi}(\vec{0})\Psi(\vec{x})
\]
for every $\vec{x}\in \TT^d,$ where $\Psi(\vec{x})^{\intercal}$ is the transpose of the column vector $\Psi(\vec{x})$.
\end{proposition}
\begin{proof} We write $\vec{x}=(x_1,\hdots,x_d)$ in this proof. By Lemma~\ref{diag} there exists a matrix $g\in \GL_d(\cO_{\infty})$ such that $g^{\intercal} H_{\phi}(\vec{0})g=\text{diag}(\lambda_1,\dots,\lambda_d).$ Since $H_{\phi}(\vec{0}) \in \GL_d(\cO_{\infty}),$    $\lambda_i \in \cO_{\infty}$ and $|\lambda_i|=1.$ By changing the variables with $g$, we may assume without loss of generality that $H_{\phi}(\vec{0})$ is a diagonal matrix. We proceed by induction. Our induction hypothesis is that if for some $d'\leq d$
\[
 \phi(\vec{x})= \phi(\vec{0})+\sum_{1 \leq i\leq  j\leq d'}x_ix_j(\delta_{i,j}\frac{\lambda_i}{2}+h_{i,j}(\vec{x}))
\]
with   $h_{i,j}\in C^{\omega}(\TT^{d})$ and $h_{i,j}(\vec{0})=0,$  then
\begin{equation}\label{induc}
\phi(\vec{x})=\phi(\vec{0})+\sum_{1\leq j \leq d' }\frac{\lambda_j}{2} \psi_j^2(\vec{x}),
\end{equation}
where $\psi_{j}(\vec{x})=x_j+h_j(\vec{x})$ with $h_j(\vec{0})=0$ and $h_j(\vec{x})$ having a critical point at 0. We induct on $d'.$ First, we prove the induction hypothesis  for $d'=1.$  We have
\[
 \phi(\vec{x})=  \phi(\vec{0})+ x_1^2\left(\frac{\lambda_1}{2}+h_{1,1}(\vec{x})\right).
\]
Let 
\[
\psi_1(\vec{x}):=x_1\big(1+2\lambda_1^{-1}h_{1,1}\big)^{1/2}\in  \cA_{\infty}(\TT),
\]
where we use the Taylor expansion $(1+x)^{1/2}=\sum_{k=0}^{\infty}\binom{1/2}{k}x^k.$ It is easy to check that 
\[
\phi(\vec{x})= \phi(\vec{0})+\frac{\lambda_1}{2} \psi_1^2(\vec{x}).
\] This completes the proof of the induction hypothesis for $d'=1.$
 \\
\\

Suppose that the induction hypothesis holds for $d'-1.$ We show it for $d'.$ We rewrite~\eqref{induc} as
\begin{multline*}
 \phi(\vec{x})= \phi(\vec{0})+ x_{d'}^2(\frac{\lambda_{d'}}{2}+h_{d',d'}(\vec{x}))+ \sum_{1\leq j  \leq d'-1  }x_{d'}x_j h_{j,d'}(\vec{x}) 
\\
  +\sum_{1\leq i\leq j \leq d'-1 }x_ix_j(\delta_{i,j}\frac{\lambda_i}{2}+h_{i,j}(\vec{x})).
\end{multline*}
Define 
\[
\psi_{d'}:=x_{d'}\big(1+2\lambda_{d'}^{-1}h_{d',d'}(\vec{x})\big)^{1/2}+\big(\lambda_{d'}^{-1}\sum_{1 \leq j \leq d'-1 }x_j h_{j,d'}(\vec{x}) \big)\big(1+2\lambda_{d'}^{-1}h_{d',d'}(\vec{x})\big)^{-1/2}.
\] We have 
\begin{equation}
\begin{split}
\phi (\vec{x})&=\phi(\vec{0})+\frac{\lambda_{d'}}{2} \psi_{d'}(\vec{x})^2+\sum_{1\leq i,j\leq d'-1 }x_ix_j(\delta_{i,j}\lambda_i+h^{\prime}_{i,j}(\vec{x})),\\
\end{split}
\end{equation}
for some $h^{\prime}_{i,j}(\vec{x})\in C^{\omega}(\TT^d),$ where $h_{i,j}(\vec{0})=0.$ By the induction hypothesis for $d'-1$, we have   
\[
\phi(\vec{x})=\phi(\vec{0})+\sum_{1\leq j \leq d' }\frac{\lambda_j}{2} \psi_j^2(\vec{x}),
\]
where $\psi_{j}=x_j+h_j(\vec{x})$ with $h_j(\vec{0})=0$ and $h_j(\vec{x})$ having a critical point at 0. Taking $d'=d$ and $\Psi(\vec{x})=[\psi_1(\vec{x}),\hdots,\psi_d(\vec{x})]^{\intercal}$ concludes our proof. 
\end{proof}
\begin{remark}Proposition~\ref{prop:morse} is true for every $\vec{x}\in\TT^d$, and so it is a global statement in contrast to the usual Morse lemma. In particular, it shows that $\phi$ has $0$ as a single critical point in $\TT^d$. Consequently, only one critical point plays a role in the stationary phase theorem that we prove in Proposition~\ref{station}.
\end{remark}
\subsection{Stationary phase theorem over function fields}
In this section, we prove a version  of the stationary phase theorem in the function fields setting that we use for computing the oscillatory integrals $I_{g,r}(\vec{c})$. The proof is similar in spirit to that of the classical stationary phase theorem in real analysis. See~\cite[Prop. 6, p. 344]{Stein}.\\
\\
Let $f\in K_{\infty}$ and define  
\begin{equation}\label{epsfact}
\cG(f):=\begin{cases}\min(|f|^{-1/2},1)   &\text{ if } \deg(f)\text{ is } even,
\\
|f|^{-1/2} \varepsilon_{f} &\text{ if } \deg(f)\geq 1\text{ and is } odd,
\\
1 &\text{ otherwise,}
\end{cases}
\end{equation}
where $\varepsilon_{f}:=\frac{G(f)}{|G(f)|}$ and  $G(f):=\sum_{x\in \mathbb{F}_q} e_q(a_f x^2)$ is the Gauss sum associated to $a_f$, the top degree coefficient of $f.$ Suppose that $\phi\in C^{\omega}(\TT^d)$ has a critical point at $0$ with the Hessian $H_{\phi}(\vec{0})$, where $|\det(H_{\phi}(\vec{0}))|=1.$
\begin{proposition}\label{station}
Suppose the above assumptions on $\phi$ and $f.$ We have
\[
\int_{\TT^d}\psi(f\phi(\vec{u}))d\vec{u}=\psi(f \phi(\vec{0})) \prod_{i=1}^{d}\cG (f\lambda_i/2),
\]
where $\lambda_i \in \cO_{\infty}$ for $1 \leq i\leq d$ are diagonal elements of a diagonal matrix $g^{\intercal}H_{\phi}(\vec{0})g$ for some $g\in \GL_d(\cO_{\infty})$ ($H_{\phi}(\vec{0})$ is diagonalizable in such a way by Lemma~\ref{diag}).
\end{proposition}
We first prove a special case of Proposition~\ref{station} for quadratic polynomials. 

\subsubsection{Gaussian integrals over function field }
We explicitly compute the analogue of Gaussian integrals over $\Ki$. 
\begin{lemma}\label{ggg} For every $f\in K_{\infty},$
we have 
\[
\int_{\TT} \psi(fu^2) du=\cG(f).
\]

\end{lemma}
\begin{proof}First, suppose that $\deg(f)=2k,$ where $k\geq 0.$ We partition $\TT$ into the cosets of $ t^{-k} \TT.$ Let $\alpha + t^{-k} \TT \subset \TT.$ We show that 
\[
\int_{\alpha + t^{-k} \TT} \psi(fu^2) du= 0
\]
for $\alpha \notin t^{-k} \TT.$ We have 
\begin{equation*}
\begin{split}
\int_{\alpha + t^{-k} \TT} \psi(fu^2) du= \int_{t^{-k} \TT} \psi(f(\alpha +v)^2) dv
&=\psi(f\alpha^2 )\int_{t^{-k} \TT} \psi(f(2\alpha v +v^2)) dv
\\
&=\psi(f\alpha^2 )\int_{t^{-k} \TT} \psi(2f\alpha v ) dv=0,
\end{split}
\end{equation*}
where we used Lemma~\ref{lem:orthog}, $\deg(fv^2)\leq -2$ and $\deg(\alpha f )\geq k.$ Therefore, 
\[
\int_{\TT} \psi(fu^2) du=\int_{ t^{-k} \TT} \psi(fu^2) du=\int_{ t^{-k} \TT}  du =|f|^{-1/2}=\mathcal{G}(f).
\]
On the other hand, suppose $\deg(f)=2k-1,$ where $k\geq 1.$ If $\alpha \notin t^{-k+1} \TT$ 
\begin{equation*}
\begin{split}
\int_{\alpha + t^{-k} \TT} \psi(fu^2) du= \int_{v\in t^{-k} \TT} \psi(f(\alpha +v)^2) du
&=\psi(f\alpha^2 )\int_{ t^{-k} \TT} \psi(f(2\alpha v +v^2)) dv
\\
&=\psi(f\alpha^2 )\int_{t^{-k} \TT} \psi(2f\alpha v ) dv=0,
\end{split}
\end{equation*}
where we used Lemma~\ref{lem:orthog}, $\deg(fv^2)\leq -3$ and $\deg(\alpha f )\geq k.$ Hence, it suffices to compute the integral over $t^{-k+1}\TT+t^{-k}\TT=t^{-k+1}\TT$:
\[
\int_{\TT} \psi(fu^2) du=\int_{ t^{-k+1}\TT} \psi(fu^2) du.
\]
The last integral is computed as follows. By definition,
\begin{eqnarray*}
\int_{t^{-k+1}\TT} \psi(fu^2) du&=& \lim_{m\rightarrow +\infty}q^{-m-k+1}\sum_{a_{-m}t^{-m-k+1}+\hdots+a_{-1}t^{-k}:a_i\in\mathbb{F}_q}\psi((a_{-m}t^{-m-k+1}+\hdots+a_{-1}t^{-k})^2f)\\&=& \lim_{m\rightarrow +\infty}q^{-m-k+1}\sum_{a_{-m},\hdots,a_{-1}\in\mathbb{F}_q}e_q(a_fa_{-1}^2)\\ &=& q^{-k}\sum_{x\in\mathbb{F}_q}e_q(a_fx^2)=q^{-k}G(f)=\mathcal{G}(f).
\end{eqnarray*}
It is well-known that $G(f)=q^{1/2}\varepsilon_f$. Consequently, $q^{-k}G(f)=|f|^{-1/2}\varepsilon_f$. We have therefore proved the result for $\deg(f)=2k-1$, $k\geq 1$. \\
\\
Finally, if $\deg(f)\leq -1$, then $\deg(fu^2)\leq -3$ for $u\in\mathbb{T}$. Consequently,
\[\int_{\TT}\psi(fu^2)du=\int_{\TT}du=1.\] 
This concludes the proof.
\end{proof}
Next, we give a formula for the Gaussian integral associated to any symmetric matrix $A \in M_{d\times d}(K_{\infty}).$ Define
\[
\cG(A):=\int_{\TT^d}\psi(\vec{u}^{\intercal}A\vec{u}).
\]

\begin{lemma}\label{gauint}
We have 
\[
\cG(A)=\prod_{i=1}^d \cG(\lambda_i),
\]
where $\lambda_i \in K_{\infty}$ for $1 \leq i\leq d$ are diagonal elements of a diagonal matrix $g^{\intercal}Ag$ for some $g\in \GL_d(\cO_{\infty})$ ($A$ is diagonalizable in such a way by Lemma~\ref{diag}). \end{lemma}
\begin{proof}
By Lemma~\ref{diag}, there exists $g\in \GL_d(\cO_{\infty}) $ such that $g^{\intercal}Ag=\text{diag}(\lambda_1,\dots,\lambda_d).$  By the change of variables formula in Lemma~\ref{lem:change}, we have 
\begin{equation*}
\begin{split}
\cG(A)=\int_{\TT^d}\psi(\vec{u}^{\intercal}A\vec{u}) d\vec{u}&=\int_{\TT^d}\psi\left((g^{-1}\vec{u} )^{\intercal} g^{\intercal}Ag (g^{-1}\vec{u})\right) d\vec{u}
\\
&=\int_{\TT^d}\psi\left( \sum_{i=1}^d \lambda_i v_i^2 \right) d\vec{v}=\prod_{i=1}^d \cG(\lambda_i),
\end{split}
\end{equation*}
where $\begin{bmatrix}v_1 \hdots v_d \end{bmatrix}=\vec{v}=g^{-1}\vec{u}.$ This completes the proof of the lemma. 
\end{proof}

Finally, we give a proof of Proposition~\ref{station}.
\begin{proof}[Proof of Proposition~\ref{station} ]
By Proposition~\ref{prop:morse}, there exists $\Psi \in \cA_{\infty}(\TT^d)$ such that $\phi=\phi(\vec{0})+\frac{1}{2}\Psi^{\intercal}H_{\phi}(\vec{0})\Psi.$ Hence,
\[
\int_{\TT^d}\psi(f\phi(\vec{u}))d\vec{u}=\int_{\TT^d}\psi(f(\phi(\vec{0})+\frac{1}{2}\Psi(\vec{u})^{\intercal}H_{\phi}(\vec{0})\Psi(\vec{u})))d\vec{u}.
\]
Since $\Psi\in \cA_{\infty}(\TT^d)$, Proposition~\ref{auto} implies that $\Psi$ is a measure-preserving automorphism of $\TT^d,$ and so we have the following equality of volume measures: $d\Psi(\vec{u})=d\vec{u}$. By Lemma~\ref{gauint}, 
\[
\int_{\TT^d}\psi(f(\phi(\vec{0})+\frac{1}{2}\Psi^{\intercal}H_{\phi}(\vec{0})\Psi))d\Psi=\psi(f \phi(\vec{0})) \prod_{i=1}^{d}\cG (f\lambda_i/2),
\]
where $\lambda_i \in \cO_{\infty}$, $1 \leq i\leq d$, are diagonal elements of a diagonal $g^{\intercal}H_{\phi}(\vec{0})g$ for some $g\in \GL_d(\cO_{\infty})$ obtained using Lemma~\ref{diag}. This concludes the proof of our proposition.
\end{proof}
\section{Bounds on the oscillatory integrals $I_{g,r}(\vec{c})$}\label{osil}
 
In this section, we give explicit formulas for the oscillatory integrals $I_{g,r}(\vec{c})$ in terms of Kloosterman or Sali\'e sums. By Lemma~\ref{diag}, we suppose  that $F(\gamma\vec{u})=\sum_{\eta_i} \eta_iu_i^2,$ where $\gamma\in \GL_d(\cO_{\infty}).$
 Recall the additive character $\psi:K_\infty\to \CC^*$ from 
\S \ref{s:add-characters}, and 
\[h(x,y)=\begin{cases}|x|^{-1} & \mbox{if $|y|<|x|$}\\ 0 & \mbox{otherwise.}\end{cases}\]
\subsection{Weight function}
In this section, we define the test function $w$ that we use for estimating the oscillatory integrals $I_{g,r}(\vec{c})$ at the end of this section. Recall Definition~\ref{def:aniso} of an anisotropic cone.
\begin{definition}[Anisotropic cone]
$\Omega\subset K_{\infty}^d$ is an \textit{anisotropic cone} with respect to the quadratic form $F(\vec{x})$ if there exist fixed positive integers $\omega$  and  $\omega^{\prime}$ such that: 
\begin{enumerate}
\item \label{prop1} if $\vec{x}\in \Omega$ then $f\vec{x}\in \Omega$ for every $f\in K_\infty;$
\item \label{prop2} if $\vec{x}\in\Omega$ and $\vec{y} \in  K_\infty^d $ with $|\vec{y}|\leq |\vec{x}|/ \hat{\omega},$ then $\vec{x}+\vec{y} \in \Omega;$ and 
\item \label{prop3}  if $\vec{x}\in\Omega,$  $ \hat{\omega^{\prime}} |F(\vec{x})| \geq |\vec{x}|^2.$
\end{enumerate}
\end{definition}

\begin{lemma}\label{cone}
Let $F(\vec{x})$ be a non-degenerate quadratic form in $d\geq 4$ variables. There is an anisotropic cone $\Omega$ in $\Ki^d$ such that for any given $f$, $X_f(\Ki)\cap\Omega\neq\emptyset$.
\end{lemma}
\begin{proof}
We first make some reductions. Since being an anisotropic cone is preserved by $\Ki$-linear change of coordinates, we may assume without loss of generality that $F$ is a diagonal form and the coefficients of $F$ are among the representatives in $A:=\{1,\nu,t,\nu t\}$ of the classes in $\Ki^{\times}/\Ki^{\times 2}$, where $\nu$ is a quadratic non-residue of $\mathbb{F}_q$. Furthermore, we may assume without loss of generality that $f$ is one of the representatives in $A$. It is enough to show that for every $f\in A$, there is an anisotropic cone $\Omega_f$ intersecting $X_f(\Ki)$. We show this for $f=\nu t$; the other cases will follow by scalings. The Lemma would then follow by taking 
\[\Omega:=\bigcup_{f\in A}\Omega_f.\]
Suppose the class of $f$ is $\nu t$. Consider the set $\Omega_{\nu t}:=\{\vec{x}\in\Ki^d:|F(\vec{x})|\geq \frac{1}{q^2}|\vec{x}|^2\}$. This is an anisotropic cone with $\omega=4$ and $\omega'=2$. This satisfies properties (1) and (3) of Definition~\ref{def:aniso}. We check property (2). Suppose $\vec{x}\in\Omega_{\nu t}$ and $\vec{y}\in \Ki^d$ such that $|\vec{y}|\leq |\vec{x}|/q^4$. We must show that $\vec{x}+\vec{y}\in\Omega_{\nu t}$. Note that our norm is non-archimedean, and so $|\vec{x}+\vec{y}|=|\vec{x}|$. Furthermore, writing $F(\vec{x})=\vec{x}^{\intercal}A\vec{x}$, where $A$ is a diagonal matrix with diagonal entries among $1,\nu,t,\nu t$, we have $F(\vec{x}+\vec{y})=F(\vec{x})+F(\vec{y})+2\vec{x}^{\intercal}A\vec{y}$. From $|\vec{y}|\leq |\vec{x}|/q^4$, we have $|F(\vec{y})|\leq |\vec{x}|/q^7$. Furthermore, $|\vec{x}^{\intercal}A\vec{y}|\leq |\vec{x}|^2/q^3$. Therefore, $|F(\vec{x}+\vec{y})|=|F(\vec{x})|$ as well. This completes the proof that $\Omega_{\nu t}$ is an anisotropic cone.\\
\\
If one of the coefficients of $F$ is $\nu t$, then we have a solution in $\Omega_{\nu t}$. Otherwise, the coefficients are among $1,\nu,t$ and so at least two of the coefficients are equal since $d\geq 4$. If $-1=a^2$ is a square, then every element of $\Ki$ can be written as the sum of two squares. Indeed, every element of a finite field is a sum of two squares, and so $\nu$ is a sum of two squares. Furthermore, we also have
\[t=\left(t\left(1+\frac{1}{t}\right)^{1/2}\right)^2+(at)^2,\]
\[\nu t=\left(t\left(1+\frac{\nu}{t}\right)^{1/2}\right)^2+(at)^2.\]
Since at least one coefficient repeats, this implies that we can represent any element by the quadratic form in this case.\\
\\
On the other hand, $-1$ may be a quadratic non-residue, in which case we may assume $\nu=-1$. If both $1$ and $-1$ show up as coefficients, we may represent any element of $\Ki$, as can be shown as above. Therefore, let us assume otherwise. We are reduced to showing that there is a solution in the anisotropic cone to the equations $t(x_1^2+x_2^2+x_3^2+1)=\pm x_4^2$, $t(x_1^2+x_2^2+1)=\pm (x_3^2+x_4^2)$, $t(x_1^2+1)=\pm (x_2^2+x_3^2+x_4^2)$, and $x_1^2+\hdots+x_4^2=\pm t$ for any choice of signs.\\
\\
$x_1^2+x_2^2+1=0$ is solvable modulo any odd prime, and so the first and second equations have a solution in the anisotropic cone. Take $a,b\in\mathbb{F}_q^{\times}$ such that $a^2+b^2=-1$ (since $-1$ is a quadratic non-residue, $ab\neq 0$). For the third equation, let $(x_1,x_2,x_3,x_4)=\left(0,at\left(1\mp\frac{1}{t}\right)^{1/2},bt\left(1\mp\frac{1}{t}\right)^{1/2},t\right)$. Note that such squareroots exist in $\Ki$ because $q$ is odd (see the beginning of Section~\ref{analyticsection} for the formula). For the final equation $\pm t=x_1^2+\hdots+x_4^2$ let $(x_1,x_2,x_3,x_4)=\left(at\left(1\mp\frac{1}{t}\right)^{1/2},bt\left(1\mp\frac{1}{t}\right)^{1/2},t,0\right)$.\\
\\
The other classes can be dealt with similarly; at the beginning, you can multiply the quadratic form by $\nu$ or $t$ and scale the coordinates to reduce it to the above case that $f$ has class $\nu t$. The classes $\nu,t,\nu t$ have norm at most $q$. It then follows that we can take
\[\Omega:=\{\vec{x}\in\Ki^d:|F(\vec{x})|\geq \frac{1}{q^3}|\vec{x}|^2\}.\]
\end{proof}
Fix an anisotropic cone $\Omega$  with respect to $F(\vec{x})$ (such that $\Omega\cap X_f(\Ki)\neq\emptyset$). 
\begin{lemma}\label{ultra}
Suppose that $\vec{x}\in \Omega$ and $\vec{y}\notin \Omega.$ Then
\[
|\vec{x}\pm \vec{y}|\geq \max{(|\vec{x}|,|\vec{y}|)}/\hat{\omega},
\]
where $\omega$ appears in property~\eqref{prop2} of the definition of anisotropic cone $\Omega.$ 
\end{lemma}
\begin{proof}
If $|\vec{x}|\neq |\vec{y}|$, then $|\vec{x}\pm \vec{y}|=\max(|\vec{x}|,|\vec{y}|)$ and the lemma follows. Otherwise, $|\vec{x}|=|\vec{y}|$ and 
the lemma follows from property~\eqref{prop2} as $\vec{y}\notin \Omega$.  
\end{proof}
If $A=[a_{ij}]$ is a matrix with $\Ki$ coefficients, its norm $|A|:=\max|a_{ij}|$. 
\begin{lemma}\label{opnorm}
For matrices $A,B$ with $\Ki$ coefficients, $|AB|\leq |A||B|.$ If $g\in\GL_d(\mathcal{O}_{\infty})$ and $C\in \textup{M}_{d\times d'}(\Ki)$, then $|gC|=|C|.$
\end{lemma}
\begin{proof}$|AB|\leq |A||B|$ follows from the fact that the norm on $\Ki$ is non-archimedean. For the second part, first note that if $g\in\GL_d(\mathcal{O}_{\infty})$, then $|g|\leq 1$ and $|g^{-1}|\leq 1$. Furthermore, $|g||g^{-1}|\geq |gg^{-1}|=1$. Consequently, $|g|=1$. If $g\in\GL_d(\mathcal{O}_{\infty})$ and $C\in \textup{M}_{d\times d'}(\Ki)$, then $|gC|\leq |C|$. On the other hand, $|C|=|g^{-1}(gC)|\leq |g^{-1}||gC|=|gC|$. This completes the proof.
\end{proof}
For non-degenerate quadratic form $F(\vec{x})=\vec{x}^{\intercal}A\vec{x},$ we say $F^{*}(\vec{x})=\vec{x}^{\intercal}A^{-1}\vec{x}$ is the dual of $F(\vec{x}).$ Note that $F(\vec{x})=F^*(A\vec{x}).$ Let $\Omega^*:=A\Omega.$ By Lemma~\ref{diag}, choose $g\in\GL_d(\mathcal{O}_{\infty})$ such that 
\[
g^{\intercal}Ag=\text{diag}(\eta_1,\dots,\eta_d)
\]
is a diagonal matrix with $\eta_i\in\Ki$. A consequence of Lemma~\ref{opnorm} is that if $g\in\GL_d(\mathcal{O}_{\infty})$, then $\max|\eta_i|=|g^{\intercal}Ag|=|A|$ only depends on $A$. Similarly, $|A^{-1}|=|g^{-1}A^{-1}(g^{\intercal})^{-1}|=\max|\eta_i|^{-1}=\frac{1}{\min|\eta_i|}$ implies that $\min|\eta_i|$ only depends on $A$.
\begin{example}
Let $A:=\begin{bmatrix}1+t^2& t\\ t& 1\end{bmatrix}$ and  $g=\begin{bmatrix}1& -\frac{t^{-1}}{1+t^{-2}}\\ 0& 1\end{bmatrix}\in\GL_2(\mathcal{O}_{\infty})$. Then $g^{\intercal}Ag=\diag(t^2+1,\frac{t^{-2}}{1+t^{-2}})$. In this case, $\min\deg\eta_i=-2$ and $\max\deg\eta_i=2$.
\end{example}
\begin{lemma}\label{dualcone}
  $\Omega^*$ is an anisotropic cone with respect to $F^*$ with parameter $\omega^*:=\omega+\max\deg(\eta_i)-\min\deg(\eta_i)$ in property~\eqref{prop2} and parameter $\omega'^*:=\omega'+2\max\deg\eta_i$ in property~\eqref{prop3}. 
\end{lemma}
\begin{proof}
Let $\Omega$ have parameters $\omega$ and $\omega'$. Note that if $\vec{x}^*\in\Omega^*$, then $\vec{x}^*=A\vec{x}$ for some $\vec{x}\in\Omega$. If $f\in\Ki$, then $f\vec{x}^*=A(f\vec{x})\in\Omega^*$ as $f\vec{x}\in\Omega$ since $\Omega$ is an anisotropic cone. This verifies property~\eqref{prop1} of anisotropic cones for $\Omega^*.$ 
We verify property~\eqref{prop2} with parameter $\omega^*$. Assuming  
\(|\vec{y}^*|\leq |\vec{x}^*|/\hat{\omega^*},\)  $\vec{x}^*\in\Omega^*$ and $\vec{y}^*\in\Ki^d$, we have 
\[
|A^{-1}\vec{y}^*|\leq \frac{|\vec{y}^*|}{\min|\eta_i|}\leq\frac{|\vec{x}^*|}{\hat{\omega^*}\min|\eta_i|}=\frac{|A(A^{-1}\vec{x}^*)|}{\hat{\omega^*}\min|\eta_i|}\leq\frac{\max|\eta_i|}{\min|\eta_i|\hat{\omega^*}}|A^{-1}\vec{x}^*|= |A^{-1}\vec{x}^*|/\hat{\omega}.
\]
Since $A^{-1}\vec{x}^*\in\Omega$ and $|A^{-1}\vec{y}^*|\leq |A^{-1}\vec{x}^*|/\hat{\omega}$, property~\eqref{prop2} of anisotropic cones for $\Omega$ implies that 

 $A^{-1}(\vec{x}^*+\vec{y}^*)\in\Omega$, and so $\vec{x}^*+\vec{y}^*\in\Omega^*$. Finally, we verify property~\eqref{prop3} for parameter $\omega'^*.$ Suppose $\vec{x}^*=A\vec{x}$, $\vec{x}\in\Omega$. Then 
\[\hat{\omega'^*}|F^*(\vec{x}^*)|=\hat{\omega'^*}|F^*(A\vec{x})|=\hat{\omega'^*}|F(\vec{x})|\geq \max|\eta_i|^2|\vec{x}|^2= |A|^2|A^{-1}\vec{x}^*|^2\geq |\vec{x}^*|^2.\]
\end{proof}
Recall the assumption of Theorem~\ref{mainthm} that $\Omega$ is an anisotropic cone such that $X_{f}(\Ki)\cap\Omega\neq\emptyset$. Choose $\vec{x}_0\in X_{f}(\Ki)\cap\Omega$, and let $w$ be the characteristic function of a ball centered at $\vec{x}_0$ defined as follows:
\begin{equation}\label{defw}
w(\vec{x}):=\begin{cases}
1 \text{ if } |\vec{x}-\vec{x}_0| < |t^{-\alpha_0}f|^{1/2},
\\
0 \text{ otherwise,}
\end{cases}
\end{equation}
where 
\begin{equation}\label{alpha0}
\alpha_0:=2\omega+2\max{\deg(\eta_i)}+3|\min\deg(\eta_i)|+\omega'.
\end{equation}

\begin{lemma}
If $w(\vec{x})\neq 0,$  then $\vec{x}\in \Omega.$  
\end{lemma}
\begin{proof}If $w(\vec{x})\neq 0$, then $|\vec{x}-\vec{x}_0|<|t^{-\alpha_0}f|^{1/2}$. Since $\vec{x}_0\in X_f(\Ki)\cap\Omega$ and $\Omega$ is an anisotropic cone with parameters $\omega$ and $\omega'$, by property~\eqref{prop2} of anisotropic cones, it suffices to have $|t^{-\alpha_0}f|^{1/2}\leq |\vec{x}_0|/\hat{\omega}$. In fact, $f=F(\vec{x}_0)$, and so $|t^{-\alpha_0}f|^{1/2}=|F(\vec{x}_0)|^{1/2}/\hat{\alpha_0}^{1/2}\leq\max|\eta_i|^{1/2}|\vec{x}_0|/\hat{\alpha_0}^{1/2}$. Therefore, any $\alpha_0\geq 2\omega+\max\deg(\eta_i)$ works. Note that since $F$ is a quadratic form with coefficients in $\mathcal{O}$, $|A|\geq 1$ and so $\max\deg(\eta_i)\geq 0$. Our choice of $\alpha_0$ satisfies this inequality. The conclusion follows.
\end{proof}
\begin{lemma}
$$\left\{\vec{y}\in K_\infty^d: |\vec{y}-A\vec{x}_0|<  |t^{-\alpha_0}f|^{1/2} \right\}  \subset  \Omega^*$$
\end{lemma}
\begin{proof}
Since $\vec{x}_0\in X_f(\Ki)$, $F^*(A\vec{x}_0)=F(\vec{x}_0)=f$. Therefore, $|\vec{y}-A\vec{x}_0|<|t^{-\alpha_0}f|^{1/2}$ implies that
\begin{equation}\label{kop}|\vec{y}-A\vec{x}_0|<|t^{-\alpha_0}F^*(A\vec{x}_0)|^{1/2}\leq \frac{1}{\min(|\eta_i|)^{1/2}}|A\vec{x}_0|/\hat{\alpha_0}^{1/2}.\end{equation}
Note that $A\vec{x}_0\in\Omega^*$. By Lemma~\ref{dualcone}, $\Omega^*$ is an anisotropic cone with parameter $\omega^*:=\omega+\max\deg(\eta_i)-\min\deg(\eta_i)$ in property~\eqref{prop2}. Hence, $\vec{y}\in\Omega^*$ if 
\[|\vec{y}-A\vec{x}_0|\leq |A\vec{x}_0|/\hat{\omega^*}.\]
By~\eqref{kop}, this is satisfied if
\[\frac{1}{\min(|\eta_i|)^{1/2}}|A\vec{x}_0|/\hat{\alpha_0}^{1/2}\leq |A\vec{x}_0|/\hat{\omega^*},\]
that is, if $\hat{\alpha_0}^{1/2}\geq \frac{\hat{\omega^*}}{\min(|\eta_i|)^{1/2}}$. Therefore, any $\alpha_0\geq 2\omega+2\max\deg(\eta_i)-3\min\deg(\eta_i)$ works. Our choice of $\alpha_0$ satisfies this inequality. 
\end{proof}
Write $\vec{x}_0=g\vec{t}_0+\boldsymbol{\lambda}$ and $\vec{x}=g\vec{t}+\boldsymbol{\lambda}$, where $\vec{t}_0, \vec{t}\in\Ki^d$. Substituting these in \eqref{defw}, we obtain 
\[
w(g\vec{t}+\boldsymbol{\lambda})=\begin{cases}
1 \text{ if } |\vec{t}-\vec{t}_0|  < \hR,
\\
0 \text{ otherwise,}
\end{cases}
\]
where 
\begin{equation}\label{R}R:= \lceil{\deg(f)/2} -deg(g)-\alpha_0/2\rceil.\end{equation}
By the assumptions of Theorem~\ref{mainthm}, $\deg f \geq 4\deg g +O(1)$, and so we may assume $R>0.$
\begin{lemma}\label{t0}
Assume that $R>0.$ Then $\vec{x}_0\in \Omega$ is equivalent to $\vec{t}_0\in \Omega.$
\end{lemma}
\begin{proof} Suppose that  $\vec{x}_0\in \Omega.$ Since $\Omega$ is an anisotropic cone, 
\(
\vec{t}_0+\frac{\boldsymbol{\lambda}}{g}= \frac{1}{g}\vec{x}_0 \in \Omega
\). By property~\eqref{prop2}, if $\frac{|\vec{x}_0|}{|g|\hat{\omega}} \geq |\frac{\boldsymbol{\lambda}}{g}|$, then $\vec{t}_0\in \Omega.$  We have 
\[
\frac{|\vec{x}_0|}{|g|\hat{\omega}}  \geq  \left( \frac{|\vec{x}_0^{\intercal}A\vec{x}_0|}{|A|}\right)^{1/2}\frac{1}{|g|\hat{\omega}} \geq
\frac{|f|^{1/2}}{|g||A|^{1/2}\hat{\omega}}\geq \hat{R}\geq 1.
\]
On the other hand, we have $ |\frac{\boldsymbol{\lambda}}{g}| \leq 1.$ This concludes one side of the lemma, and the converse  follows similarly. 

\end{proof}

\subsection{Bounding $I_{g,r}(\vec{c})$}
Recall the notations that we introduced in Section~\ref{deltamethod}.  Let
\begin{equation}\label{defQ}
Q:=  \lceil{\deg(f)/2} -deg(g)\rceil +\max_{i}(\deg(\eta_i))+\omega^{\prime}
\end{equation} in equation~\eqref{newequu}
\[N(w,\boldsymbol{\lambda})=\frac{1}{|g|\hQ^2}\sum_{\substack{
r\in \cO\\
|r|\leq \hat Q\\
\text{$r$ monic}
} }\sum_{\vec{c}\in\mathcal{O}^d}|gr|^{-d}S_{g,r}(\vec{c})I_{g,r}(\vec{c}).\]

  We have  
\begin{equation}\label{intt}
I_{g,r}(\vec{c})= \int_{\Ki^d} h\left(\frac{r}{t^Q},\frac{G(\vec{t})}{t^{2Q}}\right) w(g\vec{t}+\boldsymbol{\lambda}) \psi\left(\frac{\left< \vec{c},\vec{t} \right>}{gr}    \right) d\vec{t}= \int_{\substack{|\vec{t}-\vec{t}_0| < \hR \\ |G(\vec{t})|< \hQ|r|}} \frac{\hQ}{|r|}  \psi\left(\frac{\left< \vec{c},\vec{t} \right>}{gr}    \right) d\vec{t}
\end{equation}
where $G(\vec{t}):=\frac{F(g\vec{t}+\boldsymbol{\lambda})-f}{g^2}=F(\vec{t})+\frac{1}{g}(2\boldsymbol{\lambda}^{\intercal} A\vec{t}-k)$ with $k=\frac{f-F(\boldsymbol{\lambda})}{g}$ defined in equation~\eqref{defk}.
 Let  \begin{equation}
 \label{kappadef}
 \kappa:=\left|\frac{\vec{c}}{g }\right|.
 \end{equation}
\begin{lemma}\label{noosc}
Suppose that $\kappa< \frac{|r|}{\hR},$ then  $I_{g,r}(\vec{c})= \psi\left(\frac{\left< \vec{c},\vec{t}_0 \right>}{gr}    \right) I_{g,r}(\vec{0}).$
\end{lemma}
\begin{proof}
Since $\max_{i}(|c_i|)< \frac{|gr|}{\hR}$ and $|\vec{t}-\vec{t}_0|<\hR$,  $ \psi\left(\frac{\left< \vec{c},\vec{t} \right>}{gr}    \right)=\psi\left(\frac{\left< \vec{c},\vec{t}_0 \right>}{gr}    \right).$ 
Hence, we have 
\[
I_{g,r}(\vec{c})=\psi\left(\frac{\left< \vec{c},\vec{t}_0 \right>}{gr}    \right) \int_{\substack{|\vec{t}-\vec{t}_0| < \hR \\ |G(\vec{t})|< \hQ|r|}} \frac{\hQ}{|r|}   d\vec{t}= \psi\left(\frac{\left< \vec{c},\vec{t}_0 \right>}{gr}    \right) I_{g,r}(\vec{0}).
\]
This completes the proof of our lemma. 
\end{proof}
\begin{lemma}\label{auxlem}
Let $Q, R$ and $\vec{t}_0$ be as above, and let $k=\frac{f-F(\boldsymbol{\lambda})}{g}$ as before. Suppose that $|\vec{t}-\vec{t}_0|< \hR.$ Then $|G(\vec{t})| < \hQ|r|$ is equivalent to $|F(\vec{t})-k/g|<\hQ |r|.$ Moreover, if $|G(\vec{t})|< \hQ|r|$, then $|G(\vec{t}+\boldsymbol{\zeta})|  < \hQ|r|$ for every $\boldsymbol{\zeta}\in K_\infty^d,$ where $|\boldsymbol{\zeta}| \leq \min(|r|,\hR).$
\end{lemma}
\begin{proof}
Since $\vec{x}_0\in \Omega\cap X_f(\ki),$ by property~\eqref{prop3} of anisotropic cones for $\Omega$, $|\vec{t}_0+\frac{\boldsymbol{\lambda}}{g}|\leq |f|^{1/2} \hat{\omega^{\prime}}^{1/2}/|g|.$ We may assume that $\left|\frac{\boldsymbol{\lambda}}{g} \right|< 1,$ which implies    $|\vec{t}_0| \leq |f|^{1/2} \hat{\omega^{\prime}}^{1/2}/|g|.$ 
 From this   we obtain   $|\frac{1}{g}(2\boldsymbol{\lambda}^{\intercal} A\vec{t}_0)| <\hQ.$ For $|\vec{t}-\vec{t}_0|<\hR,$ $|\vec{t}|\leq \hQ$ and
 \[
 \left|\frac{2}{g}\boldsymbol{\lambda}^{\intercal} A\vec{t}\right| \leq \max\left(\left |\frac{2}{g}\boldsymbol{\lambda}^{\intercal} A\vec{t_0}\right|,    \left|\frac{2}{g}\boldsymbol{\lambda}^{\intercal} A(\vec{t}-\vec{t_0})\right| \right)<\hQ,
 \]
 where we used $|A|\hR \leq \hQ.$
  Hence, $|G(\vec{t})| < \hQ|r|$ is equivalent to $|F(\vec{t})-k/g|<\hQ |r|.$ Moreover, suppose that $|\boldsymbol{\zeta}| \leq \min(|r|,\hR).$  Then 
\[
|G(\vec{t}+\boldsymbol{\zeta})-G(\vec{t})| \leq \max\left( |F(\boldsymbol{\zeta})| ,|\boldsymbol{\zeta}^{\intercal}A(\vec{t}+\boldsymbol{\lambda}/g)|\right) \leq \max(|\boldsymbol{\zeta}^{\intercal}A\boldsymbol{\zeta}|,\hQ|\boldsymbol{\zeta}|)\leq \hQ |r|,
\]

where we used  $\frac{|\boldsymbol{\lambda}|}{|g|}<1 ,$ $|A\boldsymbol\zeta|\leq |A|\hR\leq \hQ.$ Hence, if $|G(\vec{t})|< \hQ|r|$, then
  \begin{equation}\label{invbal}|G(\vec{t}+\boldsymbol{\zeta})| \leq \max(|G(\vec{t})|, |G(\vec{t}+\boldsymbol{\zeta})-G(\vec{t})| )  < \hQ|r|.
  \end{equation}
This concludes the proof of our lemma.
\end{proof}
We say  $\vec{c}$ is an \textit{ordinary vector} if 
\begin{eqnarray}\label{typeI}
\kappa \geq \hQ/\hR.
\end{eqnarray}
\begin{lemma}\label{o}
 
Suppose that $\vec{c}$ is an ordinary vector and $|r| \leq\hQ$. Then, \begin{equation}\label{type1}
I_{g,r}(\vec{c})=0.
\end{equation}
\end{lemma}
\begin{proof}
By \eqref{intt} and  \eqref{invbal}, we have \begin{equation*}
\begin{split}
I_{g,r}(\vec{c})&=\int_{\substack{|\vec{t}| < \hR \\ |G(\vec{t})|< \hQ|r|}} \frac{\hQ}{|r|}  \psi\left(\frac{\left< \vec{c},\vec{t} \right>}{gr}    \right) d\vec{t}= \frac{\hQ}{|r|} \int_{\substack{|\vec{t}| < \hR \\ |G(\vec{t})|< \hQ|r|}} \frac{1}{\min(|r|,\hR)^d} \int_{|\boldsymbol\zeta|<  \min(|r|,\hR)}  \psi\left(\frac{\left< \vec{c},\vec{t}+\boldsymbol{\zeta} \right>}{gr}    \right) d\boldsymbol\zeta d\vec{t}.
\end{split}
\end{equation*}
By Lemma~\ref{lem:orthog},
\(
 \int_{|\boldsymbol\zeta|< \min(|r|,\hR)}  \psi\left(\frac{\left< \vec{c},\boldsymbol{\zeta} \right>}{gr}    \right) d\boldsymbol\zeta=0
\) 
if $\left|\frac{\vec{c}}{gr}\right|\geq\frac{1}{\min(|r|,\hR)}$. We consider two cases. If $|r|\leq\hR$, then this is equivalent to $|\vec{c}|\geq |g|$. Since $\vec{c}$ is ordinary, $|\vec{c}|\geq |g|\hQ/\hR$. Since $\hQ\geq \hR$, we obtain the required inequality $|\vec{c}|\geq |g|$. On the other hand, if $|r|>\hR$, then we need to have $|\vec{c}|\geq\frac{|gr|}{\hR}$. However, since $|r|\leq\hQ$ and $\vec{c}$ is ordinary, $|\vec{c}|\geq |g|\hQ/\hR\geq \frac{|gr|}{\hR}$. This concludes the lemma. 
\end{proof}
We say $\vec{c}\neq 0$ is an \textit{exceptional vector} if  
\begin{equation}\label{typeII}\kappa<  \hQ/\hR.\end{equation}

\begin{proposition}\label{c}
Suppose that $\vec{c}$ is an exceptional vector and $\kappa\geq  \eta \frac{|r|}{\hR}$ and $d\geq 4$, where   $\eta> \max(\hat{\omega^*}, \hat{\omega'^*})$ is a fixed large enough constant integer.  For $\vec{c}\in \Omega^*$, we have  
 \begin{equation}
\label{excep}
|I_{g,r}(\vec{c})|\ll_{F,\Omega}  \hQ^d\Big(\frac{|\vec{c}|\hQ}{|gr|}   \Big)^{-\frac{d-1}{2}}.
\end{equation}
Otherwise, $\vec{c}\notin \Omega^*$  and $I_{g,r}(\vec{c})=0.$
\end{proposition}

We give the proof of the above proposition after proving an auxiliary lemma. For $\alpha \in K_{\infty}$ with $|\alpha|^{1/2}={\hl},$ define 
\[
\Kl_{\infty}(\alpha,\psi):=\int_{|x|=\hl} \psi\left(\frac{\alpha}{ x}+x\right) dx,
\]
and 
\[
\Sa_{\infty}(\alpha,\psi):=\int_{|x|=\hl} \varepsilon_{x} \psi\left(\frac{\alpha}{ x}+x\right) dx,
\]
where $\varepsilon_x$ were defined in \eqref{epsfact}.  These are the analogues of the classical Kloosterman and Sali\'e sums defined as:
\[
\Kl(\alpha^{\prime},\FF_q):= \sum_{x\in \FF_q^* }  e_q\left(\frac{\alpha'}{ x}+x\right)
\]
and 
\[
\Sa(\alpha^{\prime},\FF_q):=  \sum_{x\in \FF_q^* } \chi(x) e_q\left(\frac{\alpha'}{ x}+x\right)
\]
for $\alpha'\in \FF_q,$ where $\chi$ is the quadratic character on $\FF^*_q$. By Weil's estimate on Kloosterman and Sali\'e sums, 
 \(
|\Kl(\alpha^{\prime},\FF_q)| \leq 2q^{1/2} \text { and } \Sa(\alpha^{\prime},\FF_q)  \leq 2q^{1/2}.
 \)

For $\alpha\in \Ki$ and $l\in \mathbb{Z}$,  define
\begin{equation}\label{Binf}
\begin{split}
B_{\infty}(\psi,l,\alpha):=\int_{|x|=\hl} \psi(\frac{\alpha}{ x}+x) dx,
\\
\widetilde{B}_{\infty}(\psi,l,\alpha):=\int_{|x|=\hl} \varepsilon_{x} \psi(\frac{\alpha}{ x}+x) dx.
\end{split}
\end{equation}
   We write  $\alpha= t^{2l+k}\alpha^{\prime} (1+\tilde{\alpha})$ and  $x=t^lx^{\prime}(1+\tilde{x})$ for unique $\tilde{\alpha},\tilde{x}\in \TT $ and $\alpha^{\prime}, x^{\prime} \in \FF_q^*.$  Note that for $k=0,$ we have   $B_{\infty}(\psi,l,\alpha)=\Kl_{\infty}(\psi,\alpha)$ and $\widetilde{B}_{\infty}(\psi,l,\alpha)=\Sa_{\infty}(\psi,\alpha)$.
In the following lemma, we give an explicit formula for  $B_{\infty}(\psi,l,\alpha)$ in terms of the Kloosterman sums; see  \cite[Lemma 3.4]{CPS} for a similar calculation.  
\begin{lemma}\label{Bessel}

We have 
\[
B_{\infty}(\psi,l,\alpha)=\begin{cases} 
(q-1)\hl &\text{ if } \max(l+k,l)<-1, \text{ and } k\neq 0,
 \\
-\hl &\text{ if } \max(l+k,l)=-1, \text{ and } k\neq 0,
\\
0   &\text{ if } \max(l+k,l)>-1, \text{ and } k\neq 0. \end{cases}
\]
\[
\Kl_{\infty}(\psi,\alpha)=\begin{cases} 
(q-1)\hl &\text{ if } l<-1,
\\
\hl \Kl(\alpha^{\prime},\FF_q) &\text{ if } l=-1,
\\
\hl\sum_{{x^{\prime}}^2= \alpha^{\prime}}  \psi\Big(2t^{l}x^{\prime}(1+\tilde{\alpha})^{1/2} \Big) \cG(x^{\prime}t^l)   &\text{ if } \alpha^{\prime} \text{ is a quadratic residue,}
\\
0   &\text{ if } \alpha^{\prime} \text{ is not a quadratic residue.} \end{cases}
\]
Similarly,  
\[
\widetilde{B}_{\infty}(\psi,l,\alpha)=\begin{cases} 
(q-1)\hl &\text{ if } \max(l+k,l)<-1, \text{ and } k\neq 0,
 \\
-\hl &\text{ if } l+k=-1, \text{ and } k> 0,
\\
\hl \tau_{\psi} &\text{ if } l=-1, \text{ and } k< 0,
\\
0   &\text{ if } \max(l+k,l)>-1, \text{ and } k\neq 0. \end{cases}
\]
where $\tau_{\psi}:=\sum_{a\in \FF_q}e_q(a)\chi(a),$ where $\chi$ is the quadratic character in $\FF_q$. Finally,
\[
\Sa_{\infty}(\psi,\alpha)=\begin{cases} 
(q-1)\hl &\text{ if } l<-1,
\\
\hl \Sa(\alpha^{\prime},\FF_q) &\text{ if } l=-1,
\\
\hl\sum_{{x^{\prime}}^2= \alpha^{\prime}}  \psi\Big(2t^{l}x^{\prime}(1+\tilde{\alpha})^{1/2} \Big) \cG(x^{\prime}t^l)   &\text{ if } \alpha^{\prime} \text{ is a quadratic residue,}
\\
0   &\text{ if } \alpha^{\prime} \text{ is not a quadratic residue.} \end{cases}
\]
In particular, 
 \(
 \Kl_{\infty}(\psi,\alpha) \ll |\alpha|^{1/4} \text { and } \Sa_{\infty}(\psi,\alpha)  \ll |\alpha|^{1/4}.
 \)  
\end{lemma}
\begin{proof}
Suppose that $k>0$.  We have 
\begin{equation*}
\begin{split}
B_{\infty}(\psi,l,\alpha)=\int_{|x|=\hl} \psi(\frac{\alpha}{ x}+x)dx=\hl \sum_{x^{\prime}\in \FF_q^{*}} \int_{\TT} \psi\Big(\frac{ t^{l+k}\alpha^{\prime} (1+\tilde{\alpha})}{x^{\prime}(1+\tilde{x})} + t^lx^{\prime}(1+\tilde{x}) \Big) d\tilde{x}.
\end{split}
\end{equation*}
Fix $\tilde{\alpha}\in \TT $ and $\alpha^{\prime}, x^{\prime} \in \FF_q^*,$ and define the analytic function $u(\tilde{x})$ as 
\[
u(\tilde{x}):=\frac{ \alpha^{\prime} (1+\tilde{\alpha})}{x^{\prime}(1+\tilde{x})} + t^{-k}x^{\prime}(1+\tilde{x})- \left(  \frac{ \alpha^{\prime} (1+\tilde{\alpha})}{x^{\prime}} + t^{-k}x^{\prime} \right), 
\]
where $\tilde{x}\in\TT.$ We note that $u(0)=0,$ and  $|\frac{\partial u}{\partial \tilde{x}}(0)|=|-\frac{ \alpha^{\prime} (1+\tilde{\alpha})}{(1+\tilde{x})^2x^{\prime}} + t^{-k}x^{\prime}|=1.$ Hence    $u\in \cA_{\infty}(\TT).$
By changing the variable to $u(\tilde{x})$ and Proposition~\ref{auto}, we have 
\[
B_{\infty}(\psi,l,\alpha)=\hl \sum_{x^{\prime}\in \FF_q^{*}}  \psi\Big(\frac{\alpha^{\prime} (1+\tilde{\alpha}) t^{l+k}}{x^{\prime}}+x^{\prime}t^{l} \Big)  \int_{\TT} \psi( t^{l+k}u)du=\begin{cases} 
(q-1)\hl &\text{ if } l+k<-1,
\\
-\hl &\text{ if } l+k=-1,
\\
0   &\text{otherwise.} \end{cases}
\]
On the other hand, suppose that  $k<0$. 
Fix $\tilde{\alpha}\in \TT $ and $\alpha^{\prime}, x^{\prime} \in \FF_q^*,$ and define the analytic function $v(\tilde{x})$ as 
\[
v(\tilde{x}):=t^{k}\frac{ \alpha^{\prime} (1+\tilde{\alpha})}{x^{\prime}(1+\tilde{x})} + x^{\prime}(1+\tilde{x})- \left( t^{k}  \frac{ \alpha^{\prime} (1+\tilde{\alpha})}{x^{\prime}} +x^{\prime} \right), 
\]
where $\tilde{x}\in\TT.$ We note that $|\frac{\partial v}{\partial \tilde{x}}(0)|=|-\frac{ t^{k}\alpha^{\prime} (1+\tilde{\alpha})}{(1+\tilde{x})^2x^{\prime}} + x^{\prime}|=1.$ Hence    $v\in \cA_{\infty}(\TT).$
By changing the variable to $v(\tilde{x})$ and Proposition~\ref{auto}, we have 
\[
B_{\infty}(\psi,l,\alpha)=\hl \sum_{x^{\prime}\in \FF_q^{*}}  \psi\Big(\frac{\alpha^{\prime} (1+\tilde{\alpha}) t^{l+k}}{x^{\prime}}+x^{\prime}t^{l} \Big)  \int_{\TT} \psi( t^{l}v)dv=\begin{cases} 
(q-1)\hl &\text{ if } l<-1,
\\
-\hl &\text{ if } l=-1,
\\
0   &\text{otherwise.} \end{cases}
\]
Finally suppose that $k=0.$ Fix $\tilde{\alpha}\in \TT $ and $\alpha^{\prime}, x^{\prime} \in \FF_q^*.$
\[
\Kl_{\infty}(\psi,\alpha) =\hl \sum_{x^{\prime}\in \FF_q^{*}} \int_{\TT} \psi\Big(\frac{ t^{l}\alpha^{\prime} (1+\tilde{\alpha})}{x^{\prime}(1+\tilde{x})} + t^lx^{\prime}(1+\tilde{x}) \Big) d\tilde{x}.
\]
 Suppose that ${x^{\prime}}^2\neq \alpha^{\prime}$ in $\FF_q,$  and  define the analytic function $w(\tilde{x})$ as 
\[
w(\tilde{x}):=\frac{ \alpha^{\prime} (1+\tilde{\alpha})}{x^{\prime}(1+\tilde{x})} + x^{\prime}(1+\tilde{x})- \left(   \frac{ \alpha^{\prime} (1+\tilde{\alpha})}{x^{\prime}} +x^{\prime} \right), 
\]
where $\tilde{x}\in\TT.$ We note that 
$|\frac{\partial w}{\partial \tilde{x}}(0)|=|-\frac{ \alpha^{\prime} (1+\tilde{\alpha})}{(1+\tilde{x})^2x^{\prime}} + x^{\prime}|=|-\frac{ \alpha^{\prime} (1+\tilde{\alpha})-{(1+\tilde{x})^2x^{\prime}}^2}{(1+\tilde{x})^2x^{\prime}} |=1$ and    $w\in \cA_{\infty}(\TT).$ Otherwise ${x^{\prime}}^2= \alpha^{\prime}$ in $\FF_q$. Define $x_0:=(1+\tilde{\alpha})^{1/2}-1\in\TT$ and 
\[
h(\tilde{x}):=\frac{ \alpha^{\prime} (1+\tilde{\alpha})}{x^{\prime}(1+\tilde{x})} + x^{\prime}(1+\tilde{x})- 2x^{\prime}(1+\tilde{\alpha})^{1/2} .
\]
It is easy to see that $h(x_0)=0$, $\frac{\partial{h}}{\partial \tilde{x}}(x_0)=0$ and $\frac{\partial^2{h}}{\partial^2 \tilde{x}}(x_0)=\frac{2x^{\prime}}{(1+\tilde{\alpha})^{1/2}}$. Hence $x_0$ is a critical point with $|\frac{\partial^2{h}}{\partial^2 \tilde{x}}(x_0)|=1.$ By Proposition~\ref{station}, we have 

\begin{align*}
\hl \sum_{x'^2=\alpha ' } \int_{\TT} \psi\Big(\frac{ t^{l}\alpha^{\prime} (1+\tilde{\alpha})}{x^{\prime}(1+\tilde{x})} + t^lx^{\prime}(1+\tilde{x}) \Big) d\tilde{x}&=\hl \sum_{x'^2=\alpha ' }  \psi\Big(2t^{l}x^{\prime}(1+\tilde{\alpha})^{1/2} \Big) \int \psi(t^lh(\tilde{x})) d\tilde{x}
\\&= \hl\sum_{{x^{\prime}}^2= \alpha^{\prime}}  \psi\Big(2t^{l}x^{\prime}(1+\tilde{\alpha})^{1/2} \Big) \cG(x^{\prime}t^l).
\end{align*}

Therefore, 
\[
\Kl_{\infty}(\psi,\alpha)=\hl \sum_{{x^{\prime}}^2\neq \alpha^{\prime}}  \psi\Big(\frac{\alpha^{\prime} (1+\tilde{\alpha})t^l}{x^{\prime}}+x^{\prime}t^{l} \Big)  \int_{\TT} \psi( t^{l}w)dw+\hl\sum_{{x^{\prime}}^2= \alpha^{\prime}}  \psi\Big(2t^{l}x^{\prime}(1+\tilde{\alpha})^{1/2} \Big) \cG(x^{\prime}t^l) \]
Suppose that $\alpha^{\prime}$ is a quadratic non-residue in $\FF_q^*.$ Then, from above it follows that
 \[
\Kl_{\infty}(\psi,\alpha)=\begin{cases} 
(q-1)\hl &\text{ if } l<-1,
\\
\hl \Kl(\alpha^{\prime},\FF_q) &\text{ if } l=-1,
\\
0   &\text{otherwise.} \end{cases}
\]
Finally, assume that $\alpha^{\prime}$ is a quadratic residue in $\FF_q^*.$ We have 
 \[
\Kl_{\infty}(\psi,\alpha)=\begin{cases} 
(q-1)\hl &\text{ if } l<-1,
\\
\hl \Kl(\alpha^{\prime},\FF_q) &\text{ if } l=-1,
\\
\hl\sum_{{x^{\prime}}^2= \alpha^{\prime}}  \psi\Big(2t^{l}x^{\prime}(1+\tilde{\alpha})^{1/2} \Big) \cG(x^{\prime}t^l)   &\text{otherwise.} \end{cases}
\]
This concludes the proof of the first part of the lemma. By Weil's bound on  Kloosterman sums and \eqref{epsfact}, it is easy to check that $\Kl_{\infty}(\psi,\alpha) \ll (\hat{l}) ^{1/2} =|\alpha|^{1/4}.$ 
\\
\\
 The arguments for $\widetilde{B}_{\infty}(\psi,l,\alpha)$ and $\Sa_{\infty}(\psi,\alpha)$ of the lemma follow along the same lines, and we skip the details.  
\end{proof}

\begin{proof}[Proof of Proposition~\ref{c}]
By Lemma~\ref{auxlem},  $|G(\vec{t})| < \hQ|r|$ is equivalent to $|F(\vec{t})-k/g|<\hQ |r|$ for $|\vec{t}-\vec{t}_0|<\hR.$ By Lemma~\ref{lem:orthog}, we have 
$$
\int_{\TT} \psi\Big(\frac{\alpha} {rt^Q } (F(\vec{t})-k/g)\Big) d \alpha=\begin{cases}
 1, &\mbox{if $|F(\vec{t})-k/g|< \hQ|r|$},\\
0, &\mbox{otherwise.}
\end{cases}
$$
We replace the above integral for detecting $|F(\vec{t})-k/g|< \hQ|r|.$ Hence, by \eqref{intt}
\[
I_{g,r}(\vec{c})=\frac{\hQ}{|r|} \int _{\TT}\int_{|\vec{t}-\vec{t}_0| <\hR}\psi\left(\frac{\left< \vec{c},\vec{t} \right>}{gr} +\frac{\alpha} {rt^Q } (F(\vec{t})-k/g)   \right) d\vec{t} d\alpha.\]
Recall  that $F(\gamma\vec{y})=\sum_{\eta_i} \eta_iy_i^2$ for some  $\gamma\in \GL_d(\cO_{\infty}).$  We change  variables to $\vec{y}=\begin{bmatrix}y_1 \\ \vdots \\y_d  \end{bmatrix}=\gamma^{-1} \vec{t},$ and obtain 
 \[
\frac{\left< \vec{c},\vec{t} \right>}{gr} +\frac{\alpha} {rt^Q } (F(\vec{t})-k/g)=\frac{-\alpha k}{rgt^Q}+ \frac{1}{r} \Big(\sum_{i} \frac{c_i^{\prime}y_i}{g} +\frac{\alpha  \eta_i y_i^2}{t^Q}   \big),
 \]
 where $\vec{c}':=\begin{bmatrix}c_1^{\prime} \\ \vdots \\c_d^{\prime}  \end{bmatrix}:=\gamma^{\intercal } \vec{c}.$ Let $\vec{y}_0:= \begin{bmatrix}{y_1}_0 \\ \vdots \\{y_d}_0 \end{bmatrix}:= \gamma^{-1}\vec{t}_0.$ Then 
 $\gamma$ is a bijection between  $\left\{\vec{t}\in K_\infty^d: \ |\vec{t}-\vec{t}_0|<\hR  \right\}$ and  $\left\{\vec{y}\in K_\infty^d: \ |\vec{y}-\vec{y}_0|<\hR  \right\}.$
 Hence,
 \(
I_{g,r}(\vec{c})=\frac{\hQ}{|r|} \int _{\TT} \psi(\frac{-\alpha k}{rgt^Q})  I_{g,r}(\alpha,\vec{c})d\alpha,\)
where 
\[
I_{g,r}(\alpha,\vec{c}):=\prod_{i=1}^{d}\int_{|y_i-{y_i}_0|<\hR} \psi \left( \frac{1}{r} \left( \frac{c_i^{\prime}y_i}{g} +\frac{\alpha  \eta_i y_i^2}{t^Q}   \right)\right) dy_i.
 \]
  We write $z_i:=y_i-{y_i}_0.$ We have  
 \[
I_{g,r}(\alpha,\vec{c}):=\prod_{i=1}^{d}\int_{|z_i|<\hR} \psi \left( \frac{1}{r} \left( \frac{c_i^{\prime}(z_i+{y_i}_0)}{g} +\frac{\alpha  \eta_i (z_i+{y_i}_0)^2}{t^Q}   \right)\right) dz_i,
 \] 
 The phase function has a critical point at $\frac{-c^{\prime}_i t^Q}{2g\eta_i \alpha}-{y_i}_0,$ when $\alpha\neq 0.$ Without loss of generality, we may assume that $\alpha\neq 0$ as $\{0\}$ has measure $0$ in $\TT$. This critical point is inside the domain of the integral if 
 $|\kappa_i|<\hR,$ where $\kappa_i:=\frac{c^{\prime}_i t^Q}{g\eta_i \alpha}+2{y_i}_0.$  Note that $\kappa_i$ is a function of $\alpha.$  Let $\boldsymbol{\kappa}:= \begin{bmatrix}\kappa_1 \\ \vdots \\ \kappa_d \end{bmatrix}.$
  Given $\alpha\in\TT$, we partition the indices into: 
\begin{equation*}
\begin{split}
CR:=\left\{1\leq i\leq d: |\kappa_i|<\hR  \right\},
\\
NCR:=\left\{1\leq i\leq d: |\kappa_i|\geq \hR \right\}.
\end{split}
\end{equation*}
For $i\in NCR,$ we change the variables to 
\(
v_i=z_i + \kappa_i^{-1} z_i^2
\). For $i\in CR$, we change the variables to  $w_i= z_i+\kappa_i/2.$   Hence,
\begin{equation}\label{oceq1}
\begin{split}
I_{g,r}(\alpha,\vec{c})= \prod_{i\in NCR} \psi \left( \frac{1}{r} \left( \frac{c_i^{\prime}{y_i}_0}{g} +\frac{\alpha  \eta_i {y_i}_0^2}{t^Q}   \right)\right)  \int_{|v_i|<\hR}\psi\left(  \frac{\alpha \eta_i }{rt^{Q}} \kappa_i v_i \right) dv_{i}
\\
\times \prod_{i\in CR}\psi (-\frac{t^Q{c_i^{\prime}}^2}{4rg^2\eta_i \alpha}) \int_{|w_i|<\hR}\psi
\left( \frac{\alpha \eta_i }{rt^{Q}}w_i^2  \right) dw_i.
\end{split}
\end{equation}
By Lemma~\ref{lem:orthog} and  Lemma~\ref{ggg}, we have 
\begin{equation}\label{osceq2}
\int_{|v_i|<\hR}\psi\left(    \frac{\alpha \eta_i }{rt^{Q}} \kappa_i v_i   \right) dv_{i}=\begin{cases}
\hR, &\mbox{if $ |\frac{\alpha \eta_i }{rt^{Q}} \kappa_i | <1/\hR$  },\\
0, &\mbox{otherwise,}
\end{cases}
\end{equation}
\[
 \int_{|w_i|<\hR}\psi
\left( \frac{\alpha \eta_i }{rt^{Q}}w_i^2  \right) dw_i=\hR\cG\left(\frac{\alpha \eta_i t^{2R}}{rt^Q}\right).
\]
Note that 
\[
\boldsymbol{\kappa}= \gamma^{-1} \left( \frac{t^Q}{g\alpha}A^{-1} \vec{c}+2\vec{t}_0  \right).
\]
Since $\gamma^{-1} \in \GL_d(\cO_{\infty}),$ by Lemma~\ref{opnorm} 
\[
|\boldsymbol{\kappa}|=  \left|  \frac{t^Q}{g\alpha}A^{-1} \vec{c}+2\vec{t}_0\right|.
\]
By Lemma~\ref{t0} and its proof, $\vec{t}_0\in \Omega$ and $|\vec{t}_0|/\hat{\omega} \geq \hR$.  Suppose that $\vec{c}\notin\Omega^*.$ Then  $ \frac{t^Q}{g\alpha}A^{-1} \vec{c} \notin \Omega.$  By Lemma~\ref{ultra} applied to $\Omega$, 
\[
 |\boldsymbol{\kappa}| \geq |\vec{t}_0|/\hat{\omega} \geq \hR.
\]
 On the other hand, recall that  $\kappa:=\frac{|\vec{c}|}{|g|}.$  Since $\vec{c}'=\gamma^{\intercal } \vec{c}$ and $\gamma\in \GL_d(\cO_{\infty}),$ $\kappa=\frac{|\vec{c}^{\prime}|}{|g| }.$   Since $\vec{c}\notin\Omega^*$ and $A\vec{t}_0\in\Omega^*$, by Lemmas~\ref{ultra} and~\ref{dualcone} applied to $\Omega^*$, 
\[
\max_{1\leq i\leq d} \left|\kappa_i \frac{\alpha \eta_i }{t^{Q}} \right| =\max_{1\leq i\leq d}  \left(\frac{c_i^{\prime} }{g}+\frac{2\alpha \eta_i {y_i}_{0}}{t^Q} \right) = \left| \frac{\vec{c}}{g}+ \frac{2\alpha}{t^Q}A\vec{t}_0\right|
\geq  \kappa/\hat{\omega^*}    
\]
 By our assumption,   $\kappa \geq \eta  \frac{|r|}{\hR}. $ Since $\eta>\hat{\omega^*},$  
 
\[ \max_{1\leq i\leq d} |\frac{\alpha \eta_i }{rt^{Q}} \kappa_i | \geq \frac{\kappa}{|r|\hat{\omega^*} }\geq     1/\hR .\]
By equations \eqref{oceq1} and~\eqref{osceq2}, we have    $I_{g,r}(\vec{c})=0$ for   $\vec{c}\notin\Omega^*.$
 \\
 
 Next, we suppose that $\vec{c}\in\Omega^*$ and prove inequality~\eqref{excep}. We show that    $I_{g,r}(\alpha, \vec{c})=0$ unless 
 \(
 |\alpha|= \hl,
\)
where $\hl:= \kappa \frac{\hQ}{|A\vec{t}_0|}.$  Suppose that $|\alpha|\neq \hl$. Then $\left| \frac{\vec{c}}{g}\right| \neq \left| \frac{2\alpha}{t^Q}A\vec{t}_0\right|,$ which implies 
\[
\max_{1\leq i\leq d} |\frac{\alpha \eta_i }{rt^{Q}} \kappa_i |= \left| \frac{\vec{c}}{rg}+ \frac{2\alpha}{rt^Q}A\vec{t}_0\right| =  \max \left( \frac{\kappa}{|r|}, \frac{|\alpha A\vec{t}_0|}{\hQ|r|}  \right) \geq  \max\left(\frac{\eta}{\hR}, \frac{|\alpha| \hR}{\hQ|r|}  \right),
\]
where we use 
\[
| A\vec{t}_0|\geq \frac{|f|^{1/2}}{|g| |A^{-1}|} \geq \hR \frac{\hat{\alpha_0}}{|A^{-1}|} \geq \hR.
\]
The last inequality follows from $\alpha_0:=2\omega+2\max{\deg(\eta_i)}+3|\min\deg(\eta_i)|+\omega'$ (see equation~\eqref{alpha0}) and $|A^{-1}|=\frac{1}{\min|\eta_i|}.$ Therefore, there exists an index $i$ such that 
\(
\max(\hR , \frac{|r|\hQ}{\hR|\alpha \eta_i|})\leq    |\kappa_i| 
\)
which implies $i\in NCR$ and, by~\eqref{oceq1} and \eqref{osceq2}, $I_{g,r}(\alpha, \vec{c})=0.$
\\
\\
Suppose that $I_{g,r}(\alpha, \vec{c})\neq 0.$ Then  $|\alpha|= \hl.$  By equations \eqref{oceq1} and~\eqref{osceq2}, we have
\begin{equation}\label{oscformula}
\begin{split}
I_{g,r}(\alpha,\vec{c})=\hR^d\prod_{i=1}^d \left(
  \delta_{\hR \leq \kappa_i < \frac{|r|\hQ}{\hR|\alpha \eta_i|}} \psi \left( \frac{1}{r} \left( \frac{c_i^{\prime}{y_i}_0}{g} +\frac{\alpha  \eta_i {y_i}_0^2}{t^Q}   \right)\right)  +\delta_{\kappa_i<\hR}  \psi (-\frac{t^Q{c_i^{\prime}}^2}{4rg^2\eta_i \alpha})\cG\left(\frac{\alpha \eta_i t^{2R}}{rt^Q}\right)\right).
\end{split}
\end{equation}
The contribution of $NCR$ is non-zero  only if
\(
\hR \leq \frac{|r|\hQ}{\hR|\alpha \eta_i|}
\) for some $i$. This implies 
\[
|\alpha|\leq  \frac{|r|}{\hR} \left(\frac{\hQ}{\hR|\eta_i|}  \right) \ll   \frac{|r|}{\hR}.
\]
By comparing the preceding inequality with $ |\alpha|=\hl \gg \kappa $, we have 
\(
\kappa\ll \frac{|r|}{\hR}.
\)  
By choosing $\eta$ large enough, this contradicts with our assumption $\kappa \geq \eta  \frac{|r|}{\hR}. $ Therefore, for large enough $\eta$
 \begin{equation*}
\begin{split}
I_{g,r}(\vec{c})&=\frac{\hQ\hR^{d}}{|r|} \int _{ |\alpha|=\hl  }  \psi(\frac{-\alpha k}{rgt^Q})  \prod_{\kappa_i<\hR} \psi (-\frac{{t^Qc_i^{\prime}}^2}{4rg^2\eta_i \alpha}) 
\cG\left(\frac{\alpha \eta_i t^{2R}}{rt^Q}\right)
d\alpha.
\end{split}
\end{equation*}
By~\eqref{epsfact},  we have 
\[
 \prod_{i}\cG\left(\frac{\alpha \eta_i t^{2R}}{rt^Q}\right) 
  =\pm\varepsilon_{\alpha}^{v} \prod_{i}\min\left(1,\left(\frac{\hl \hR^2|\eta_i|}{|r|\hQ}\right)^{-1/2}\right), 
\]
where $v=0 ,1$ depending on parity of the degrees of $\eta_i$ and $\alpha$ and quadratic residue of their top coefficients.   Hence,
\[
I_{g,r}(\vec{c})=\pm\frac{\hQ\hR^{d}}{|r|} \prod_{i}\min\left(1,\left(\frac{\hl \hR^2|\eta_i|}{|r|\hQ}\right)^{-1/2}\right) \int _{ |\alpha| =\hat{l}  }  \psi(\frac{-\alpha k}{rgt^Q})   \psi (-\frac{t^QF^*(\vec{c})}{4rg^2 \alpha}) \varepsilon_{\alpha}^v d\alpha,
\]

where $F^*(\vec{c})=\sum_{i} \frac{{c_i^{\prime}}^2}{\eta_i}.$  By substituting $\beta=\frac{-\alpha k}{rgt^Q},$ definitions~\eqref{Binf}, and
  Lemma~\ref{Bessel}, we have 
\begin{equation*}
\begin{split}
\left| \int _{ |\alpha| =\hat{l}  }  \psi(\frac{-\alpha k}{rgt^Q})   \psi (-\frac{t^QF^*(\vec{c})}{4rg^2 \alpha}) \varepsilon_{\alpha}^v  d\alpha \right|&
= |\frac{rgt^Q}{ k}|\left| \int _{ |\beta| =\hat{l}|\frac{k}{rgt^Q}|  }  \psi(\beta+\frac{kF^*(\vec{c})}{4r^2g^3}\frac{1}{\beta}) \varepsilon_{\beta}^v  d \beta \right|
\\&=\begin{cases} |\frac{rgt^Q}{ k}| \left|B_{\infty}(\psi,l+\deg(\frac{ k}{rgt^Q}),\frac{kF^*(\vec{c})}{4r^2g^3}) \right| &\text{ for } v=0,
\\
\\
 |\frac{rgt^Q}{ k}| \left|\tilde{B}_{\infty}(\psi,l+\deg(\frac{ k}{rgt^Q}),\frac{kF^*(\vec{c})}{4r^2g^3}) \right|&\text{ for } v=1,
\end{cases}
\end{split}
\end{equation*}
By definition, $\hl=  \kappa \frac{\hQ}{|A\vec{t}_0|},$ and, 
from the assumptions of Proposition~\ref{c},  $\kappa\geq  \eta \frac{|r|}{\hR}$ and $\eta>\hat{\omega'^*}.$ Hence,
\[
l+\deg(\frac{ k}{rgt^Q})> \deg{k}+\omega'^*-\left(R+\deg(A\vec{t}_0)+\deg g    \right).
\]
By Lemma~\ref{t0}, $\vec{t}_0\in \Omega$ implying that $A\vec{t}_0 \in \Omega^*,$ by definition. Since $|F^*(A\vec{t}_0)|=|f/g^2|$, by property~\eqref{prop3} of the anisotropic cone $\Omega^*$
\[
|A\vec{t}_0|\leq \hat{\omega'^*}^{1/2} |f|^{1/2}/|g|.
\]
Furthermore, $\deg k= \deg f -\deg g $ and $R= \lceil{\deg(f)/2} -deg(g)-\alpha_0/2\rceil.$ Therefore,
\[
l+\deg(\frac{ k}{rgt^Q})>\deg{k}+\omega'^*-\left(R+\deg(A\vec{t}_0)+\deg g    \right)> 0.
\]
The above and Lemma~\ref{Bessel} implies that 

\begin{equation*}
\begin{split}
\left| \int _{ |\alpha| =\hat{l}  }  \psi(\frac{-\alpha k}{rgt^Q})   \psi (-\frac{t^QF^*(\vec{c})}{4rg^2 \alpha}) \varepsilon_{\alpha}^v  d\alpha \right|=\begin{cases} 
 |\frac{rgt^Q}{ k}|\left| \Kl_{\infty}(\psi, \frac{kF^*(\vec{c})}{4r^2g^3})\right| &\text{  if } 2l=\deg(\frac{t^{2Q}F^*(\vec{c})}{kg}), \text{ and } v=0
 \\
 \\
 |\frac{rgt^Q}{ k}| \left|\Sa_{\infty}(\psi, \frac{kF^*(\vec{c})}{4r^2g^3})\right| &\text{  if } 2l=\deg(\frac{t^{2Q}F^*(\vec{c})}{kg}), \text{ and } v=1
 \\
 \\
 0 &\text{ otherwise.}
 \end{cases} 
\end{split}
\end{equation*}
Therefore, by using the Weil bound on Kloosterman sums or Sali\'e sums,  we have

\[
|I_{g,r}(\vec{c})| \ll \frac{\hQ\hR^{d}}{|r|}\left(\frac{ |F^*(\vec{c})|^{1/2} \hR^2}{|f|^{1/2}|r|}\right)^{-d/2} \left|\frac{rgt^Q}{ k}\right|   \left|\frac{fF^*(\vec{c})}{r^2g^4}\right|^{1/4}\ll  \hQ^d\Big(\frac{|F^*(\vec{c})|^{1/2}\hQ}{|gr|}   \Big)^{-\frac{d-1}{2}},
\]
where we used $|f|^{1/2}\gg \hQ|g|$ and equation~\eqref{defk} defining $k$. Since $|\vec{c}|\ll |F^*(\vec{c})|^{1/2} $ for $\vec{c}\in \Omega^*$, this concludes Proposition~\ref{c}.

\end{proof}

\section{Main contribution to counting function}\label{mainboundsection}
In this section, we study the main contribution to the counting function $N(w,\boldsymbol{\lambda})$. We first begin by estimating the contribution in $N(w,\boldsymbol{\lambda})$ from the terms where $\vec{c}=0$. In order to do so, we first prove the following lemma which gives an estimate on the the norm of $I_{g,r}(\vec{0})$ for $|r|\leq\hQ^{1-\varepsilon}$. We then show that the contribution from $\hQ^{1-\varepsilon}\leq |r|\leq\hQ$ and $\vec{c}=0$ is small. Finally, we show that contribution from $\vec{c}=0$ can be written in terms of local densities.
\begin{proposition}\label{lem:Izerobound}Recall the notation from before. Suppose $\varepsilon>0$ and $1\leq |r|\leq\hQ^{1-\varepsilon}$. There is a constant $0<\sigma_{\infty}\leq 1$ depending only on $f,F,\Omega$ such that
\[I_{g,r}(\vec{0})=\sigma_{\infty}\hQ^d\]
for large enough $Q$ depending only on $\varepsilon,F,\Omega$.
\end{proposition}
\begin{proof}
It follows from equation~\eqref{intt} that
\[I_{g,r}(\vec{0})=\frac{\hQ}{|r|} \int_{\substack{|\vec{t}-\vec{t}_0|< \hat{R}\\ |G(\vec{t})|< \hQ|r|}} d\vec{t}=\frac{\hQ}{|r|} \int_{\substack{|g\vec{t}+\boldsymbol{\lambda}-\vec{x}_0| \leq |t^{-\alpha_0}f|^{1/2}\\ |F(g\vec{t}+\boldsymbol{\lambda})-f|<\hQ|r||g|^2}} d\vec{t}.\]
Making the substitution $\vec{x}=g\vec{t}+\boldsymbol{\lambda}$ gives us the equality
\[I_{g,r}(\vec{0})=\frac{\hQ}{|r||g|^d}\int_{|\vec{x}-\vec{x}_0|\leq |t^{-\alpha_0}f|^{1/2}:|F(\vec{x})-f|< \hQ|r||g|^2}d\vec{x}.\]
Let $f=\alpha_fu^2$, where $\alpha_f\in\{1,\nu,t,\nu t\}$ is the quadratic residue of $f$. Let $D:=E+\deg u, E:=\lceil\frac{1}{2}(-\alpha_0+\deg \alpha_f+1)\rceil$, and write $\tilde{\vec{x}}_0:=(t^{-E}/u)\vec{x}_0$. By Lemma~\ref{lem:orthog} and Fubini, we may rewrite this as
\begin{eqnarray*}I_{g,r}(\vec{0})&=&\frac{\hQ}{|r||g|^d}\int_{|\vec{x}-\vec{x}_0|\leq |t^{-\alpha_0}f|^{1/2}}\int_{\TT}\psi\left(\frac{F(\vec{x})-f}{rg^2t^{Q}}\alpha\right)d\alpha d\vec{x}\\ &=&\frac{\hQ}{|r||g|^d}\int_{\TT}\int_{|\vec{x}|< \hat{D}}\psi\left(\frac{F(\vec{x}+\vec{x}_0)-f}{rg^2t^Q}\alpha\right)d\vec{x}d\alpha\\ &=& \frac{\hQ\hat{D}^d}{|r||g|^d}\int_{\TT}\int_{\TT^d}\psi\left(\frac{F(\vec{x}+\tilde{\vec{x}}_0)-F(\tilde{\vec{x}}_0)}{rg^2t^{Q-2E}/u^2}\alpha\right)d\vec{x}d\alpha
\end{eqnarray*}
where the last equality follows from a change of variables of $\vec{x}$ coordinates by a factor of $t^Eu$. Note that $|t^Eu|=\hat{D}$ and $F(\tilde{\vec{x}}_0)=\frac{\alpha_f}{t^{2E}}$. Making the substitution $\beta=\frac{\alpha}{rg^2t^{Q-2E}/u^2}$, we obtain the equality
\begin{align}\label{integralIzero}I_{g,r}(\vec{0})&=\frac{\hQ^2\hat{D}^{d-2}}{|g|^{d-2}}\int_{|\beta|<\frac{\hat{D}^2}{\hQ|r||g|^2}}\int_{\TT^d}\psi\left((F(\vec{x}+\tilde{\vec{x}}_0)-F(\tilde{\vec{x}}_0))\beta\right)d\vec{x}d\beta\nonumber\\&=
q^{C(f,F,\Omega)}\hQ^d\int_{|\beta|<\frac{\hat{D}^2}{\hQ|r||g|^2}}\int_{\TT^d}\psi\left((F(\vec{x}+\tilde{\vec{x}}_0)-F(\tilde{\vec{x}}_0))\beta\right)d\vec{x}d\beta,\end{align}
where $C(f,F,\Omega):=E-\lceil\frac{1}{2}\deg\alpha_f\rceil-\max\deg\eta_i-\omega'$ is bounded in terms of $F,\Omega$; note that $\deg\alpha_f\in\{0,1\}$. Using Lemma~\ref{lem:orthog}, double integral is equal to
\begin{equation}\label{doubleint}\frac{\hat{D}^2}{\hQ|r||g|^2}\vol\left(\left\{\vec{x}\in\TT^d:|F(\vec{x}+\tilde{\vec{x}}_0)-F(\tilde{\vec{x}}_0)|< \frac{\hQ|r||g|^2}{\hat{D}^2}\right\}\right)\geq 0.
\end{equation}
Let $L$ be such that $\hat{L}=\frac{\hat{D}^2}{\hQ|r||g|^2}$, and write $\Theta(\vec{x}):=F(\vec{x}+\tilde{\vec{x}}_0)-F(\tilde{\vec{x}}_0)$. ~\eqref{doubleint} then becomes
\[\hat{L}\vol\left(\left\{\vec{x}\in\TT^d:|\Theta(\vec{x})|< \hat{L}^{-1}\right\}\right).\]
We show that the above quantity does not depend on $L\geq 0$. 
\\
\\
Suppose $L_1\geq 0$, and let 
\[
\delta({L_1}):=\left\{\vec{x}\in\TT^d:|\Theta(\vec{x})|< \hat{L}_1^{-1}\right\}.
\]
 Let $L_2:=L_1+\deg(A(\tilde{\vec{x}}_0)).$  Next, we show that
  \[\delta({L_1})=\bigsqcup_{i=1}^{M(L_1)} (\vec{x}_i(L_1)+t^{-L_2}\TT^d),
  \]
for finitely many $\vec{x}_i\in \TT^d$.
Suppose that $\vec{x}_1,\vec{x}_2\in\TT^d$ and $|\vec{x}_1-\vec{x}_2|< \hat{L}_2^{-1}$. We write 
\[
\vec{x}_1-\vec{x}_2= \overline{\vec{x}_1-\vec{x}_2}t^{-L_2-1}+O(t^{-L_2-2}),
\]
where  $\overline{\vec{x}_1-\vec{x}_2}\in \mathbb{F}_q^d,$ and  $O(t^{-L_2-2})$ means a vector with degree at most $-L_2-2.$  
We have
\begin{eqnarray*}
\Theta(\vec{x}_1)-\Theta(\vec{x}_2)=F(\vec{x}_1-\vec{x}_2)+2(\vec{x}_1-\vec{x}_2)^{\intercal}A(\vec{x}_2+\tilde{\vec{x}}_0).
\end{eqnarray*}
Using property~\eqref{prop3} of anisotropic cones,
\[\left|\frac{\alpha_f}{t^{2E}}\right|=|F(\tilde{\vec{x}}_0)|\leq |\tilde{x}_0||A\tilde{\vec{x}}_0|\leq|A\tilde{\vec{x}}_0|\hat{\omega'}^{1/2}|F(\tilde{\vec{x}}_0)|^{1/2}=|A\tilde{\vec{x}}_0|\hat{\omega'}^{1/2}\left|\frac{\alpha_f}{t^{2E}}\right|^{1/2},\]
from which it follows that
\[|A\tilde{\vec{x}}_0|\geq \frac{1}{\hat{\omega'}^{1/2}}\left|\frac{\alpha_f}{t^{2E}}\right|^{1/2}.\]
Recall that $E=\lceil\frac{1}{2}(-\alpha_0+\deg \alpha_f+1)\rceil$ and $\alpha_0=2\omega+2\max{\deg(\eta_i)}+3|\min\deg(\eta_i)|+\omega'$. Therefore,
\[|A\tilde{\vec{x}}_0|\geq q^{\omega+\max{\deg(\eta_i)}+\frac{3}{2}|\min\deg(\eta_i)|}\geq |A|.\]
Since we also have $\vec{x}_2\in\TT^d$, this implies that $|A\vec{x}_2| < |A\tilde{\vec{x}}_0|,$ and so we may write  
\[
A(\vec{x}_2+\tilde{\vec{x}}_0)= t^{\deg(A(\tilde{\vec{x}}_0))}\overline{A(\tilde{\vec{x}}_0)}+ O(t^{\deg(A\tilde{\vec{x}}_0)-1}),
\]
where $\overline{A(\tilde{\vec{x}}_0)}\in \mathbb{F}_q^d$ is non-zero. Since $|A\tilde{\vec{x}}_0| \geq |A|\geq 1$,  
\[
|F(\vec{x}_1-\vec{x}_2)| \leq |A| |\vec{x}_1-\vec{x}_2|^2 \leq |A| q^{-2}\hat{L}_2^{-2} = \frac{|A|}{q^2 |A\tilde{\vec{x}}_0|^2 \hat{L}_1^2} \leq q^{-2}\hat{L}_1^{-1}.
\]Therefore,
\begin{equation}\label{thetaeq}
\Theta(\vec{x}_1)-\Theta(\vec{x}_2)= 2 \overline{(\vec{x}_1-\vec{x}_2)}^{\intercal}\overline{A(\tilde{\vec{x}}_0)} t^{-L_1-1} +O(t^{-L_1-2}).
\end{equation}
The above inequality shows that if $\vec{x}_1\in \delta(L_1)$ then $\vec{x}_1+t^{-L_2}\TT^d \subset \delta(L_1).$ Since our norm is non-archimedean, the cosets are either identical or disjoint. Hence 
  \[\delta({L_1})=\bigsqcup_{i=1}^{M(L_1)} (\vec{x}_i(L_1)+t^{-L_2}\TT^d), \]
and 
\begin{equation}\label{volid}
\vol (\delta(L_1))= M(L_1)\hat{L}_2^{-d}.
\end{equation}
We write $\Theta(\vec{x}_i(L_1))=\overline{\Theta(\vec{x}_i(L_1))}t^{-L_1-1}+O(t^{-L_1-2}),$ where $\overline{\Theta(\vec{x}_i(L_1))} \in \mathbb{F}_q.$ Note that $\delta(L_1+1)\subset \delta(L_1).$
It follows from \eqref{thetaeq} that
\[
\delta(L_1+1)=\bigsqcup_{i=1}^{M(L_1)} \bigsqcup_{j=1}^{N_i}  (\vec{x}_i(L_1) +\boldsymbol{\xi}_{ij}t^{-L_2-1}+t^{-L_2-1}\TT^d), 
\]
where $N_i$ is the number of solutions $\boldsymbol{\xi}_{ij}\in \mathbb{F}_q^d$  to $2\boldsymbol{\xi}_{ij}^{\intercal}\overline{A(\tilde{\vec{x}}_0)}=\overline{\Theta(\vec{x}_i(L_1))}.$ Since $\overline{A(\tilde{\vec{x}}_0)}\neq 0$, $N_i=q^{d-1}.$
Therefore,
\[
M(L_1+1)=M(L_1)q^{d-1},
\]
and so by~\eqref{volid},
\[
\hat{L_1+1}\vol(\delta({L_1+1}))=\hat{L}_1\vol(L_1).
\]
This shows that $\hat{L}_1\vol(\delta(L_1))$ is independent of $L_1\geq 0$. Since $\Theta(\vec{0})=0$, $M(L_1)>0$. Therefore, $\hat{L}\vol(\delta(L))=\vol(\delta(0))>0$ for every $L\geq 0$. Recall that in our case, $\hat{L}=\frac{\hat{D}^2}{\hQ|r||g|^2}$. Therefore,~\eqref{doubleint} stabilizes once $\hat{L}\geq 1$. By our choice of $Q=\lceil{\deg(f)/2}-\deg(g)\rceil +\max_{i}(\deg(\eta_i))+\omega^{\prime}$ and $1\leq |r|\leq\hQ^{1-\varepsilon}$, we have
\[\hat{L}=\frac{\hat{D}^2}{\hQ|r||g|^2}\gg_{F,\Omega}\hQ^{\varepsilon}\]
which is at least $1$ for large enough $Q$ depending on $\varepsilon,F,\Omega$.
The proposition follows from~\eqref{integralIzero} and the above by setting $\sigma_{\infty}:=q^{C(f,F,\Omega)}\vol(\delta(0))$.
\end{proof}
We now show that when $\hQ^{1-\varepsilon}\leq |r|\leq\hQ$, then the contribution of the terms in $N(w,\boldsymbol{\lambda})$ when $\vec{c}=\vec{0}$ and corresponding to such $r$ is small. This follows from the following more general statement for all $\vec{c}$.
\begin{lemma}\label{lem:small}For every $\varepsilon>0$,
\[\sum_{\hQ^{1-\varepsilon}\leq |r|\leq\hQ}|gr|^{-d}|S_{g,r}(\vec{c})||I_{g,r}(\vec{c})|\ll_{\varepsilon,F,\Omega}|g|^{\varepsilon}\hQ^{\frac{d+3}{2}+2\varepsilon}.\]
\end{lemma}
\begin{proof}
Suppose $\hQ^{1-\varepsilon}\leq |r|\leq\hQ$. Using equation~\eqref{intt} and the normalization $\int_{\TT}d\alpha=1$, we obtain
\[|I_{g,r}(\vec{c})|= \frac{\hQ}{|r|}\left|\int_{\substack{|\vec{t}-\vec{t}_0| < \hR \\ |G(\vec{t})|< \hQ|r|}} \psi\left(\frac{\left<\vec{c},\vec{t}\right>}{g}\right)d\vec{t}\right|\leq \frac{\hQ}{|r|}\vol(\vec{t}\in\Ki^d:|\vec{t}-\vec{t}_0|<\hR)=\frac{\hQ\hR^d}{|r|}.\]
By the definitions of $R,Q$ in~\eqref{R},\eqref{defQ}, respectively, we have $\hR\ll_{F,\Omega}\hQ$. Using $|r|\geq\hQ^{1-\varepsilon}$, we obtain
\[|I_{g,r}(\vec{c})|\ll_{F,\Omega}\hQ^{d+\varepsilon}.\]
Using this, we obtain
\begin{eqnarray*}\sum_{\hQ^{1-\varepsilon}\leq |r|\leq\hQ}|gr|^{-d}|S_{g,r}(\vec{c})||I_{g,r}(\vec{c})|&=& \sum_{\hQ^{1-\varepsilon}\leq |r|\leq\hQ}|r|^{-\frac{d-1}{2}}|g|^{-d}|r|^{-\frac{d+1}{2}}|S_{g,r}(\vec{c})||I_{g,r}(\vec{c})|\\ &\ll_{F,\Omega}& \hQ^{d+\varepsilon}\sum_{(1-\varepsilon)Q<k\leq Q}\left(q^k\right)^{-\frac{d-1}{2}}\sum_{|r|=q^k}|g|^{-d}|r|^{-\frac{d+1}{2}}|S_{g,r}(\vec{c})|
\end{eqnarray*}
By Proposition~\ref{prop:Supperbound},
\[\sum_{|r|=q^k}|g|^{-d}|r|^{-\frac{d+1}{2}}|S_{g,r}(\vec{c})|\ll_{\varepsilon,F} |g|^{\varepsilon}\left(q^k\right)^{1+\varepsilon}.\]
Therefore,
\begin{align*}
\hQ^{d+\varepsilon}\sum_{(1-\varepsilon)Q<k\leq Q}\left(q^k\right)^{-\frac{d-1}{2}}\sum_{|r|=q^k}|g|^{-d}|r|^{-\frac{d+1}{2}}|S_{g,r}(\vec{0})|&\ll_{\varepsilon,F,\Omega} |g|^{\varepsilon}\hQ^{d+\varepsilon}\sum_{(1-\varepsilon)Q<k\leq Q}\left(q^k\right)^{-\frac{d-3}{2}+\varepsilon}\\&\ll_{\varepsilon,F,\Omega} |g|^{\varepsilon}\hQ^{\frac{d+3}{2}+2\varepsilon},
\end{align*}
from which the conclusion follows.
\end{proof}
In the following lemma, we bound the tail of the series  $\sum_r |r|^{-d}S_{g,r}(\vec{0})$ for $d\geq 4$. 
\begin{lemma}\label{lem:estimate}For $d\geq 4$ and  $\varepsilon>0$, we have
\[ \sum_{r}|r|^{-d}S_{g,r}(\vec{0}) = \sum_{r:1\leq |r|\leq \hat{T}}|r|^{-d}S_{g,r}(\vec{0})+O_{\varepsilon,F}(|g|^{d+\varepsilon}\hat{T}^{3/2-\frac{d}{2}+2\varepsilon}).\]
\end{lemma}
\begin{proof}
Write 
\[\sum_{r}|r|^{-d}S_{g,r}(\vec{0})=\sum_{r:1\leq |r|\leq \hat{T}}|r|^{-d}S_{g,r}(\vec{0})+\sum_{|r|>\hat{T}}|r|^{-d}S_{g,r}(\vec{0}).\]
The triangle inequality gives us
\[
\left|\sum_{|r|>\hat{T}}|r|^{-d}S_{g,r}(\vec{0})\right|\leq \sum_{\hat{N}=\hat{T}}^{\infty}\hat{N}^{-d}\sum_{|r|=\hat{N}}|S_{g,r}(\vec{0})|.
\]
From $S_{g,r}(\vec{0})=S_1S_2$ and Proposition~\ref{lem:S1bound}, we have
\[|S_{g,r}(\vec{0})|\ll_{\Delta} |g|^d\tau(r_1)|r|^{d/2}|r_1|^{1/2}|r_2||\gcd(r_1,f)|^{1/2},\]
using which we obtain
\begin{eqnarray*}\sum_{N=T}^{\infty}\hat{N}^{-d}\sum_{|r|=\hat{N}}|S_{g,r}(\vec{0})|&\ll_{\Delta}& |g|^d\sum_{N=T}^{\infty}\hat{N}^{-d/2}\sum_{|r|=\hat{N}} \tau(r_1)|r_1|^{1/2}|r_2||\gcd(r_1,f)|^{1/2}\\ &\ll_{\Delta} & |g|^{d+\varepsilon}\sum_{N=T}^{\infty}\hat{N}^{1-\frac{d}{2}}\sum_{|r_1|\leq\hat{N}} \tau(r_1)|r_1|^{-1/2}|\gcd(r_1,f)|^{1/2}\\ &\ll_{\Delta} & |g|^{d+\varepsilon}\sum_{N=T}^{\infty}\hat{N}^{3/2-\frac{d}{2}+\varepsilon}\sum_{|r_1|\leq\hat{N}}\frac{|\gcd(r_1,f)|^{1/2}}{|r_1|}\\ &\ll_{\Delta}& |g|^{d+\varepsilon}\sum_{N=T}^{\infty}\hat{N}^{3/2-\frac{d}{2}+2\varepsilon}=O_{\varepsilon,F}(|g|^{d+\varepsilon}\hat{T}^{3/2-\frac{d}{2}+2\varepsilon}),
\end{eqnarray*}
where we use $|r_2|=\frac{\hat{N}}{|r_1|}$ and that prime factors of $r_2$ are those dividing $\Delta g$ in the second inequality.  This implies that given $r_1$, we only have  $O_{\Delta}(|g|^{\varepsilon})$ possibilities for $r_2.$
The convergence of the last series follows  from $d\geq 4$.
 Using this, we obtain that
\[\sum_{1\leq |r|\leq \hat{T}}|r|^{-d}S_{g,r}(\vec{0})=\sum_{r}|r|^{-d}S_{g,r}(\vec{0})+O_{\varepsilon,F}(|g|^{d+\varepsilon}\hat{T}^{3/2-\frac{d}{2}+2\varepsilon}).\]
The conclusion follows.
\end{proof}
We now want to show that the infinite sum
\[\sum_{r}|r|^{-d}S_{g,r}(\vec{0})\]
can be entirely written in terms of number theoretic information.
\begin{lemma}\label{lem:densitysum} Suppose that all  conditions in Theorem~\ref{strong} are  satisfied. Then
\[\sum_{r}|gr|^{-d}S_{g,r}(\vec{0})=\prod_{\varpi}\sigma_{\varpi},\]
where $\varpi$ ranges over the monic irreducible polynomials in $\mathcal{O}$, and
\[\sigma_{\varpi}:=\lim_{k\rightarrow\infty}\frac{|\left\{\vec{x}\bmod \varpi^{k+\ord_\varpi(g)}:F(\vec{x})\equiv f\bmod\varpi^{k+\ord_\varpi(g)},\ \vec{x}\equiv\boldsymbol{\lambda}\bmod\varpi^{\ord_{\varpi}(g)}\right\}|}{|\varpi|^{(d-1)k}}.\]
Moreover,
\(
\prod_{\varpi}\sigma_{\varpi} \gg_{\varepsilon,F} |f|^{-\varepsilon}
\)
for every $\varepsilon>0.$
\end{lemma}
\begin{proof} Let $M\in \mathcal{O}.$ 
Write
\begin{eqnarray*}
&&\sum_{r|M}|r|^{-d}S_{g,r}(\vec{0})\\&=&\sum_{r|M}|r|^{-d}\sum_{\substack{a\bmod gr\\ (a,r)=1}}\sum_{\vec{b}\in\mathcal{O}^d/(gr)}\psi\left(\frac{a(2\boldsymbol{\lambda}^{\intercal}A\vec{b}-k+gF(\vec{b}))}{gr}\right)\\&=&\frac{1}{|M|^d}\sum_{r|M}\sum_{\substack{a\bmod gr\\ (a,r)=1}}\left|\frac{M}{r}\right|^d\sum_{\vec{b}\in\mathcal{O}^d/(gr)}\psi\left(\frac{a(2\boldsymbol{\lambda}^{\intercal}A\vec{b}-k+gF(\vec{b}))}{gr}\right)\\&=&\frac{1}{|M|^d}\sum_{r|M}\sum_{\substack{a\bmod gr\\ (a,r)=1}}\sum_{\vec{b}\in\mathcal{O}^d/(gM)}\psi\left(\frac{a(2\boldsymbol{\lambda}^{\intercal}A\vec{b}-k+gF(\vec{b}))}{gr}\right)\\&=& \frac{1}{|M|^d}\sum_{\vec{b}\in\mathcal{O}^d/(gM)}\sum_{r|M}\sum_{\substack{a\bmod gr\\ (a,r)=1}}\psi\left(\frac{a(2\boldsymbol{\lambda}^{\intercal}A\vec{b}-k+gF(\vec{b}))}{gr}\right).
\end{eqnarray*}
Since
\[\sum_{r|M}\sum_{\substack{a\bmod gr\\ (a,r)=1}}=\sum_{a\bmod gM},\]
\[\sum_{r|M}|r|^{-d}S_{g,r}(\vec{0})=\frac{1}{|M|^d}\sum_{a\bmod gM}\sum_{\vec{b}\in\mathcal{O}^d/(gM)}\psi\left(\frac{a(2\boldsymbol{\lambda}^{\intercal}A\vec{b}-k+gF(\vec{b}))}{gM}\right).\]
By Lemma~\ref{lem:orthogsum}, we obtain 
\begin{equation}\label{sh}
\sum_{r|M}|r|^{-d}S_{g,r}(\vec{0})=|g|\frac{\left|\{\vec{b}\in\mathcal{O}^d/(gM):2\boldsymbol{\lambda}^{\intercal}A\vec{b}-k+gF(\vec{b})\equiv 0\bmod gM\}\right|}{|M|^{d-1}}.
\end{equation}
 Define for each integer $N\geq 0$ the analogue of the factorial
\[(N)!:=\prod_{\substack{|f|\leq\hat{N}\\ f\ monic}}f.\] Let us write $(N)!=\varpi_1^{a_1}\hdots\varpi_{\ell}^{a_{\ell}}$. Then by the Chinese remainder theorem, 
\[2\boldsymbol{\lambda}^{\intercal}A\vec{b}-k+gF(\vec{b})\equiv 0 \bmod g(N)!\]
is equivalent to
\[F(g\vec{b}+\lambda)\equiv f \bmod \varpi_i^{a_i+2\ord_{\varpi_i}(g)}\]
for every $i=1,\hdots,\ell$. Letting $M=(N)!$ in \eqref{sh} and using  the above,
\begin{eqnarray*}&&\sum_{r|(N)!}|r|^{-d}S_{g,r}(\vec{0})\\&=&|g|^d\prod_{\varpi|(N)!}\frac{\left|\{\vec{b}\in\mathcal{O}^d/(\varpi^{\ord_{\varpi}((N)!)+\ord_{\varpi}(g)}):F(\varpi^{\ord_{\varpi}(g)}\vec{b}+\boldsymbol{\lambda})\equiv f\bmod \varpi^{\ord_{\varpi}((N)!)+2\ord_{\varpi}(g)}\}\right|}{|\varpi^{\ord_{\varpi}((N)!)+\ord_{\varpi}(g)}|^{d-1}}\\&=&
|g|^d\prod_{\varpi|(N)!}\frac{\left|\{\vec{x}\in\mathcal{O}^d/(\varpi^{\ord_{\varpi}((N)!)+2\ord_{\varpi}(g)}):F(\vec{x})\equiv f\bmod \varpi^{\ord_{\varpi}((N)!)+2\ord_{\varpi}(g)},\ \vec{x}\equiv \boldsymbol{\lambda}\bmod \varpi^{\ord_{\varpi}(g)}\}\right|}{|\varpi^{\ord_{\varpi}((N)!)+\ord_{\varpi}(g)}|^{d-1}}.
\end{eqnarray*}
 Letting $N\rightarrow\infty$ and using Lemma~\ref{lem:estimate} for absolute convergence gives us
\[\sum_{r}|gr|^{-d}S_{g,r}(\vec{0})=\prod_{\varpi}\sigma_{\varpi},\]
where $\sigma_{\varpi}$ are as in the statement of the lemma. This completes the proof of the first part of the lemma. 

Suppose that $\gcd(\varpi,g\Delta)=1.$ By \eqref{sh}, we have 
\[
\sigma_{\varpi}=|g|^{-d}\sum_{k\geq 0}|\varpi|^{-kd}S_{g,\varpi^k}(\vec{0})=1+\sum_{k\geq 1}|g|^{-d} |\varpi|^{-kd}S_{g,\varpi^k}(\vec{0}).
\]
 By Proposition~\ref{lem:S1bound},
\[
|g|^{-d}|\varpi|^{-kd}S_{g,\varpi^k}(\vec{0})\leq |\varpi|^{-k\frac{d-1}{2}}(k+1)|\gcd(\varpi^k,f)|^{1/2}
\]
Hence,
\[
\sigma_{\varpi}=\begin{cases}1+O(1/|\varpi|)\neq 0,  &\text{ if } \varpi|f,
\\ 
1+O(1/|\varpi|^{3/2})\neq 0,  &\text{ otherwise.} 
\end{cases}
\]
This implies that 
\[
|f|^{\varepsilon}\gg \prod_{\gcd(\varpi,g\Delta)=1}\sigma_{\varpi} \gg  |f|^{-\varepsilon}.
\]
Suppose that $\varpi|g$ and  $\gcd(\Delta,\varpi)=1.$ Since we also have $\boldsymbol{\lambda}\neq 0 \bmod \varpi$, $\nabla F(\boldsymbol{\lambda}) \neq 0 \bmod \varpi.$ Then, 
by Hensel's lemma
\[
\sigma_{\varpi}=1.
\]
Finally, suppose that  $\varpi|\Delta$. Since the local condition at $\varpi$ is satisfied, there exists $\vec{a}$ mod $\varpi$ such that
\[
F(\vec{a})=f \mod \varpi.
\] 
Furthermore, $\gcd(f,\Delta^{\infty})=O(1),$ which implies   $\nabla F(\boldsymbol{\vec{a}}) \neq 0 \bmod \varpi^{O(1)}.$ Then, 
by Hensel's lemma
\[
 1 \gg_{\Delta} \prod_{\varpi|\Delta}\sigma_{\varpi}\gg_{\Delta} 1.
\]
This concludes the proof of our lemma.
\end{proof}
\section{Proof of the main theorem}\label{finalproof}
In this section, we prove Theorem~\ref{mainthm}. Though we obtain a theorem for $d\geq 4$, it is only optimal when $d\geq 5$. As before, we assume the conditions of Theorem~\ref{mainthm}. Recall the following definitions:
\[Q=\lceil{\deg(f)/2} -deg(g)\rceil +\max_{i}(\deg(\eta_i))+\omega^{\prime}\]
and
\[R=\lceil{\deg(f)/2} -deg(g)-\alpha_0/2\rceil,\]
where
\[\alpha_0=2\omega+2\max{\deg(\eta_i)}+3|\min\deg(\eta_i)|+\omega',\]
and $\omega,\omega'$ are the parameters of the anisotropic cone $\Omega$ and $\max|\eta_i|=|A|$, $\min|\eta_i|=|A^{-1}|^{-1}$ are well-defined in terms of the $A$. Recall that $\eta$ is the parameter appearing in Proposition~\ref{c}, a constant depending only on $\Omega,F$. Furthermore, recall that $\kappa=\left|\vec{c}/g\right|$, and that $\vec{c}$ is said to be exceptional if $\kappa<\hQ/\hR$.\\
\\
We first give a bound on the contributions of the nonzero exceptional vectors to our counting function.
\begin{proposition}\label{prop:small}For any non-degenerate quadratic form $F$ over $\mathcal{O}$ in $d\geq 4$ variables, and for any $\varepsilon>0$, we have
\[\sum_{1\leq |r|\leq\hQ}\sum_{\vec{c}\neq 0}^{\exc}|gr|^{-d}|S_{g,r}(\vec{c})||I_{g,r}(\vec{c})|\ll_{\varepsilon,F,\Omega} \hQ^{\frac{d+3}{2}+\varepsilon}|g|^{\frac{d-3}{2}+\varepsilon}(1+|g|^{-\frac{d-5}{2}+\varepsilon}),\]
where $\sum^{\exc}$ denotes summation over exceptional vectors.
\end{proposition}
We prove this proposition by rewriting
\[\sum_{1\leq |r|\leq\hQ}\sum_{\vec{c}\neq 0}^{\exc}|gr|^{-d}S_{g,r}(\vec{c})I_{g,r}(\vec{c})=E_1+E_2,\]
where
\[E_1:=\sum_{\vec{c}\neq 0}^{\exc}\sum_{1\leq |r|\leq\frac{\hR |\vec{c}|}{\eta |g|}}|gr|^{-d}S_{g,r}(\vec{c})I_{g,r}(\vec{c})\]
and
\[E_2:=\sum_{\vec{c}\neq 0}^{\exc}\sum_{\frac{\hR |\vec{c}|}{\eta |g|}< |r|\leq\hQ}|gr|^{-d}S_{g,r}(\vec{c})I_{g,r}(\vec{c}),\]
and then showing that $E_1$ and $E_2$ satisfy the above bound. This division of the sum into two parts is suggested by Proposition~\ref{c}.
\begin{lemma}\label{lem:e1}
\[|E_1|\ll_{\varepsilon,F,\Omega} \hQ^{\frac{d+3}{2}+\varepsilon}|g|^{\frac{d-3}{2}+\varepsilon}(1+|g|^{-\frac{d-5}{2}+\varepsilon}).\]
\end{lemma}
\begin{proof}
By Proposition~\ref{c}, we know that for $|r|\leq \frac{\hR |\vec{c}|}{\eta |g|}$
\[|I_{g,r}(\vec{c})|\ll_{F,\Omega}\hQ^d\left(\frac{\hQ|\vec{c}|}{|gr|}\right)^{-\frac{d-1}{2}}.\]
Using this, we obtain
\begin{eqnarray*}|E_1|&\ll& \hQ^d\sum_{\vec{c}\neq 0}^{\exc}\sum_{1\leq |r|\leq\frac{\hR |\vec{c}|}{\eta |g|}}|gr|^{-d}|S_{g,r}(\vec{c})|\left(\frac{\hQ|\vec{c}|}{|gr|}\right)^{-\frac{d-1}{2}}\\&=&\hQ^{\frac{d+1}{2}}\sum_{\vec{c}\neq 0}^{\exc}\left(\frac{|\vec{c}|}{|g|}\right)^{-\frac{d-1}{2}}\sum_{1\leq |r|\leq\frac{\hR |\vec{c}|}{\eta |g|}}|g|^{-d}|r|^{-\frac{d+1}{2}}|S_{g,r}(\vec{c})|.
\end{eqnarray*}
By Proposition~\ref{prop:Supperbound},
\[\sum_{1\leq |r|\leq\frac{\hR |\vec{c}|}{\eta |g|}}|g|^{-d}|r|^{-\frac{d+1}{2}}|S_{g,r}(\vec{c})|\ll_{\varepsilon,F,\Omega}|g|^{\varepsilon}\left(\frac{\hR |\vec{c}|}{\eta |g|}\right)^{1+\varepsilon}\ll_{\varepsilon,F,\Omega} |g|^{\varepsilon}\left(\frac{\hQ |\vec{c}|}{|g|}\right)^{1+\varepsilon}.\]
Here, we are also using the fact that $R$ and $Q$ are of the same order up to a constant depending on the quadratic form and $\Omega$. Consequently,
\[|E_1|\ll \hQ^{\frac{d+1}{2}}\sum_{\vec{c}\neq 0}^{\exc}\left(\frac{|\vec{c}|}{|g|}\right)^{-\frac{d-1}{2}}|g|^{\varepsilon}\left(\frac{\hQ |\vec{c}|}{|g|}\right)^{1+\varepsilon}=\hQ^{\frac{d+3}{2}+\varepsilon}|g|^{\frac{d-3}{2}+\varepsilon}\sum_{\vec{c}\neq 0}^{\exc}|\vec{c}|^{-\frac{d-3}{2}}.\]
By Lemma~\ref{lem:cform}, we may assume without loss of generality that $\vec{c}$ are all congruent to $\alpha A\boldsymbol{\lambda}$ modulo $g$ for some varying polynomial $\alpha$. By assumption, at least one coordinate of $\boldsymbol{\lambda}$ is relatively prime to $g$, say the first one. Since $\vec{c}$ is congruent to $\alpha A\boldsymbol{\lambda}\bmod g$ for some $\alpha$ depending on $\vec{c}$ and $\gcd(\Delta,g)=1$, the first coordinate varies through all polynomials modulo $g$ as $\vec{c}$ and so as $\alpha$ varies. Furthermore, since $\vec{c}$ is exceptional and $\hQ/\hR=O_{F,\Omega}(1)$, $|\vec{c}|\leq O_{F,\Omega}(1)|g|$. Consequently,
\[\sum_{\vec{c}\neq 0}^{\exc}|\vec{c}|^{-\frac{d-3}{2}+\varepsilon}\ll_{\varepsilon,F,\Omega} \sum_{0\neq|\alpha|<|g|}|\alpha|^{-\frac{d-3}{2}+\varepsilon}\ll_{\varepsilon,F,\Omega} 1+|g|^{-\frac{d-5}{2}+\varepsilon},\]
from which we obtain
\[|E_1|\ll_{\varepsilon,F,\Omega} \hQ^{\frac{d+3}{2}+\varepsilon}|g|^{\frac{d-3}{2}+\varepsilon}(1+|g|^{-\frac{d-5}{2}+\varepsilon}).\]
\end{proof}
Similarly, we have the same bound on $E_2$.
\begin{lemma}\label{lem:e2}
\[|E_2|\ll_{\varepsilon,F,\Omega} \hQ^{\frac{d+3}{2}+\varepsilon}|g|^{\frac{d-3}{2}+\varepsilon}(1+|g|^{-\frac{d-5}{2}+\varepsilon}).\]
\end{lemma}
\begin{proof}
Note that by our choice of $Q$ and $R$, for every $\varepsilon>0$ we have $\frac{\hR |\vec{c}|}{\eta |g|}\geq \hQ^{1-\varepsilon}$ for large $Q$ which is guaranteed when $\deg f=4\deg g+O_{\varepsilon,F,\Omega}(1)$. This follows from $\kappa=\frac{|\vec{c}|}{|g|}<\frac{\hQ}{\hR}=O_{F,\Omega}(1)$. The sum $E_2$ is over $|r|>\frac{\hR |\vec{c}|}{\eta |g|}$. The argument at the beginning of the proof of Lemma~\ref{lem:small} show that
\[|I_{g,r}(\vec{c})|\ll_{F,\Omega}\hQ^{d+\varepsilon}.\]
Using this, we obtain
\begin{eqnarray*}|E_2| &\ll_{\varepsilon,F,\Omega}& \hQ^{d+\varepsilon}\sum_{\vec{c}\neq 0}^{\exc}\sum_{\frac{\hR |\vec{c}|}{\eta |g|}< |r|\leq \hQ}|gr|^{-d}|S_{g,r}(\vec{c})|\\ &=&\hQ^{d+\varepsilon}\sum_{\vec{c}\neq 0}^{\exc}\sum_{\frac{\hR |\vec{c}|}{\eta |g|}< |r|\leq \hQ}|r|^{-\frac{d-1}{2}}|g|^{-d}|r|^{-\frac{d+1}{2}}|S_{g,r}(\vec{c})|\\ &=& \hQ^{d+\varepsilon}\sum_{\vec{c}\neq 0}^{\exc}\sum_{k=1+\log_q\frac{\hR |\vec{c}|}{\eta |g|}}^Q\left(q^k\right)^{-\frac{d-1}{2}}\sum_{|r|=q^k}|g|^{-d}|r|^{-\frac{d+1}{2}}|S_{g,r}(\vec{c})|.
\end{eqnarray*}
By Proposition~\ref{prop:Supperbound}, for each $k$, 
\[\sum_{|r|=q^k}|g|^{-d}|r|^{-\frac{d+1}{2}}|S_{g,r}(\vec{c})|\ll_{\varepsilon,F} |g|^{\varepsilon}(q^k)^{1+\varepsilon}.\]
Hence
\begin{eqnarray*}|E_2|&\ll_{\varepsilon,F,\Omega}& \hQ^{d+\varepsilon}|g|^{\varepsilon}\sum_{\vec{c}\neq 0}^{\exc}\sum_{k=1+\log_q\frac{\hR |\vec{c}|}{\eta |g|}}^Q\left(q^k\right)^{-\frac{d-3}{2}+\varepsilon}\\&\ll_{\varepsilon,F,\Omega}& \hQ^{d+\varepsilon}|g|^{\varepsilon}\sum_{\vec{c}\neq 0}^{\exc}\left(\frac{\hQ |\vec{c}|}{|g|}\right)^{-\frac{d-3}{2}+\varepsilon}\\&=&\hQ^{\frac{d+3}{2}+2\varepsilon}|g|^{\frac{d-3}{2}+\varepsilon}\sum_{\vec{c}\neq 0}^{\exc}|\vec{c}|^{-\frac{d-3}{2}+\varepsilon}.
\end{eqnarray*}
As before,
\[\sum_{\vec{c}\neq 0}^{\exc}|\vec{c}|^{-\frac{d-3}{2}+\varepsilon}\ll_{\varepsilon,F,\Omega} \sum_{0\neq|\alpha|<|g|}|\alpha|^{-\frac{d-3}{2}+\varepsilon}\ll_{\varepsilon,F,\Omega} 1+|g|^{-\frac{d-5}{2}+\varepsilon},\]
from which the conclusion follows.
\end{proof}
We are now ready to prove our main theorem. Note that from remark~\ref{optimality} this is optimal for $d\geq 5$.
\begin{proof}[Proof of the main theorem~\ref{mainthm}] Recall from~\eqref{newequu}
\begin{equation*}N(w,\boldsymbol{\lambda})=\frac{1}{|g|\hQ^2}\sum_{\substack{
r\in \cO\\
|r|\leq \hat Q\\
\text{$r$ monic}
} }\sum_{\vec{c}\in\mathcal{O}^d}|gr|^{-d}S_{g,r}(\vec{c})I_{g,r}(\vec{c}).\end{equation*}
By Lemma~\ref{o}, Lemma~\ref{lem:small}, and Proposition~\ref{prop:small}, we have
\begin{eqnarray*}N(w,\boldsymbol{\lambda})&=&\frac{1}{|g|\hQ^2}\sum_{\substack{r\in \cO\\|r|\leq \hat Q^{1-\varepsilon}\\ \text{$r$ monic}}}|gr|^{-d}S_{g,r}(\vec{0})I_{g,r}(\vec{0})+O_{\varepsilon, F, \Omega}\left(\frac{\hQ^{\frac{d+3}{2}+2\varepsilon}|g|^{\frac{d-3}{2}+\varepsilon}(1+|g|^{-\frac{d-5}{2}+\varepsilon})}{|g|\hQ^2}\right)\\&=&\frac{1}{|g|\hQ^2}\sum_{\substack{r\in \cO\\|r|\leq \hat Q^{1-\varepsilon}\\ \text{$r$ monic}}}|gr|^{-d}S_{g,r}(\vec{0})I_{g,r}(\vec{0})+O_{\varepsilon, F, \Omega}\left(\hQ^{\frac{d-1}{2}+2\varepsilon}|g|^{\frac{d-5}{2}+\varepsilon}(1+|g|^{-\frac{d-5}{2}+\varepsilon})\right).\end{eqnarray*}
By Proposition~\ref{lem:Izerobound}, $I_{g,r}(\vec{0})=\sigma_{\infty}\hQ^d$ for some constant $\sigma_{\infty}>0$ and $\hQ$ sufficiently large depending on $\varepsilon,F,\Omega$. Hence for such $\hQ$,
\[\frac{1}{|g|\hQ^2}\sum_{\substack{r\in \cO\\|r|\leq \hat Q^{1-\varepsilon}\\ \text{$r$ monic}}}|gr|^{-d}S_{g,r}(\vec{0})I_{g,r}(\vec{0})=\frac{\sigma_{\infty}\hQ^{d-2}}{|g|}\sum_{\substack{r\in \cO\\|r|\leq \hat Q^{1-\varepsilon}\\ \text{$r$ monic}}}|gr|^{-d}S_{g,r}(\vec{0}).\]
On the other hand, by Lemma~\ref{lem:estimate} and Lemma~\ref{lem:densitysum}, 
\[\sum_{\substack{r\in \cO\\|r|\leq \hat Q^{1-\varepsilon}\\ \text{$r$ monic}}}|gr|^{-d}S_{g,r}(\vec{0})=\prod_{\varpi}\sigma_{\varpi}+O_{\varepsilon,F}\left(|g|^{\varepsilon}\hQ^{-\frac{d-3}{2}+2\varepsilon}\right).\]
As a result, we finally obtain
\begin{eqnarray*}N(w,\boldsymbol{\lambda})&=&\frac{\sigma_{\infty}\hQ^{d-2}}{|g|}\left(\prod_{\varpi}\sigma_{\varpi}+O_{\varepsilon,F}\left(|g|^{\varepsilon}\hQ^{-\frac{d-3}{2}+2\varepsilon}\right)\right)+O_{\varepsilon,F,\Omega}\left(\hQ^{\frac{d-1}{2}+2\varepsilon}|g|^{\frac{d-5}{2}+\varepsilon}(1+|g|^{-\frac{d-5}{2}+\varepsilon})\right)\\&=&\frac{\sigma_{\infty}\hQ^{d-2}}{|g|}\prod_{\varpi}\sigma_{\varpi}+O_{\varepsilon,F,\Omega}\left(\hQ^{\frac{d-1}{2}+2\varepsilon}|g|^{\frac{d-5}{2}+\varepsilon}(1+|g|^{-\frac{d-5}{2}+\varepsilon})\right)\\&=& \frac{\sigma_{\infty}\hQ^{d-2}}{|g|}\prod_{\varpi}\sigma_{\varpi}\left(1+O_{\varepsilon,F,\Omega}\left(\frac{|f|^{\varepsilon}\hQ^{\frac{d-1}{2}+2\varepsilon}|g|^{\frac{d-3}{2}+\varepsilon}(1+|g|^{-\frac{d-5}{2}+\varepsilon})}{\hQ^{d-2}}\right)\right)\\&=&\frac{\sigma_{\infty}\hQ^{d-2}}{|g|}\prod_{\varpi}\sigma_{\varpi}\left(1+O_{\varepsilon,F,\Omega}\left(\frac{|g|^{d-3-\varepsilon}(1+|g|^{-\frac{d-5}{2}+\varepsilon})}{|f|^{\frac{d-3}{4}-2\varepsilon}}\right)\right)\\&=& \frac{\sigma_{\infty}\hQ^{d-2}}{|g|}\prod_{\varpi}\sigma_{\varpi}\left(1+O_{\varepsilon,F,\Omega}\left(|g|^{7\varepsilon}(1+|g|^{-\frac{d-5}{2}+\varepsilon})\left(\frac{|g|^{4}}{|f|}\right)^{\frac{d-3}{4}-2\varepsilon}\right)\right).\end{eqnarray*}
Note that in the third equality, we have also used Lemma~\ref{lem:densitysum} ensuring that the product of the local densities is $\gg |f|^{-\varepsilon}$. Note that $\varepsilon>0$ is arbitrary. From the above, it follows that the conclusion of Theorem~\ref{mainthm} holds for $d\geq 5$ and $|f|\gg_{\varepsilon,F,\Omega} |g|^{4+\varepsilon}$ with a potential modification of $\varepsilon$ if necessary. If $d=4$, we can take $|f|\gg_{\varepsilon,F,\Omega} |g|^{6+\varepsilon}$.
\end{proof}
\begin{proof}[Proof of Theorem~\ref{infmainthm}]
Recall the statement and notation of Theorem~\ref{infmainthm}. We prove this theorem by reducing it to Theorem~\ref{mainthm}. Consider the equation $F(\vec{x})=\vec{x}^{\intercal}A\vec{x}=f$. We may assume without loss of generality that $a_{11}(t)$ has the maximal degree $d_0=\max\deg a_{ij}(t)$ satisfying $\gcd(a_{11},t)=1$, and that $a_{11}$ has the same quadratic residue as $f$. Write $\tilde{A}(t):=t^{d_0}[a_{ij}(1/t)]_{ij}$. Suppose that $\Omega$ has parameters $\omega$ and $\omega'$. Let $2c'\geq\omega'$ be an even integer. Writing $f=\alpha_fu^2$, where $\alpha_f$ of degree at most $1$ is the representative of $f$ in $\Ki^{*}/\Ki^{*2}$, changing $t$ to $1/t$ in $F(\vec{x})=f$, and multiplying by $t^{2c'+\deg f+d_0}$, we obtain
\[\left(t^{c'+\deg u}\vec{x}(1/t)\right)^{\intercal}\left(t^{\deg\alpha_f}\tilde{A}(t)\right)\left(t^{c'+\deg u}\vec{x}(1/t)\right)=t^{2c'+d_0+\deg f}f(1/t).\]
Consider the system 
\begin{equation}\label{system}\begin{cases}\vec{y}(t)^{\intercal}\left(t^{\deg\alpha_f}\tilde{A}(t)\right)\vec{y}(t)=t^{2c'+d_0+\deg f}f(1/t),\\ \vec{y}(t)\equiv t^{c'+\deg u}\boldsymbol{\lambda}(1/t)\bmod t^N.\end{cases}\end{equation}
Since $\boldsymbol{\lambda}\in\Omega$, condition~\eqref{prop3} of anisotropic cones implies that
\[|f|=|F(\boldsymbol{\lambda})|\geq \frac{|\boldsymbol{\lambda}|^2}{\hat{\omega'}},\]
from which we obtain $c'+\deg u\geq\deg\boldsymbol{\lambda}$. From this, $t^{c'+\deg u}\boldsymbol{\lambda}(1/t)$ is a power series in $t$, and so is a polynomial modulo $t^N$.\\
\\
As in Example ~\ref{exampleaniso}, we have the anisotropic cone 
\[\tilde{\Omega}=\left\{\vec{y}=(y_1,\hdots,y_d)\in\Ki^d:\forall i>1, \deg y_i<\deg y_1\right\}\]
for the quadratic form $\vec{y}(t)^{\intercal}\left(t^{\deg\alpha_f}\tilde{A}(t)\right)\vec{y}(t)$. Since 
\[t^{2c'+d_0+\deg f}f(1/t)/(t^{\deg\alpha_f}t^{d_0}a_{11}(1/t))=t^{2c'+2\deg u}f(1/t)/a_{11}(1/t)\]
is a square, the equation 
\[\vec{y}(t)^{\intercal}\left(t^{\deg\alpha_f}\tilde{A}(t)\right)\vec{y}(t)=t^{2c'+d_0+\deg f}f(1/t)\] has a solution in $\tilde{\Omega}$. Applying Theorem~\ref{mainthm} to $\tilde{\Omega}$ and the system \eqref{system}, we obtain that there is a solution $\tilde{\vec{x}}(t)\in\mathcal{O}^d\cap\tilde{\Omega}$ if $\deg f\geq (4+\varepsilon)N+O_{\varepsilon,F,\tilde{\Omega}}(1)$ when $d\geq 5$, and if $\deg f\geq (6+\varepsilon)N+O_{\varepsilon,F,\tilde{\Omega}}(1)$ when $d=4$. Since
\begin{equation}\label{degreecounting}\tilde{\vec{x}}(t)^{\intercal}\left(t^{\deg\alpha_f}\tilde{A}(t)\right)\tilde{\vec{x}}(t)=t^{2c'+d_0+\deg f}f(1/t),\end{equation}
replacing $t$ with $1/t$ we have
\[\left(t^{c'+\deg u}\tilde{\vec{x}}(1/t)\right)^{\intercal}A(t)\left(t^{c'+\deg u}\tilde{\vec{x}}(1/t)\right)=f(t).\]
We first show that $\vec{x}(t):=t^{c'+\deg u}\tilde{\vec{x}}(1/t)\in\mathcal{O}^d$. Since $\tilde{\vec{x}}(t)\in\tilde{\Omega}$, we have by computing degrees (in $t$) in equation~\eqref{degreecounting} that
\[\deg\tilde{A}(t)+\deg\tilde{\vec{x}}(t)\leq c'+\deg u+d_0.\]
Since $\gcd(a_{11},t)=1$, we have $\deg\tilde{A}(t)=\deg A=d_0$, from which it follows that $\deg\tilde{\vec{x}}(t)\leq c'+\deg u$. From this, we obtain $\vec{x}(t)\in\mathcal{O}^d$.\\
\\
Note that since $\tilde{\vec{x}}(t)\equiv t^{c'+\deg u}\boldsymbol{\lambda}(1/t)\bmod t^N$, 
\[\left|\tilde{\vec{x}}(1/t)-t^{-c'-\deg u}\boldsymbol{\lambda}(t)\right|\leq q^{-N},\]
that is,
\[\left|\vec{x}-\boldsymbol{\lambda}\right|\leq q^{-N+c'+\deg u}\leq |f|^{1/2}q^{-N+c'}.\]
The conclusion now follows as $N\geq 0$ is an arbitrary constant and $c'$ depends only on $\Omega$.

\end{proof}
\textit{Acknowledgment.} N.T. Sardari's work is supported partially by the National Science Foundation under Grant No. DMS-2015305 and is grateful to Max Planck Institute for Mathematics in Bonn and the Institute
For Advanced Study for their hospitality and financial support. During the writing of this project, M. Zargar was supported by SFB1085: Higher invariants at the University of Regensburg, the Max Planck Institute for Mathematics in Bonn, and the University of Southern California.
\bibliographystyle{alpha}
\bibliography{part1}
\end{document}